\documentclass[a4paper,twoside]{article}  

\usepackage{amssymb,amsmath,amsthm,dsfont,amsfonts,color,latexsym,CJK}

\usepackage{hyperref}

\usepackage{mathrsfs} 

\definecolor{blue}{rgb}{0.00,0.00,1.00}
\definecolor{red}{rgb}{1.00,0.00,0.00}

\renewcommand{\baselinestretch}{1.2}
\def\bq{\begin{equation}}
\def\eq{\end{equation}}
\def\ba{\begin{array}{ccc}}
\def\bal{\begin{array}{lll}}
\def\ea{\end{array}}

\def\({\left(}\def\){\right)}
\def\[{\left[}\def\]{\right]}
\def\<{\langle}\def\>{\rangle}

    \def \C   {\mathbb{C}}

    \def \R   {\mathbb{R}}

    \def\z    {\mathbf{z}}
    \def\e    {\mathbf{e}}

    \def\S    {\mathbb{S}}
    \def\eps  {\epsilon}
    \def\intr {\int_{\R^3}}
    
    \def\ints {\int_{\S^2}}
    \def\intt {\int^t_0}

\def\BB{\mathbb{B}}
\def\AA{\mathbb{A}}

    \def \pt   {\partial}
    
    \def \Dt   {\frac{\rm d}{{\rm d}t}}
    
    \def \dt    {\partial_t}

    \def \dx    {\partial_x}

    \def \dxa   {\partial^{\alpha}_x}

    \def \dvb   {\partial^{\beta}_v}

    \def \divx  {{\rm div}_x}

    \def\Tdx   {\nabla_x}
    \def\Tdv   {\nabla_v}

       \def\bq{\begin{equation}}
       \def\eq{\end{equation}}
       \def\be{\begin{equation}}
       \def\ee{\end{equation}}
       \def\bma#1\ema{{\allowdisplaybreaks\begin{align}#1\end{align}}}
       \def\bmas#1\emas{{\allowdisplaybreaks\begin{align*}#1\end{align*}}}
       \def\bln#1\eln{{\allowdisplaybreaks\begin{aligned}#1\end{aligned}}}
       \def\nnm{\notag}
       \def\bgr#1\egr{\allowdisplaybreaks\begin{gather}#1\end{gather}}
       \def\bgrs#1\egrs{\allowdisplaybreaks\begin{gather*}#1\end{gather*}}

       \theoremstyle{plain}
       \newtheorem{lem}{\bf Lemma}[section]
       \newtheorem{thm}[lem]{\textbf{Theorem}}

       \newtheorem{remark}[lem]{\bf Remark}

\begin{document}

\title{ Diffusion Limit with Optimal Convergence Rate of Classical Solutions to the  Vlasov-Maxwell-Boltzmann System }
\author{  Tong Yang$^1$, Mingying Zhong$^2$\\[2mm]
 \emph{\small\it $^1$Department of Applied Mathematics, The Hong Kong Polytechnic University,  Hong
    Kong.}\\
    {\small\it E-mail: t.yang@polyu.edu.hk} \\
    {\small\it  $^2$School of  Mathematics and Information Sciences,
    Guangxi University,  China.}\\
    {\small\it E-mail:\ zhongmingying@gxu.edu.cn}\\[5mm]
    }
\date{ }

\pagestyle{myheadings}
\markboth{Vlasov-Maxwell-Boltzmann system }%
{ T. Yang, M.-Y. Zhong }

 \maketitle

 \thispagestyle{empty}

\begin{abstract}\noindent
We study the diffusion limit of the classical solution to the Vlasov-Maxwell-Boltzmann (VMB)  system with initial data near a global Maxwellian. By introducing a new decomposition of the solution to identify the essential components for generating the initial layer, we prove the convergence and establish the opitmal convergence rate of the classical solution to the VMB system to the solution of the  Navier-Stokes-Maxwell  system based on the spectral analysis.


\end{abstract}

%

\tableofcontents

\section{Introduction}
The Vlasov-Maxwell-Boltzmann (VMB) system is a fundamental model in plasma physics  describing the time evolution of dilute charged particles,
such as electrons and ions,  under the influence of the self-induced Lorentz forces governed by Maxwell equations, cf. \cite{ChapmanCowling}
for derivation and the physical background.
The rescaled two-species VMB system  in the  incompressible diffusive regime takes the form \cite{Arsenio}
 \be \label{VMB1a}
 \left\{\bln
&\dt F^{+}_{\eps}+\frac{1}{\eps}v\cdot\Tdx F^{+}_{\eps}+\frac{1}{\eps}(\alpha E_{\eps}+\beta v\times B_{\eps})\cdot\Tdv F^{+}_{\eps}=\frac{1}{\eps^2}Q(F^{+}_{\eps},F^{+}_{\eps})+\frac{\delta^2}{\eps^2}Q(F^{+}_{\eps},F^{-}_{\eps}),\\
&\dt F^{-}_{\eps}+\frac{1}{\eps}v\cdot\Tdx F^{-}_{\eps}-\frac{1}{\eps}(\alpha E_{\eps}+\beta v\times B_{\eps})\cdot\Tdv F^{-}_{\eps}=\frac{1}{\eps^2}Q(F^{-}_{\eps},F^{-}_{\eps})+\frac{\delta^2}{\eps^2}Q(F^{-}_{\eps},F^{+}_{\eps}), \\
&\gamma\dt E_{\eps}-\Tdx\times B_{\eps}=- \frac{\beta}{\eps^2}\intr (F^{+}_{\eps}-F^{-}_{\eps})vdv, \\
&\gamma\dt B_{\eps}+\Tdx\times E_{\eps}=0,\\
&\Tdx\cdot E_{\eps}= \frac{\alpha}{\eps^2}\intr (F^{+}_{\eps}-F^{-}_{\eps})dv,\quad \Tdx\cdot B_{\eps}=0,
 \eln\right.
\ee
where $\eps>0$ is a small parameter with $\epsilon^2$ proportional to  the mean free path, $\delta>0$ measures the strength of interactions, and $\alpha,\beta,\gamma$ have the following
physical meanings:
 \begin{itemize}
\item  $\alpha$ measures the electric repulsion according to Gauss  law;
\item $\beta$ measures the magnetic induction according to Amp\`{e}re  law;
\item $\gamma$ is the ratio of the bulk velocity to the speed of light.
 \end{itemize}
Notice that the parameters $\alpha,\beta,\gamma$ satisfy the relation
$$\beta=\frac{\alpha\gamma}{\eps}.$$
In addition, in \eqref{VMB1a}, $F^{\pm}_{\eps}=F^{\pm}_{\eps}(t,x,v)$ are the
density distribution functions of charged particles at $(t,x,v)\in \R_+\times \R^3_x\times\R^3_v$, and $E_{\eps}(t,x)$, $B_{\eps}(t,x)$ denote the electro and magnetic fields respectively.
Since we study the diffusive limit when $\epsilon $ tends to zero, we assume $\eps\in(0,1)$ in the following analysis. As usual,  $Q(F,G)$  is the  Boltzmann collision operator for hard sphere model given by
 \bq
 Q(F,G)=\intr\ints
 |(v-v_*)\cdot\omega|(F(v')G(v'_*)-F(v)G(v_*))dv_*d\omega,\label{binay_collision}
 \eq
where  $v,v_*$ are the velocities of gas particles before collision and $v',v'_*$ are the velocities after
collision:
$$
 v'=v-[(v-v_*)\cdot\omega]\omega,\quad
 v'_*=v_*+[(v-v_*)\cdot\omega]\omega,\quad \omega\in\S^2.
$$
That is,  the collisions are elastic so that the following conservation  of momentum and energy hold
\be
v+v_* =v'+v'_*,\quad
|v|^2+|v_*|^2 =|v'|^2+|v'_*|^2. \label{conser}
\ee

In this paper, we consider a typical case when $\alpha = \eps, \beta = 1, \gamma= 1,\delta = 1$, i.e.,
 \be \label{VMB1}
 \left\{\bln
&\dt F^{+}_{\eps}+\frac{1}{\eps}v\cdot\Tdx F^{+}_{\eps}+\frac{1}{\eps}(\eps E_{\eps}+ v\times B_{\eps})\cdot\Tdv F^{+}_{\eps}=\frac{1}{\eps^2}[Q(F^{+}_{\eps},F^{+}_{\eps})+Q(F^{+}_{\eps},F^{-}_{\eps})],\\
&\dt F^{-}_{\eps}+\frac{1}{\eps}v\cdot\Tdx F^{-}_{\eps}-\frac{1}{\eps}(\eps E_{\eps}+ v\times B_{\eps})\cdot\Tdv F^{-}_{\eps}=\frac{1}{\eps^2}[Q(F^{-}_{\eps},F^{+}_{\eps})+Q(F^{-}_{\eps},F^{-}_{\eps})], \\
& \dt E_{\eps}-\Tdx\times B_{\eps}=- \frac{1}{\eps^2}\intr (F^{+}_{\eps}-F^{-}_{\eps})vdv, \\
& \dt B_{\eps}+\Tdx\times E_{\eps}=0,\\
&\Tdx\cdot E_{\eps}= \frac{1}{\eps}\intr (F^{+}_{\eps}-F^{-}_{\eps})dv,\quad \Tdx\cdot B_{\eps}=0.
 \eln\right.
\ee

The Vlasov-Maxwell-Boltzmann system has been intensively studied and many important progress has been made in \cite{Duan4,Duan5,Guo4,Jang,Strain}.
For instance, the global existence of unique strong solution with initial data near the normalized global Maxwellian was obtained in spatial period domain~\cite{Guo4} and  in three dimensional  space  \cite{Strain} for hard sphere collision, and then in \cite{Duan1,Duan2} for general collision kernels with or without angular cut-off assumption.
For  the long time behaviors, it was shown in \cite{Duan5} that the total energy of the linearized one-species VMB system decays at the rate $(1+t)^{-\frac38}$, and in \cite{Duan4} that the total energy of nonlinear two-species VMB system decays at the rate $(1+t)^{-\frac34}$.
 The spectrum structure and the optimal decay rate of the global solution  to the VMB systems  for both one-spices and two-spices were investigated in \cite{Li3}.  The diffusive limit for two-species VMB system  was shown in \cite{Arsenio,Jang,Jiang2}.

On the other hand, the diffusion limit to the Boltzmann equation is a classical problem with pioneer work by Bardos-Golse-Levermore in \cite{FL-1}, and significant  progress on the limit of renormalized solutions to Leray solution to Navier-Stokes system in \cite{Golse-SR}.  One effective approach to study the fluid dynamic in the perturbative framework is based on the spectral analysis. For example,  Ellis-Pinsky \cite{Ellis} first studied the linear compressible Euler limit of the linear Boltzmann equation and showed the convergence rate outside the initial layer. The initial layer in the fluid limit arises from the incompatibility of the initial data, in particular due to the high oscillation of the eigen-modes in the system.  For the Boltzmann equation,
Bardos-Ukai \cite{FL-3} firstly studied  the incompressible Navier-Stokes limit with the estimation on the initial layer for the linear  Boltzmann equation. The analysis in \cite{FL-3} uses an estimate on the semigroup with highly oscillating eigen-modes in \cite{Ukai2} which is about the incompressible limit of  the compressible Euler equation. Note that the estimation on the initial layer is not optimal in these papers.
In contrast to the extensive study   on Boltzmann equation~\cite{FL-1,FL-2,FL-3,Guo5,FL-4,FL-5}, the VPB system \cite{Guo2,Li1,Wang1} and the VMB system \cite{Arsenio,Jang,Jiang2},  the convergence rate of the classical solution to the  VMB system \eqref{VMB1}  towards its fluid dynamical limits and the estimation of the initial layer have not been given despite of its importance.

Let $ F_{\eps}=F^+_{\eps}+F^-_{\eps}$ and $ G_{\eps}=F^+_{\eps}-F^-_{\eps}$. Then the system \eqref{VMB1} become to
\be
 \left\{\bln     \label{VMB2}
&\dt F_{\eps}+\frac{1}{\eps}v\cdot\Tdx F_{\eps}+\frac{1}{\eps}(\eps E_{\eps}+v\times B_{\eps})\cdot\Tdv G_{\eps}
=\frac{1}{\eps^2}Q(F_{\eps},F_{\eps}),\\
&\dt G_{\eps}+\frac{1}{\eps}v\cdot\Tdx G_{\eps}+\frac{1}{\eps}(\eps E_{\eps}+v\times B_{\eps})\cdot\Tdv F_{\eps}
=\frac{1}{\eps^2}Q(G_{\eps},F_{\eps}),\\
& \dt E_{\eps}-\Tdx\times B_{\eps}=- \frac{1}{\eps^2}\intr G_{\eps}vdv,\\
& \dt B_{\eps}+\Tdx\times E_{\eps}=0,\\
& \Tdx\cdot E_{\eps} =\frac{1}{\eps}\intr G_{\eps}dv,\quad \Tdx\cdot B_{\eps}=0.
\eln\right.
 \ee

In this paper, we study the diffusion limit
of the strong solution to the rescaled VMB system~\eqref{VMB2}  with initial data near the equilibrium $(F_*,G_*,E_*,B_*)=(M(v),0,0,0)$, where $M(v)$ is the  normalized Maxwellian given by
$$
 M=M(v)=\frac1{(2\pi)^{3/2}}e^{-\frac{|v|^2}2},\quad v\in\R^3.
$$
Hence, we define the perturbation of $(F_{\eps},G_{\eps},E_{\eps},B_{\eps})$ as
$$
 F_{\eps}=M+\eps\sqrt{M}f_{\eps},\quad G_{\eps}= \eps\sqrt{M}g_{\eps} .
$$
Then Cauchy problem of the VMB system~\eqref{VMB2} for $ (f_{\eps},g_{\eps},E_{\eps},B_{\eps})$ can be rewritten as
 \bma
 &\dt f_{\eps}+\frac{1}{\eps}v\cdot\Tdx f_{\eps} -\frac{1}{\eps^2}Lf_{\eps}=H^1_{\eps},\label{VMB4}\\
 &\dt g_{\eps}+\frac{1}{\eps}v\cdot\Tdx g_{\eps}-\frac{1}{\eps^2}L_1g_{\eps}- \frac{1}{\eps}v\sqrt{M}\cdot E_{\eps}=H^2_{\eps},\label{VMB4a}\\
&\dt E_{\eps}-\Tdx\times B_{\eps}=- \frac{1}{\eps}\intr g_{\eps}v\sqrt{M}dv,  \label{VMB4b}\\
&\dt B_{\eps}+\Tdx\times E_{\eps}=0,\label{VMB4c}\\
&\Tdx\cdot E_{\eps}= \intr g_{\eps}\sqrt{M}dv,\quad \Tdx\cdot B_{\eps}=0,\label{VMB4d}
\ema
where the nonlinear terms $H^1_{\eps},H^2_{\eps}$ are defined by
\bma
&H^1_{\eps}=\frac12 (v\cdot E_{\eps})g_{\eps}-\(E_{\eps}+\frac{1}{\eps} v\times B_{\eps}\)\cdot\Tdv g_{\eps}+\frac{1}{\eps} \Gamma(f_{\eps},f_{\eps}),\\
&H^2_{\eps}=\frac12 (v\cdot E_{\eps})f_{\eps}-\(E_{\eps}+\frac{1}{\eps} v\times B_{\eps}\)\cdot\Tdv f_{\eps}+\frac{1}{\eps} \Gamma(g_{\eps},f_{\eps}). \label{VMB5}
\ema
The initial condition is given by
\be (f_{\eps},g_{\eps})(0,x,v)=(f_{0},g_0)(x,v)  ,\quad  (E_{\eps},B_{\eps})(0,x)=(E_0, B_0)(x) ,  \label{VMB2i}\ee
which is independent of $\eps$. On the other hand, the initial data should satisfy the compatibility conditions
\be  \Tdx\cdot E_0(x)=\intr g_{0}\sqrt{M}dv,\quad  \Tdx\cdot B_0(x)=0. \label{com}\ee
In \eqref{VMB4}--\eqref{VMB5}, the linear operators $L,L_1$ and the nonlinear operator $\Gamma(f,g)$ are defined by
\be \label{gamma}
\left\{\bln
Lf&=\frac1{\sqrt M}[Q(M,\sqrt{M}f)+Q(\sqrt{M}f,M)],\\
L_1f&=\frac1{\sqrt M}Q(\sqrt{M}f,M),\\
\Gamma(f,g)&=\frac1{\sqrt M}Q(\sqrt{M}f,\sqrt{M}g).
\eln\right.
\ee

As usual, the linearized  operators $L$ and $L_1$ can be written as (cf. \cite{Cercignani,Yu})
\be\label{L_1}
\left\{\bln
(Lf)(v)&=(Kf)(v)-\nu(v) f(v),\quad (L_1f)(v)=(K_1f)(v)-\nu(v) f(v),\\
(Kf)(v)&=\intr k(v,v_*)f(v_*)dv_*,\quad (K_1f)(v)=\intr k_1(v,v_*)f(v_*)dv_*,\\
\nu(v)&=\sqrt{2\pi}\bigg(e^{-\frac{|v|^2}2}+\(|v|+\frac1{|v|}\)\int^{|v|}_0e^{-\frac{|u|^2}2}du\bigg),\\
k(v,v_*)&=\frac2{\sqrt{2\pi}|v-v_*|}e^{-\frac{(|v|^2-|v_*|^2)^2}{8|v-v_*|^2}-\frac{|v-v_*|^2}8}-\frac{|v-v_*|}{2\sqrt{2\pi}}e^{-\frac{|v|^2+|v_*|^2}4},\\
k_1(v,v_*)&=\frac2{\sqrt{2\pi}|v-v_*|}e^{-\frac{(|v|^2-|v_*|^2)^2}{8|v-v_*|^2}-\frac{|v-v_*|^2}8},
\eln\right.
\ee
where $\nu(v)$ is  the collision frequency,   $K$ and $K_1$ are self-adjoint compact operators
on $L^2(\R^3_v)$ with  real symmetric integral kernels $k(v,v_*)$ and $k_1(v,v_*)$.
In addition, $\nu(v)$ satisfies
\be
\nu_0(1+|v|) \leq\nu(v)\leq \nu_1(1+|v|). \label{nu}
\ee

The nullspace of the operator $L$  denoted by $N_0$  is a subspace
spanned by the orthonormal basis $\{\chi_j,\ j=0,1,\cdots,4\}$  given by
\bq \chi_0=\sqrt{M},\quad \chi_j=v_j\sqrt{M} \,\, (j=1,2,3), \quad
\chi_4=\frac{(|v|^2-3)\sqrt{M}}{\sqrt{6}},\label{basis}\eq
and the null space of the operator $L_1$  denoted by $N_1$  is
spanned only by $\sqrt{M}$.

Let   $L^2(\R^3)$ be a Hilbert space of complex-value functions $f(v)$
on $\R^3$ with the inner product and the norm
$$
(f,g)=\intr f(v)\overline{g(v)}dv,\quad \|f\|=\(\intr |f(v)|^2dv\)^{1/2}.
$$
And let $P_{ 0},P_{ d}$ be the projection operators from $L^2(\R^3_v)$ to the subspace $N_0, N_1$ with
\bma
 &P_{ 0}f=\sum_{j=0}^4(f,\chi_j)\chi_j,\quad P_1=I-P_{ 0}, \label{P10}
 \\
 &P_{ d}f=(f,\sqrt M)\sqrt M,   \quad P_r=I-P_{ d}. \label{Pdr}
 \ema
For any $U=(g,X,Y)\in L^2\times \R^3\times \R^3$, we denote
\be P_2U=(P_dg, X,Y), \quad P_3U=(I-P_2)U=(P_rg, 0,0). \label{P23} \ee

By Boltzmann  H-theorem, the linearized collision operators $L$ and $L_1$ are non-positive. Precisely,  there is a constant $\mu>0$ such that \bma
 (Lf,f)&\leq -\mu \| P_1f\|^2, \quad  \ f\in D(L),\\
 (L_1f,f)&\leq -\mu \|P_rf\|^2, \quad  \ f\in D(L_1),\label{L_4}
 \ema
where $D(L)$ and $D(L_1)$ are the domains of $L$ and $L_1$ given by
$$ D(L)=D(L_1)=\left\{f\in L^2(\R^3)\,|\,\nu(v)f\in L^2(\R^3)\right\}.$$
Without the loss of generality, we assume  $\nu(0)\ge \nu_0\ge \mu>0$.

This paper  aims to study  the optimal convergence rate of the classical solution $(f_{\eps},g_{\eps}, E_{\eps},B_{\eps})$ of the system \eqref{VMB4}--\eqref{VMB2i} to  $(u_1 , u_2,E,B)$, where $u_1=n \chi_0+m\cdot v\chi_0+q\chi_4$ and $u_2=\rho  \chi_0 $ with $(n,m,q,\rho ,E,B)(t,x)$ being the solution of the following bipolar incompressible Navier-Stokes-Maxwell-Fourier (NSMF) system:
\be \label{NSM_2}
\left\{\bal
\Tdx\cdot m=0,\quad n+ \sqrt{\frac23}q =0, \\
\dt m-\kappa_0\Delta_x m +\Tdx p=\rho E+j\times B-\Tdx\cdot (m\otimes m), \\
 \dt q -\kappa_1\Delta_x q= -\Tdx\cdot (q m), \\
 \dt E-\Tdx\times B=-j,\quad \Tdx\cdot E=\rho, \\
 \dt B+\Tdx\times E=0, \quad \Tdx\cdot B=0,\\
j=-\eta(\Tdx \rho -E)+(\rho m -\eta m\times B).
\ea\right.
\ee
Here,  $p$ is the pressure and the initial data $(n,m,q,\rho ,E,B)(0)$ satisfies
\be\label{NSP_5i}
\left\{\bal
m(0) = (f_{0},v\chi_0)-\Delta^{-1}_x\Tdx\divx (f_{0},v\chi_0),\\
n(0)=-\sqrt{\frac23} q(0)=\sqrt{\frac25}(f_0,\sqrt{\frac25}\chi_0-\sqrt{\frac35}\chi_4), \\
\rho(0)=\divx E_0, \quad
E(0)=E_0, \quad B(0)=B_0.
\ea\right.
\ee
Correspondingly,   the viscosity coefficients $\kappa_0$, $\kappa_1,\, \eta>0$ are defined by
\be\label{coe}
\left\{\bln
\kappa_0&=-(L^{-1}P_1(v_1\chi_2),v_1\chi_2),\quad \kappa_1=-\frac35(L^{-1}P_1(v_1\chi_4),v_1\chi_4), \\
\eta&=-(L^{-1}_1 (v_1\chi_0),v_1\chi_0).
\eln\right.
\ee

There have been extensive studies on the existence and solution behavior  about the incompressible Navier-Stokes-Maxwell system. Precisely,
the existence, uniqueness and an exponential growth estimate of global strong solutions were proved in \cite{Masmoudi} for a slightly different system.
  And the existence of global small mild solution was given in \cite{Ibrahim}.
In addition,  the authors in \cite{Luo,Jiang1,ZHANG}  studied the global  classical solutions to the two-fluid incompressible Navier-Stokes-Fourier-Maxwell system with Ohm's law with small initial data.  Regularity results for  the Cauchy problem of the incompressible Navier-Stokes-Maxwell system with Ohm's law in two and three space  dimensions were given in \cite{Wen}.

On the other hand, for the compressible Navier-Stokes-Maxwell system,   the Green's function to the linearized system  with applications were obtained in \cite{Duan3}.
 The existence and uniqueness of global strong solutions with large initial data and vacuum were given in \cite{Zhu}. And the large-time behavior of solutions to the outflow problem
 was studied in \cite{Zhu2}. Moreover,
the authors in \cite{Zhu3} studied  the large-time asymptotic behavior of solutions to the
superposition  of a viscous contact
wave with two rarefaction waves.

Back to the convergence of VMB to NSMF, the convergence is not uniform near $t=0$ because of  initial layer unless we impose extra assumption on  the initial data $(f_{0}, g_{0}, E_0,B_0)$:
\be \label{initial}
\left\{\bal
f_{0}(x,v)=n_{0}(x)\chi_0+m_{0}(x)\cdot v\chi_0+q_{0}(x)\chi_4,\\
g_{0}(x,v)=\rho_{0}(x)\chi_0, \  \ \Tdx\cdot E_0(x)=\rho_0(x), \\
\Tdx\cdot m_{0}=0, \ \ n_{0} + \sqrt{\frac23}  q_{0}=0, \ \ \Tdx\cdot B_0=0.
\ea\right.
\ee

In order to estimate the initial layer, we need the following decomposition.
For  $U=U(x)\in \R^3$, we denote the Helmholtz's decomposition $U=U_{||}+U_{\bot} $ with
\be U_{||}=\Delta^{-1}_x\Tdx\divx U, \quad U_{\bot}=\Delta^{-1}_x(\Tdx\times \Tdx\times U). \label{div}\ee
Correspondingly, for   $f=f(x,v)\in L^2$, we define the following  projections
\be \label{div1}
\left\{\bln
P_{||}f&=(f,v\chi_0)_{||}v\chi_0+(f,\tilde{h}_1) \tilde{h}_1,\\
P_{\bot}f&=(f,v\chi_0)_{\bot}v\chi_0+ (f,\tilde{h}_0 ) \tilde{h}_0+P_1f,
\eln\right.
\ee
where
\be \tilde{h}_0=\sqrt{\frac25}\chi_0-\sqrt{\frac35}\chi_4, \quad \tilde{h}_1=\sqrt{\frac35}\chi_0+\sqrt{\frac25}\chi_4. \ee


 To state the main results, we need the following  notations. First of all, $C, c$ denote some generic constants. For any $\alpha=(\alpha_1,\alpha_2,\alpha_3)\in \mathbb{N}^3$ and $\beta=(\beta_1,\beta_2,\beta_3)\in \mathbb{N}^3$, set
$$\dxa=\pt^{\alpha_1}_{x_1}\pt^{\alpha_2}_{x_2}\pt^{\alpha_3}_{x_3},\quad \dvb=\pt^{\beta_1}_{v_1}\pt^{\beta_2}_{v_2}\pt^{\beta_3}_{v_3}.$$
The Fourier transform of $h=h(x)$
is denoted by
$$\hat{h}(\xi)=\mathcal{F}h(\xi)=\frac1{(2\pi)^{3/2}}\intr h(x)e^{-  i x\cdot\xi}dx,$$
where and throughout this paper we denote $i=\sqrt{-1}$.

For any $q\in [1,\infty]$,  the Sobolev Space $L^{q}=L^q_x(L^2_v)$  for function $f=f(x,v)$ or $L^{q}=L^q_x(L^2_v)\times L^q_x\times L^q_x$ for vector $U=(g(x,v),E(x),B(x)) $
is defined  with the norms
\bmas
\|f\|_{L^{q}}&=\bigg(\intr\bigg(\intr|f(x,v)|^2 dv\bigg)^{q/2}dx\bigg)^{1/q},\\
\|U\|_{L^q}&=\bigg(\intr\bigg(\intr|g(x,v)|^2 dv\bigg)^{q/2}dx\bigg)^{1/q}+\bigg(\intr (|E(x)|^q+|B(x)|^q) dx\bigg)^{1/q}.
\emas
Also for an  integer $k\ge 1$ and $q\in [1,\infty]$,  the Sobolev Space $H^{k}=H^{k}_x(L^2_v)$ (or $H^{k}=H^{k}_x(L^2_v)\times H^{k}_x\times H^{k}_x$) for a function $f=f(x,v)$ (or a  vector $U=(g(x,v),E(x),B(x)) $) is defined with the norms
\bmas
\|f\|_{H^k}&=\(\intr (1+|\xi|^2)^k\intr |\hat{f}|^2dv d\xi \)^{1/2}, \\
\|U\|_{H^k}&=\(\intr (1+|\xi|^2)^k \(\intr |\hat{g}|^2dv+ |\hat E|^2+| \hat B|^2\)d\xi\)^{1/2};
\emas
and the Sobolev space $ W^{k,q}=W^{k,q}_x(L^2_v)  $ (or $W^{k,q}=W^{k,q}_x(L^2_v)\times W^{k,q}_x\times W^{k,q}_x$) is defined with the norms
\bmas
\|f\|_{W^{k,q}}&=\sum_{|\alpha|\le k}\(\intr  \(\intr |\dxa f|^2dv\)^{q/2} dx\)^{1/q}, \\
\|U\|_{W^{k,q}}&=\sum_{|\alpha|\le k}\(\intr  \(\intr |\dxa g|^2dv+ |\dxa E|^2+|\dxa B|^2\)^{q/2} dx\)^{1/q}.
\emas
Some weighted Sobolev space $ D^k_{l} $ ($D^k=D^k_0$)  will also be used with the norms given by
\bmas
 \|f\|_{D^k_{l}}&=\sum_{|\alpha|+|\beta|\le k}\|\nu^l\dxa\dvb f\|_{L^2 },\\
 \|U\|_{D^k_{l}}&=\sum_{|\alpha|+|\beta|\le k}\|\nu^l\dxa\dvb g\|_{L^2 }+ \sum_{|\alpha| \le k}(\|\dxa E\|^2_{L^2_x }+\| \dxa B\|^2 _{L^2_x })
\emas



We now state the following two main results in this paper.

\begin{thm}\label{thm1.1}
 For any $\eps\in (0,1)$, there exists a small constant $\delta_0>0$ such that if the initial data $U_0=(f_0,g_0,E_0,B_0)$ satisfy that $\|U_{0}\|_{D^6_1 }+\|U_{0}\|_{L^1 } \le \delta_0$,  then the VMB system \eqref{VMB4}--\eqref{VMB2i}  admits a unique global solution  $U_{\eps}(t,x,v)= (f_{\eps},g_{\eps},E_{\eps},B_{\eps})$ satisfying the following  time-decay estimate:
\be
   \|U_{\eps}(t)\|_{D^4_1} \le C\delta_0 (1+t)^{-\frac34}.
\ee

 Also, there exists a small constant $\delta_0>0$ such that if $\|U_{0}\|_{H^4 }+\|U_{0}\|_{L^1 } \le \delta_0$, then
the NSMF system \eqref{NSM_2}--\eqref{NSP_5i} admits a unique global solution $\tilde{U}(t,x)=(n,m,q,\rho,E,B) $  satisfying the following  time-decay estimate:
\be
   \|\tilde{U}(t)\|_{H^4_x} \le C\delta_0 (1+t)^{-\frac34}. 
\ee
\end{thm}

\begin{thm}\label{thm1.2}
There exist small positive constants $\delta_0$ and $\eps_0$ such that if the initial data $U_0=(f_0,g_0,E_0,B_0)$ satisfies that $\|U_{0}\|_{D^9_1 }+\|U_{0}\|_{L^1 } \le \delta_0$, then there exists a unique function $U_1=(u_1,V_1) $ such that for any $\eps\in (0,\eps_0)$, the solution $U_{\eps}=(f_{\eps},V_{\eps})$ with $V_{\eps}=(g_{\eps},E_{\eps},B_{\eps})$ to the VMB system \eqref{VMB4}--\eqref{VMB2i} satisfies
\bma &\quad \|P_{||}f_{\eps} (t) \|_{W^{2,\infty} }+\|P_{\bot}f_{\eps} (t)-u_1 (t)\|_{H^{2} }+\| V_{\eps} (t)-V_1 (t)\|_{H^{2} } \nnm\\
&\le C\delta_0\bigg(\eps|\ln\eps|^2 (1+t)^{-\frac12} +\(1+ \frac{t}{\eps}\)^{-1}\bigg), \label{limit0}
\ema
where $u_1=n \chi_0+m\cdot v\chi_0+q\chi_4$ and $V_1=(\rho  \chi_0,E,B) $ with $(n,m,q,\rho ,E,B)(t,x)$ being the solution to the incompressible NSMF system \eqref{NSM_2}--\eqref{NSP_5i}.

Moreover, if the initial data $U_0$ satisfies \eqref{initial} and   $\|U_{0}\|_{H^9} +\|U_{0}\|_{L^1}\le \delta_0$, then we have
\be
  \|P_{||}f_{\eps} (t) \|_{W^{2,\infty} }+\|P_{\bot}f_{\eps} (t)-u_1 (t)\|_{H^{2} }+\| V_{\eps} (t)-V_1 (t)\|_{H^{2} }\le C\delta_0 \eps|\ln\eps| (1+t)^{-\frac12} . \label{limit_1a}
\ee
\end{thm}


\begin{remark}
By Sobolev Embedding Theorem, we have the following estimates.

{\rm (1)} Under the first assumption in Theorem \ref{thm1.2},
 the solution $U_{\eps} $ to the VMB system \eqref{VMB4}--\eqref{VMB2i} satisfies
\be  \| U_{\eps} (t)-U_1 (t)\|_{L^{\infty} }\le C\delta_0\bigg(\eps|\ln\eps|^2 (1+t)^{-\frac12} +\(1+ \frac{t}{\eps}\)^{-1}\bigg).
\ee

{\rm (2)} Under the second assumption of Theorem \ref{thm1.2}, it holds that
\be
 \| U_{\eps} (t)-U_1 (t)\|_{L^{\infty} }\le C\delta_0\eps|\ln\eps| (1+t)^{-\frac12} .
\ee
\end{remark}

Before the rest of the introduction, we will briefly present the main ideas and the approach of the analysis in the proof.
 The convergence rates given in Theorem~\ref{thm1.2}  of diffusion limit of the VMB system  are proved based on the spectral analysis \cite{Li4} and the ideas inspired by \cite{FL-3,Li1}.
First of all, the solution $U_{\eps}(t)=(f_{\eps},V_{\eps} )(t)$ with $V_{\eps}(t)=( g_{\eps},E_{\eps},B_{\eps})(t)$ to the VMB system \eqref{VMB4}--\eqref{VMB2i} can be represented by
$$
\left\{\bln
f_{\eps}(t)&=e^{\frac{t}{\eps^2}\BB_{\eps}}f_0+\intt e^{\frac{t-s}{\eps^2}\BB_{\eps}}H^1_{\eps}(s) ds,
\\
V_{\eps}(t)&=e^{\frac{t}{\eps^2}\AA_{\eps}}V_0+\intt e^{\frac{t-s}{\eps^2}\AA_{\eps}} (H^2_{\eps}(s),0,0)  ds;
\eln\right.
$$
and the solution $U(t)=(u_1,V_1 )(t)$ with $u_1 =n\chi_0+m\cdot v\chi_0+q\chi_4$ and $V_1=(\rho \chi_0,E,B)$ to the NSMF system  \eqref{NSM_2}--\eqref{NSP_5i} can be represented by
$$
\left\{\bln
u_1(t)&=Y_1(t)P_0f_0+\intt Y_1(t-s)H_4(s) ds,
\\
V_1(t)&=Y_2(t)P_2V_0+\intt Y_2(t-s) H_5(s) ds,
\eln\right.
$$
where   $\BB_{\eps}$ and $\AA_{\eps}$ are the linear Boltzmann and VMB operators defined by \eqref{B4} and \eqref{B5}, $Y_1(t)$ and $Y_2(t)$ are two semigroups defined in \eqref{v1} and \eqref{v1a}, and $H_4, H_5$ are nonlinear terms given by
$$
\left.\bln
&H_4=(\rho E+j\times B)\cdot v\chi_0-\Tdx\cdot (m\otimes m)\cdot v\chi_0-\frac53 \Tdx\cdot (qm)\chi_4,\\
&H_5=-(\Tdx\cdot(\rho m-\eta m\times B)\chi_0,\rho m-\eta m\times B,0) .
\eln\right.
$$
The main idea of the proof is to first estimate the convergence rates from $e^{\frac{t}{\eps^2}\BB_{\eps}}$ to $Y_1(t)$ and from $e^{\frac{t}{\eps^2}\AA_{\eps}}$ to $Y_2(t)$ separately by using spectral analysis. Then to obtain  the convergence rates from $(f_{\eps},V_{\eps} )(t)$ to $(u_1,V_1 )(t)$  is based on the convergence rates on semigroups and a bootstrap argument.  Note that the a priori estimates on the solutions $u_1(t)$ and $V_1(t)$ can be closed in
	$H^2$. However, even though   $L^\infty$-norm also works for $u_1(t)$, it is not suitable for $V_1(t)$.

Note that the linear Boltzmann operator $\BB_{\eps}(\xi)$ given by \eqref{B2} satisfies the scaling transformation $\BB_{\eps}(\xi)=\BB_1(\eps\xi)$. This implies that the eigenvalues $\gamma_j(|\xi|,\eps)$ of $\BB_{\eps}(\xi)$ depend only on $\eps|\xi|$ and satisfy  $\gamma_j(|\xi|,\eps)=\tilde{\gamma}_j(\eps|\xi|)$ for $\eps|\xi|$ small, where $\tilde{\gamma}_j(|\xi|)$  are the eigenvalues  of $\BB_1(\xi)$ in the low frequency regime. Precisely, there exist five eigenvalues $\gamma_{j}(|\xi|,\eps)$ for $\eps |\xi| $  small and they satisfy
 \be
 \gamma_{j}(|\xi|,\eps) =  i\mu_j\eps|\xi|-a_{j}\eps^2|\xi|^2+O\( \eps^3|\xi|^3\), \label{expan1}
\ee
where $\mu_{\pm1}=\pm \sqrt{\frac53}$, $\mu_k=0$ $(k=0,2,3)$, and $a_j>0$ are some constants.

Moreover, we can decompose the semigroup $e^{ \frac{t}{\eps^2}\BB_{\eps}(\xi)}$ into the fluid part and the remainder part, where the remainder part has the decay rate $e^{-\frac{bt}{\eps^2}}$ (See Theorem \ref{rate2}). Then by applying the expansion \eqref{expan1} to the fluid part, we can rewrite $e^{ \frac{t}{\eps^2}\BB_{\eps}(\xi)}$ as
\bmas
P_{||}(e^{ \frac{t}{\eps^2}\BB_{\eps}(\xi)}\hat{f}_0)&=\sum_{j=\pm1} e^{\frac{i\mu_j |\xi|}{\eps}t-a_{j} |\xi|^2t }P_{0j}\hat{f}_0+O(\eps e^{-c|\xi|^2t})+O(e^{-\frac{bt}{\eps^2}}), \\
P_{\bot}(e^{ \frac{t}{\eps^2}\BB_{\eps}(\xi)}\hat{f}_0)&=\sum_{j=0,2,3}  e^{-a_{j} |\xi|^2t }P_{0j}\hat{f}_0 +O(\eps e^{-c|\xi|^2t})+O(e^{-\frac{bt}{\eps^2}}),
\emas
where $P_{0j}$, $j=-1,0,1,2,3$ are the first order eigenprojections corresponding to $\gamma_j$.
Thus, by using the following key estimate:
$$\left\|\mathcal{F}^{-1}\(e^{\frac{\pm i|\xi|}{\eps}t}(1+|\xi|)^{-3}\)\right\|_{L^{\infty}_x}\le C \(\frac{t}{\eps}\)^{-1} ,$$
we can establish the optimal convergence rate of the  semigroup $e^{\frac{t}{\eps^2}\BB_{\eps}}$  to its first and second order fluid limits in a combination of $L^2-L^\infty$ norm ($L^\infty$ norm for $P_{||}$ part and $L^2$ norm for $P_{\bot}$ part) as given in Lemma \ref{limit5}. 

On the other hand, due to the influence of the electric-magnetic field,  the linear VMB operator $\AA_{\eps}(\xi)$ given in \eqref{B3} has no scaling property.
To study the corresponding eigenvalue problem, we will use  a new non-local implicit function theorem to show that there exist five eigenvalues $\lambda_j(|\xi|,\eps)$, $j=0,1,2,3,4$ of $\AA_{\eps}(\xi)$ for $\eps(1+|\xi|)$ being small and they satisfy (See Lemma \ref{eigen_4a}):
 \bma
 \lambda_{0}(|\xi|,\eps) &= \eps^2b_{0}(|\xi|) +O(\eps^4 (1+|\xi|^2)^2), \label{expan2}\\
 \lambda_k(|\xi|,\eps)&=\eps^2b_k(|\xi|)+\left\{\bal
O( \eps^4|b_k(|\xi|)|), & |\eta^2-4|\xi|^2|\ge r_0,\\
O(\eps^3),& |\eta^2-4|\xi|^2|< r_0,
\ea\right.\label{expan3}
 \ema
where $k=1,2,3,4,$ and $b_j(|\xi|)$ are defined by \eqref{bj}.  

Moreover, we can decompose the semigroup $e^{ \frac{t}{\eps^2}\AA_{\eps}(\xi)}$ into two fluid parts for low frequency and high frequency and the remainder part, where the remainder part has the decay rate $e^{-\frac{bt}{\eps^2}}$ (See Theorem \ref{rate1}). Then by applying the expansion \eqref{expan2}--\eqref{expan3} to the fluid part, we obtain
$$
e^{ \frac{t}{\eps^2}\AA_{\eps}(\xi)}\hat{V}_0=\sum^4_{j=0} e^{b_j(|\xi|)t }\tilde{P}_{0j}\hat{V}_0 +O(\eps e^{-{\rm Re}b_jt}) +O( e^{-\frac{ct}{\eps|\xi|}})1_{\{|\xi|\ge \frac{r_1}{\eps}\}}+O(e^{-\frac{bt}{\eps^2}}),
$$
where $\tilde{P}_{0j}$, $j=0,1,2,3,4$ are first order eigenprojections corresponding to $\lambda_j$.  Thus,  we can establish the optimal convergence rate of the semigroup $e^{\frac{t}{\eps^2}\AA_{\eps}}$ to its first and second order fluid limits in $L^2$ norm as listed in Lemmas \ref{fl1} and \ref{fl2}.

By using the estimates on the convergence rates for the fluid limits of the linear Boltzmann and VMB systems, we can  prove the convergence and establish the optimal convergence rate of the strong solution $(f_{\eps},V_{\eps})$ to the nonlinear VMB system towards the solution $(u_1,V_1)$ to the NSMF system.  Hence, we obtain the precise estimation on the initial layer. To illustrate why the Helmholtz's decomposition \eqref{div1} is necessary, we consider the term
$$\intt Y_1(t-s)(j_{\eps}\times B_{\eps}-j\times B)ds=\intt Y_1(t-s)(\Tdx(\rho_{\eps}-\rho)\times B_{\eps})ds+\cdots.$$
If one directly estimate this term in  $ L^\infty$, we have  non-integrability in time because
$$\left\|\intt Y_1(t-s)(\Tdx(\rho_{\eps}-\rho)\times B_{\eps})ds \right\|_{W^{2,\infty}}\le \intt (t-s)^{-\frac54}\|\rho_{\eps}-\rho\|_{H^2}\|B_{\eps}\|_{H^2}ds.$$
However, if we apply the  $L^2-L^\infty$ by noting that   $P_{||}Y_1(t)=0$, $P_{\bot}Y_1(t)=Y_1(t)$, then  the time integrability holds because
$$\left\|\intt Y_1(t-s)(\Tdx(\rho_{\eps}-\rho)\times B_{\eps})ds \right\|_{H^2}\le \intt (t-s)^{-\frac12}\|\rho_{\eps}-\rho\|_{H^2}\|B_{\eps}\|_{H^2}ds.$$

The rest of this paper will be organized as follows.
In Section~\ref{sect2}, we will present the results about the spectrum analysis of the linear operator related to the linearized VMB system. In Section~\ref{sect3}, we  will establish  the first and second order fluid approximations of the solution to the  linearized VMB system.
In Section~\ref{sect4}, we will prove the convergence and establish the convergence rate of the global solution to the original nonlinear VMB system to the solution to the nonlinear NSMF system.

\section{Spectral analysis}
\setcounter{equation}{0}
\label{sect2}

In this section, we will study  the spectral analysis of the linear VMB operator $\AA_{\eps}(\xi)$ defined in
\eqref{Axi}. 
From the system \eqref{VMB4}--\eqref{VMB4d}, we have the following linearized VMB system for $f_{\eps}$ and $U_{\eps}=(g_{\eps},E_{\eps},B_{\eps})^T$:
\be\label{LVMB0}
\left\{\bln
 &\eps^2\dt f_{\eps}=\BB_{\eps}f_{\eps},\quad t>0,\\
 & f_{\eps}(0,x,v)=f_0(x,v),
 \eln\right.
\ee
and
\be\label{LVMB1}
\left\{\bln
&\eps^2\dt U_{\eps}=\AA_{\eps}U_{\eps},\quad t>0,\\
&\Tdx\cdot E_{\eps}=(g_{\eps},\chi_0),\quad \Tdx\cdot B_{\eps}=0,\\
 &U_{\eps}(0,x,v)=U_0(x,v)=(g_0,E_0,B_0),
 \eln\right.
\ee
where
\bma
\BB_{\eps}&=L-\eps v\cdot\Tdx, \label{B4}\\
\AA_{\eps}&=\left( \ba
L_1-\eps v\cdot\Tdx & \eps v\chi_0 &0\\
-\eps P_m &0 &\eps^2\Tdx\times\\
0 &-\eps^2\Tdx\times &0
\ea\right) \label{B5}
\ema with $P_mh=(h,v\chi_0)$ for any $h\in L^2(\R^3_v)$.

Taking Fourier transform to \eqref{LVMB0} and \eqref{LVMB1} in $x$ yields
\be\label{LVMB2}
\left\{\bln
 &\eps^2\dt \hat{f}_{\eps}=\BB_{\eps}(\xi)\hat{f}_{\eps},\quad t>0,\\
 & \hat{f}_{\eps}(0,\xi,v)=\hat{f}_0(\xi,v),
 \eln\right.
\ee
and
\be\label{LVMB3}
\left\{\bln
 &\eps^2\dt \hat{U}_{\eps}= \AA_{\eps}(\xi)\hat U_{\eps},\quad t>0,\\
&i\xi\cdot \hat{E}_{\eps}=(\hat{g}_{\eps},\chi_0),\quad i\xi\cdot \hat{B}_{\eps}=0,\\
  &\hat{U}_{\eps}(0,\xi,v)=\hat{U}_0(\xi,v)=(\hat{g}_0,\hat{E}_0,\hat{B}_0),
   \eln\right.
\ee
where  
\bma
\BB_{\eps}(\xi)&=L- i\eps v\cdot\xi, \label{B2}\\
\AA_{\eps}(\xi)&=\left( \ba
L_1- i\eps v\cdot\xi & \eps v\chi_0 &0\\
-\eps P_m &0 &i\eps^2\xi\times\\
0 &-i\eps^2\xi\times &0
\ea\right). \label{B3}
\ema

By the identity $X=(X\cdot y)y-y\times y\times X$ for any $X\in \R^3$ and $y\in \mathbb{S}^2$, we can transform the system \eqref{LVMB3} to a new system for $\hat V_{\eps}=(\hat g_{\eps},\omega\times\hat  E_{\eps},\omega\times\hat B_{\eps})$ with $\omega=\xi/|\xi|$ as
 \bq
 \left\{\bln            \label{LVMB2a}
 &\dt \hat{V}_{\eps}=\tilde{\AA}_{\eps}(\xi)\hat{V}_{\eps}, \quad t>0,\\
 &\hat{V}_{\eps}(0,\xi,v)=\hat{V}_0(\xi,v)=(\hat{g}_0,\omega\times\hat{E}_0,\omega\times\hat{B}_0),
 \eln\right.
 \eq
with 
\be \label{Axi}
\tilde{\AA}_{\eps}(\xi)=\left( \ba
\tilde\BB_{\eps}(\xi) &-\eps v\chi_0\cdot\omega\times &0\\
-\eps\omega\times P_m &0 &i\eps^2\xi\times\\
0 &-i\eps^2\xi\times &0
\ea\right).
\ee
Here,  for $\xi\ne0$,
 \be
\tilde\BB_{\eps}(\xi) =L_1-i\eps v\cdot\xi -i\eps \frac{ v\cdot\xi }{|\xi|^2}P_{ d}.  \label{B(xi)}
 \ee

\begin{remark}\label{rem1.1}
The eigenvalues of the operator $\AA_{\eps}(\xi)$ are same as those
of $\tilde{\AA}_{\eps}(\xi)$,  and the eigenfunctions of $\AA_{\eps}(\xi)$ can be obtained as linear combinations
 of those for  $\tilde{\AA}_{\eps}(\xi)$. In fact, let   $\beta$  be an eigenvalue with
 the corresponding eigenvector denoted by $\mathcal{U}=(\phi,E,B)$ of $\tilde{\AA}_{\eps}(\xi)$. Then $U=(\phi,-i\frac{\xi}{|\xi|^2}(\phi,\chi_0)-\frac{\xi}{|\xi|}\times E,-\frac{\xi}{|\xi|}\times B)$ is the corresponding
 eigenvector with the eigenvalue  $\beta$ for $\AA_{\eps}(\xi)$ .
\end{remark}

\subsection{Spectrum structure of $\AA_{\eps}(\xi)$}

\subsubsection{Spectrum structure}

As usual, we denote a weighted Hilbert space $L^2_{\xi}(\R^3)$ for $\xi\ne 0$
as
$$
 L^2_{\xi}(\R^3)=\{f\in L^2(\R^3)\,|\,\|f\|_{\xi}=\sqrt{(f,f)_{\xi}}<\infty\},
$$
with the inner product defined by
$$
 (f,g)_{\xi}=(f,g)+\frac1{|\xi|^2}(P_{d} f,P_{d} g).
$$
For any fixed $\xi\ne 0$, we define a subspace of $\C^3$ by
$$\mathbb{C}^3_{\xi}=\{y\in \mathbb{C}^3\,|\, y\cdot \xi=0\}.$$

For any vectors $U=(f,E_1,B_1),V=(g,E_2,B_2)\in L^2_\xi(\R^3_v)\times \mathbb{C}^3\times \mathbb{C}^3$,  a weighted inner product with the corresponding
norm is defined by
$$ (U,V)_\xi=(f,g)_\xi+(E_1,E_2)+(B_1,B_2),\quad \|U\|_\xi=\sqrt{(U,U)_\xi}. $$
Another $L^2$ inner product and norm  is denoted by
$$ (U,V)=(f,g)+(E_1,E_2)+(B_1,B_2),\quad \|U\|=\sqrt{(U,U)}. $$

Since $P_{d}$ is a self-adjoint projection operator, it follows that
 $(P_{d} f,P_{d} g)=(P_{d} f, g)=( f,P_{d} g)$ and hence
 \bq (f,g)_{\xi}=(f,g+\frac1{|\xi|^2}P_{d}g)=(f+\frac1{|\xi|^2}P_{d}f,g).\label{C_1a}\eq
By
\eqref{C_1a}, we have for any $f,g\in L^2_{\xi}(\R^3_v)\cap D(\tilde\BB_{\eps}(\xi))$,
 \be
 (\tilde\BB_{\eps}(\xi)f,g)_{\xi}=(\tilde\BB_{\eps}(\xi) f,g+\frac1{|\xi|^2}P_{d} g)=(f,\tilde\BB_{\eps}(-\xi)g)_{\xi}. \label{L_7}
\ee
Moreover,  $\tilde\BB_{\eps}(\xi)$ is a dissipate operator in $L^2_{\xi}(\R^3)$:
 \be
 {\rm Re}(\tilde\BB_{\eps}(\xi)f,f)_{\xi}=(L_1f,f)\le 0. \label{L_8}
\ee

Note that $\tilde\BB_{\eps}(\xi)$ is a linear operator from the space $L^2_{\xi}(\R^3)$ to itself, and  for $ y\in \mathbb{C}^3_{\xi}$,
\bq \frac{\xi}{|\xi|}\times\frac{\xi}{|\xi|}\times y=-y.\label{rotat}\eq
Hence,  $L^2_{\xi}(\R^3_v)\times \mathbb{C}^3_{\xi}\times \mathbb{C}^3_{\xi}$ is an invariant subspace of the operator $\tilde{\AA}_{\eps}(\xi)$. Thus,  $\tilde{\AA}_{\eps}(\xi)$ can
be regarded as a linear operator on $L^2_{\xi}(\R^3_v)\times \mathbb{C}^3_{\xi}\times \mathbb{C}^3_{\xi}$ and it satisfies for any $U=(g,X,Y) \in L^2_{\xi}(\R^3_v)\times \mathbb{C}^3_{\xi}\times \mathbb{C}^3_{\xi}$ that
 \be
 {\rm Re}(\tilde\AA_{\eps}(\xi)U,U)_{\xi}=(L_1g,g)\le 0 . \label{L_8a}
\ee

Denote the spectrum of the operator $A$ by  $\sigma(A)$.
The essential spectrum of $A$, denoted by $\sigma_{ess}(A)$, is the set
$\{\lambda\in \C \,|\, \lambda-A~{\rm is~not~a~Fredholm~operator}\}$ (cf. \cite{Kato}). The discrete spectrum
of $A$, denoted by $\sigma_d(A)$, is the set $\sigma(A)\setminus \sigma_{ess} (A)$ which consists of  all isolated eigenvalues with finite multiplicity.
 And  $\rho(A)$ denotes its resolvent set.

We first have the following lemma about $\tilde{\AA}_{\eps}(\xi)$.


\begin{lem}\label{SG_1}
The operator $\tilde{\AA}_{\eps}(\xi)$ generates a strongly continuous contraction semigroup on
$L^2_{\xi}(\R^3)\times \C^3_{\xi}\times \C^3_{\xi}$, which satisfies
$$
\|e^{t\tilde{\AA}_{\eps}(\xi)}U\|_{\xi}\le\|U\|_{\xi}, \quad \forall\, t>0,\,U\in
L^2_{\xi}(\R^3_v)\times \C^3_{\xi}\times \C^3_{\xi}.
$$
Moreover,  $\rho(\tilde{\AA}_{\eps}(\xi))\supset\{\lambda\in \C\,|\, {\rm Re}\lambda>0\}$.
\end{lem}
\begin{proof}
We first show that both $\tilde{\AA}_{\eps}(\xi)$ and $\tilde{\AA}^*_{\eps}(\xi)$ are
dissipative operators on $L^2_{\xi}(\R_v^3)\times \C^3_{\xi}\times \C^3_{\xi}$. By \eqref{L_7}, for any $U,V\in D(\tilde\BB_{\eps}(\xi))\times \mathbb{C}^3_{\xi}\times \mathbb{C}^3_{\xi}$
it holds that
$$
 (\tilde{\AA}_{\eps}(\xi)U,V)_{\xi}=(U,\tilde{\AA}^*_{\eps}(\xi)V)_{\xi},
$$
where
\be
\tilde{\AA}^*_{\eps} (\xi)=\left( \ba
\tilde\BB_{\eps}(-\xi) &\eps v\chi_0\cdot\omega\times   &0\\
\eps\omega\times P_m &0 & -i\eps^2\xi\times\\
0 &  i\eps^2\xi\times &0
\ea\right).
\ee

By \eqref{L_8a}, both $\tilde{\AA}_{\eps}(\xi)$ and $\tilde{\AA}^*_{\eps}(\xi)$ are dissipative,  namely,
$$\mathrm{Re}(\tilde{\AA}_{\eps}(\xi)U,U)_{\xi}=\mathrm{Re}(\tilde{\AA}^*_{\eps}(\xi)U,U)_{\xi}=(L_1g,g)\leq0,\quad \forall\, U=(g,X,Y).$$
Since $\tilde{\AA}_{\eps}(\xi)$ is a densely defined closed operator, it
follows from Corollary 4.4 in \cite{Pazy} that the operator $\tilde{\AA}_{\eps}(\xi)$
generates a $C_0$-contraction semigroup on $L^2_{\xi}(\R^3_v)\times \mathbb{C}^3_{\xi}\times \mathbb{C}^3_{\xi}$. In addition, it holds that $\rho(\tilde{\AA}_{\eps}(\xi))\supset\{\lambda\in \C\,|\, {\rm Re}\lambda>0\}$.
\end{proof}

Now we denote by $T$ a linear operator on $L^2(\R^3_v)$ or
$L^2_{\xi}(\R^3_v)$ with norms by
$$
 \|T\|=\sup_{\|f\|=1}\|Tf\|,\quad
 \|T\|_{\xi}=\sup_{\|f\|_{\xi}=1}\|Tf\|_{\xi}.
$$
Obviously,
 \be
(1+|\xi|^{-2})^{-1/2}\|T\|\le \|T\|_{\xi}\le (1+|\xi|^{-2})^{1/2}\|T\|.\label{eee}
 \ee

Set
 \bq  D_{\eps}(\xi)=-\nu(v)-i\eps v\cdot\xi, \quad \mathbb{B}_2(\xi)=\left(\ba  0 & i\xi\times \\ -i\xi\times & 0 \ea\right)_{6\times6}.\label{B1}\eq
Since $\mathbb{C}^3_{\xi}\times \mathbb{C}^3_{\xi}$ is an invariant subspace of the operator $\mathbb{B}_2(\xi)$, we can regard $\mathbb{B}_2(\xi)$ as an operator on $\mathbb{C}^3_{\xi}\times \mathbb{C}^3_{\xi}$. 
 Moreover, the operator $\lambda-\mathbb{B}_2(\xi)$ is invertible on $\mathbb{C}^3_{\xi}\times \mathbb{C}^3_{\xi}$  for any $\lambda\ne \pm i|\xi|$ and satisfies (cf. \cite{Li4})
\bq  \|(\lambda-\mathbb{B}_2(\xi))^{-1}\|= \max_{j=\pm1}|\lambda-ji|\xi||^{-1}.\label{b_1(xi)}\eq

We now study  the spectrum of $\tilde{\AA}_{\eps}(\xi)$.

\begin{lem}\label{Egn}
The following statements hold for all $\xi\ne 0$ and $\eps\in [0,1)$.
 \begin{enumerate}
\item[\rm (1)]
$\sigma_{ess}(\tilde{\AA}_{\eps}(\xi))\subset \{\lambda\in \mathbb{C}\,|\, {\rm Re}\lambda\le -\nu_0\}$ and $\sigma(\tilde{\AA}_{\eps}(\xi))\cap \{\lambda\in \mathbb{C}\,|\, -\nu_0<{\rm Re}\lambda\le 0\}\subset \sigma_{d}(\tilde{\AA}_{\eps}(\xi))$.
\item[\rm (2)]
 If $\lambda$ is an eigenvalue of $\tilde{\AA}_{\eps}(\xi)$, then ${\rm Re}\lambda<0$ for any $\eps \ne 0$  and $ \lambda=0$ if and only if $\eps=0$.
 \end{enumerate}
\end{lem}

\begin{proof}
We decompose $\tilde{\AA}_{\eps}(\xi)$ into
\be \tilde{\AA}_{\eps}(\xi)=G^1_{\eps}(\xi)+G^2_{\eps}(\xi), \ee
where
 \bma
G^1_{\eps}(\xi)&=\left( \ba
D_{\eps}(\xi) &0 &0\\
0 &0 &i\eps^2\xi\times\\
0 &-i\eps^2\xi\times &0
\ea\right), \label{A1}\\
G^2_{\eps}(\xi)&=\left( \ba
K_1-i\eps \frac{v\cdot\xi}{|\xi|^2}P_{d} &-\eps v\chi_0\cdot\omega\times &0\\
-\eps\omega\times P_m &0 &0\\
0 &0 &0
\ea\right). \label{A2}
 \ema
By  \eqref{b_1(xi)} and \eqref{nu}, the operator $\lambda-G^1_{\eps}(\xi)$ is invertible  on $L^2_{\xi}(\R^3_v)\times \mathbb{C}^3_{\xi}\times \mathbb{C}^3_{\xi}$ for ${\rm
Re}\lambda>-\nu_0$ and $\lambda\ne \pm i\eps^2|\xi|$, and it satisfies
$$
(\lambda-G^1_{\eps}(\xi))^{-1}= \left(\ba (\lambda -D_{\eps}(\xi))^{-1} & 0 \\  0 & (\lambda -\eps^2\mathbb{B}_2(\xi))^{-1} \ea\right)_{7\times7}.
$$
 Since $G^2_{\eps}(\xi)$ is a compact
operator on $L^2_{\xi}(\R^3_v)\times \mathbb{C}^3_{\xi}\times \mathbb{C}^3_{\xi}$ for any fixed $\xi\ne 0$,
$\tilde{\AA}_{\eps}(\xi)$ is a compact perturbation of $G^1_{\eps}(\xi)$. Hence, it follows from
 Weyl's Theorem (Theorem 5.35 in \cite{Kato}) that  
 $$\sigma_{ess}(\tilde{\AA}_{\eps}(\xi))=\sigma_{ess}(G^1_{\eps}(\xi))=R(D_{\eps}(\xi))\subset \{\lambda\in \C\,|\, {\rm Re}\lambda\le -\nu_0\}.$$
 Thus
 the spectrum of $\tilde{\AA}_{\eps}(\xi)$ in the domain ${\rm Re}\lambda>-\nu_0$ consists of
discrete eigenvalues with possible accumulation
points only on the line ${\rm Re}\lambda= -\nu_0$. This and Lemma \ref{SG_1} prove the part (1).

We claim that for any  $\lambda \in \sigma_d(\tilde{\AA}_{\eps}(\xi))$  in the region $\mathrm{Re}\lambda>-\nu_0$, it holds that ${\rm Re}\lambda <0$
for $\eps \ne 0$. Indeed, set $\xi=s\omega$ and let $U=(f,E,B)\in L^2_{\xi}(\R^3_v)\times \mathbb{C}^3_{\xi}\times \mathbb{C}^3_{\xi}$ be the
eigenvector corresponding to the eigenvalue $\lambda$ so that
 \bq           \label{L_6}
  \left\{\bal
 \lambda f=L_1f-i\eps s(v\cdot\omega) (f+\frac1{s^2}P_{d} f )-\eps v\chi_0\cdot(\omega\times E),\\
 \lambda E=-\eps\omega\times (f,v\chi_0)+i \eps^2\xi\times B,\\
 \lambda B=-i\eps^2\xi\times E.
 \ea\right.
 \eq
Taking the inner product  $\eqref{L_6}_1$ with $f+\frac1{s^2}P_df$, we have
$$
\text{Re}\lambda\(\|f\|^2_{\xi} +|E|^2+|B|^2\)= (L_1f,f)\le 0,
$$
which  implies $\text{Re}\lambda\leq 0$.

Furthermore, if there exists an eigenvalue $\lambda$ with ${\rm
Re}\lambda=0$, then it follows from the above that $(L_1f,f)=0$,
namely, 
$f=C_0\sqrt M\in N_1$. Substitute this into \eqref{L_6}, we obtain
$$
 \lambda C_0\sqrt M=-i\eps  (v\cdot\omega)\Big(s+\frac1{s} \Big)C_0\chi_0-\eps v\chi_0\cdot(\omega\times E),
$$
which   implies that $C_0\ne 0$ and $\omega\times E\ne 0$ unless $\eps= 0$ and $\lambda= 0$. When $\eps\ne 0$, it holds that $C_0= 0$ and $\omega\times E= 0$ and hence $f=0$ and $E=0$. Substitute this into \eqref{L_6}, we obtain $B\equiv0$. This is a contradiction and
thus it holds $\text{Re}\lambda<0$ for $\eps \ne 0$. This proves the part (2) and then it completes the proof of the lemma.
\end{proof}

We now consider the spectrum and resolvent sets of $\tilde{\AA}_{\eps}(\xi)$ for $\eps|\xi|$ large. For ${\rm Re}\lambda>-\nu_0$ and $\lambda\ne \pm i\eps^2|\xi|$, we  decompose  $\lambda-\tilde{\AA}_{\eps}(\xi)$ into
\bma
\lambda-\tilde{\AA}_{\eps}(\xi)&=\lambda-G^1_{\eps}(\xi)-G^2_{\eps}(\xi)\nnm\\
 &=\(I-G^2_{\eps}(\xi)(\lambda-G^1_{\eps}(\xi))^{-1}\)\(\lambda-G^1_{\eps}(\xi)\), \label{B_d}
\ema
where $G^1_{\eps}(\xi)$ and $G^2_{\eps}(\xi)$ are defined by \eqref{A1} and \eqref{A2} respectively.
A direct computation shows that
\be\label{X_2}
\left\{\bln
&G^2_{\eps}(\xi)(\lambda-G^1_{\eps}(\xi))^{-1}= \left(\ba X^1_{\eps}(\lambda,\xi) & X^2_{\eps}(\lambda,\xi) \\  X^3_{\eps}(\lambda,\xi) & 0 \ea\right)_{7\times7},\\
&X^1_{\eps}(\lambda,\xi)=\(K_1-i\eps\frac{ v\cdot\xi}{|\xi|^2} P_d\)(\lambda -D_{\eps}(\xi))^{-1},
\\
& X^2_{\eps}(\lambda,\xi)=\(\eps v\chi_0\cdot\omega\times,0_{1\times3}\)_{1\times6}(\lambda -\eps^2\mathbb{B}_2(\xi))^{-1},
\\
 &X^3_{\eps}(\lambda,\xi)=\left(\ba -\eps\omega\times P_m(\lambda-D_{\eps}(\xi))^{-1} \\ 0_{3\times1} \ea\right)_{6\times1}.
 \eln\right.
\ee

Then, we have the estimates on the  terms on the right hand side of \eqref{X_2} as follows.

\begin{lem} \label{LP03}
The following estimates hold.
\begin{enumerate}
\item[\rm (1)] For any $\delta>0$, if ${\rm Re}\lambda\ge -\nu_0+\delta$, then we have
 \be
  \|K_1(\lambda-D_{\eps}(\xi))^{-1}\| \le  C\delta^{-\frac12}(1+\eps|\xi|)^{-\frac12}.   \label{T_7}
 \ee
 \item[\rm (2)] For any $\delta>0,\, \tau_0>0$, if  ${\rm Re}\lambda\ge -\nu_0+\delta$ and $\eps|\xi|\le \tau_0$, then we have
 \be
 \|K_1(\lambda-D_{\eps}(\xi))^{-1}\| \leq C\delta^{-1}(1+\tau_0)^{\frac12}(1+|{\rm Im}\lambda|)^{-\frac12}.\label{T_8}
 \ee

 \item[\rm (3)]  For any $\delta>0$,  if ${\rm Re}\lambda\ge -\nu_0+\delta$, then we have
 \bma
    \|P_{m}(\lambda-D_{\eps}(\xi))^{-1}\| &\le
 C\delta^{-\frac12}(1+\eps|\xi|)^{-\frac12} , \label{L_9}\\
   \|P_{m}(\lambda-D_{\eps}(\xi))^{-1}\| &\le
 C(\delta^{-1}+1) (1+\eps|\xi|)|\lambda|^{-1}. \label{L_9a}
  \ema

   \item[\rm (4)]  For any $\delta>0$,  if ${\rm Re}\lambda\ge -\nu_0+\delta$, then we have
 \bma
 \|(v\cdot\xi)|\xi|^{-2}P_{d}(\lambda-D_{\eps}(\xi))^{-1}\|
 &\leq C \delta^{-1} |\xi|^{-1}, \label{L_9b}
 \\
    \|(v\cdot\xi)|\xi|^{-2}P_{d}(\lambda-D_{\eps}(\xi))^{-1}\|
 &\leq C(\delta^{-1}+1)(|\xi|^{-1}+1)  |\lambda|^{-1}. \label{L_9c}
  \ema
\end{enumerate}
\end{lem}
\begin{proof}The estimates \eqref{L_9}--\eqref{L_9c} are proved in Lemma 3.5 of \cite{Li4}. Thus, we only prove \eqref{T_7} and \eqref{T_8} in the following. Let $\lambda=x+iy$  with $(x,y) \in \R\times \R$.
Then
\bma
&\|K_1(\lambda-D_{\eps}(\xi))^{-1}f\|^2\nnm\\
=&\intr\left|\intr k_1(v,u)(\nu(u)+\lambda+i\eps u\cdot\xi)^{-1}f(u)du\right|^2dv\nnm\\
\le& \intr\(\intr k_1(v,u)\frac{1}{|\nu(u)+\lambda+i\eps u\cdot\xi|^{2}}du\)\(\intr k_1(v,u)|f(u)|^2 du\)dv\nnm\\
\le& C\sup_{v\in\R^3}\intr k_1(v,u)\frac1{(\nu(u)+x)^2+(y+ \eps u\cdot\xi)^2}du\|f\|^2.\label{K2}
\ema

From \eqref{L_1}, one has
\be |k_1(v,u)|\le C\frac{1}{|\bar{v}-\bar{u}|}e^{-\frac{|v-u|^2}8} ,\quad \bar{u}=(u_2,u_3).\label{K1}\ee
Let $\mathbb{O}$ be a rotation in $\R^3$ satisfying $\mathbb{O}^T\xi=(|\xi|,0,0)$. By changing variables $v\to \mathbb{O} v$, $u\to \mathbb{O} u$, we obtain from \eqref{K2} and \eqref{K1} that for $\delta=\nu_0+x >0$,
\bmas
&\|K_1(\lambda-D_{\eps}(\xi))^{-1}f\|^2\nnm\\
\le& C\sup_{v\in\R^3}\intr k_1(v,u)\frac1{(\nu(u)+x)^2+(y+ \eps u_1|\xi|)^2}du\|f\|^2\nnm\\
\le& C\sup_{v\in\R^3}\int_{\R}\frac1{(\nu_0+x)^2+(y+ \eps u_1|\xi|)^2}du_1\int_{\R^2}\frac{1}{|\bar{v}-\bar{u}|}e^{-\frac{|\bar{v}-\bar{u}|^2}8} d\bar{u}\|f\|^2\nnm\\
\le& C\frac1{\eps|\xi|}\int_{\R} \frac1{(\nu_0+x)^2+u_1^2}  du_1\|f\|^2\le C\delta^{-1}(\eps|\xi|)^{-1}\|f\|^2.
\emas
This gives \eqref{T_7}.

For \eqref{T_8}, we first decompose
\bma
\|K_1(\lambda-D_{\eps}(\xi))^{-1}f\|^2
&\le 2\intr \left|\int_{|u|\le R} k_1(v,u)(\nu(u)+\lambda+i\eps u\cdot\xi)^{-1}f(u)du\right|^2dv\nnm\\
 &\quad+ 2\intr \left|\int_{|u|\ge R} k_1(v,u)(\nu(u)+\lambda+i\eps u\cdot\xi)^{-1}f(u)du\right|^2dv\nnm\\
&= :I_1+I_2. \label{K3}
\ema
For $I_1$, it holds that
\bmas
I_1
\le& C\sup_{v\in\R^3}\int_{|u|\le R} k_1(v,u)\frac1{(\nu(u)+x)^2+(y+\eps u_1|\xi|)^2}du\|f\|^2\nnm\\
\le& C\sup_{v\in\R^3}\int^R_{-R}\frac1{(\nu_0+x)^2+(y+\eps u_1|\xi|)^2}du_1\int^R_{-R}\int^R_{-R}k_1(v,u) d\bar{u}\|f\|^2\nnm\\
\le& C\int^R_{-R}\frac1{(\nu_0+x)^2+(y+\eps u_1|\xi|)^2}du_1\|f\|^2 .
\emas

If $\eps|\xi|\le \tau_0$, $|u|\le R$ and $|y|\ge 2\tau_0R$, we have
$$
|y+\eps u_1|\xi||\ge|y|-\tau_0R \geq \frac{|y|}{2}.
$$
Thus
\be
I_1 \le C\int^R_{-R}\frac{1}{\delta^2+y^2}du_1\|f\|^2=C\frac{1}{\delta^2+y^2}R\|f\|^2. \label{K4}
\ee
For $I_2$, since
$$\intr k_1(v,u)du \le C(1+|v|)^{-1},$$
we obtain
\bma
I_2\le& \intr\(\int_{|u|\ge R} k_1(v,u)\delta^{-2}du\)\(\int_{|u|\ge R} k(v,u)|f(u)|^2 du\)dv\nnm\\
\le& C\delta^{-2}\int_{|u|\ge R} \(\intr k_1(v,u) dv\)|f(u)|^2du \nnm\\
\le& C\delta^{-2}\int_{|u|\ge R}(1+|u|)^{-1}|f(u)|^2du\le C\delta^{-2}R^{-1}\|f\|^2 . \label{K5}
\ema
By choosing $R= |y|/\max\{2,2\tau_0\}$,  \eqref{T_8} follows from \eqref{K3}--\eqref{K5}. And this completes the proof of the lemma.
\end{proof}

We now state a lemma from \cite{Li4}.

\begin{lem}[\cite{Li4}]\label{inver}
Let $K_1,K_4$ be the operators on the space $X$ and $Y$, and $K_2,K_3$ be the operators  $Y\to X$ and $X\to Y$ respectively. Let $K$ be a matrix operator on $X\times Y$ defined by
$$
K=\left(\ba K_1 & K_2 \\  K_3 & K_4 \ea\right).
$$
 If the norms of $K_1,K_2,K_3$ and $K_4$ satisfy 
$$\|K_1\|<1,\quad \|K_4\|<1,\quad \|K_2\|\|K_3\|<(1-\|K_1\|)(1-\|K_4\|),$$
then the operator $I+K$ is invertible on $X\times Y$.
\end{lem}

By Lemmas \ref{LP03} and \ref{inver}, we have the
following lemma about the  spectrum structure
of the operator $\tilde\AA_{\eps}(\xi)$ for $\eps|\xi|$ being large.

\begin{lem}
\label{LP01}Fixed $\eps\in (0,1)$.
The following statements hold.
 \begin{enumerate}
\item[\rm (1)]  For any $\delta>0$, there
exists $ R_1= R_1(\delta)>0$  such that for $\eps|\xi|>R_1$,
 \bq
 \sigma(\tilde{\AA}_{\eps}(\xi))\cap\{\lambda\in\mathbb{C}\,|\,\mathrm{Re}\lambda\ge-\frac{\nu_0}2\}
 \subset
\sum_{j=\pm1}\{\lambda\in\mathbb{C}\,|\,|\lambda-\eps^2ji|\xi||\le\eps^2\delta\}.\label{sg4}
 \eq
\item[\rm (2)] For any $r_1>r_0>0$, there
exists $\alpha =\alpha(r_0,r_1)>0$ such that for  $r_0\le \eps|\xi|\le r_1$,
\bq \sigma(\tilde{\AA}_{\eps}(\xi))\subset\{\lambda\in\mathbb{C}\,|\, \mathrm{Re}\lambda<-\alpha\} .\label{sg3}\eq
 \end{enumerate}
\end{lem}

\begin{proof}
We prove \eqref{sg4} first.
By Lemma \ref{LP03}, \eqref{X_2}, \eqref{b_1(xi)} and \eqref{eee}, there exists $R_1=R_1(\delta)>0$ such that for $\mathrm{Re}\lambda\ge-\nu_0/2$, $ |\lambda-\eps^2ji|\xi||>\eps^2\delta$ and $\eps|\xi|>R_1$,
$$
\|X^1_{\eps}(\lambda,\xi)\|_{\xi}\leq 1/2,\quad \|X^2_{\eps}(\lambda,\xi)\|\le \eps^{-1}\delta^{-1}, \quad \| X^3_{\eps}(\lambda,\xi)\| \leq \eps\delta/4.
$$
This and Lemma \ref{inver} imply that the operator  $I-G^2_{\eps}(\xi)(\lambda-G^1_{\eps}(\xi))^{-1}$ is invertible on
$L^{2}_{\xi}(\R^3_v)\times \mathbb{C}^3_{\xi}\times \mathbb{C}^3_{\xi}$ and thus  $\lambda-\tilde{\AA}_{\eps}(\xi)$ is invertible on $L^{2}_{\xi}(\R^3_v)\times \mathbb{C}^3_{\xi}\times \mathbb{C}^3_{\xi}$ and
satisfies
$$
 (\lambda-\tilde{\AA}_{\eps}(\xi))^{-1}
 =\(\lambda-G^1_{\eps}(\xi)\)^{-1}\( I-G^2_{\eps}(\xi)(\lambda-G^1_{\eps}(\xi))^{-1}\)^{-1}.
$$
Therefore, it holds that for $ \eps|\xi|> R_1$,
\be \rho(\tilde{\AA}_{\eps}(\xi))\supset\{\lambda\in\mathbb{C}\,|\,\min_{j=\pm1}{|\lambda-\eps^2ji|\xi||}>\eps^2\delta,\, \mathrm{Re}\lambda\ge-\frac{\nu_0}2\}, \label{rb2}
\ee
which implies \eqref{sg4}.

Next, we turn to prove \eqref{sg3}.
By Lemma \ref{LP03}, \eqref{b_1(xi)}  and \eqref{eee}, there exists
$y_1=y_1(r_0,r_1)>0$ large enough such that for $\mathrm{Re}\lambda\geq -\nu_0/2$, $|\mathrm{Im}\lambda|>y_1 $ and $r_0\le \eps|\xi|\le r_1$,
$$
\|X^1_{\eps}(\lambda,\xi)\|_{\xi}\leq 1/6,\quad \|X^2_{\eps}(\lambda,\xi)\|\le 1/6, \quad \| X^3_{\eps}(\lambda,\xi)\| \leq 1/6.
$$
This implies that the operator $I-G^2_{\eps}(\xi)(\lambda-G^1_{\eps}(\xi))^{-1}$
is invertible on $L^2_{\xi}(\R^3_v)\times \mathbb{C}^3_{\xi}\times \mathbb{C}^3_{\xi}$, which together with
\eqref{B_d} yield that  $\lambda-\tilde{\AA}_{\eps}(\xi)$ is also invertible on $L^2_{\xi}(\R^3_v)\times \mathbb{C}^3_{\xi}\times \mathbb{C}^3_{\xi}$ when
$\mathrm{Re}\lambda\geq -\nu_0/2$,  $|\mathrm{Im}\lambda|>y_1 $ and $r_0\le \eps|\xi|\le r_1$.
Hence,  for $r_0\le \eps|\xi|\le r_1$ we have
\be
 \sigma(\tilde{\AA}_{\eps}(\xi))
 \cap\{\lambda\in\mathbb{C}\,|\,\mathrm{Re}\lambda\ge-\frac{\nu_0}2\}
\subset
 \{\lambda\in\mathbb{C}\,|\,\mathrm{Re}\lambda\ge
    -\frac{\nu_0}2,\,|\mathrm{Im}\lambda|\le y_1 \}.   \label{SpH}
 \ee

By \eqref{SpH}, it is sufficient to prove \eqref{sg3} holds for $|{\rm Im}\lambda|\le y_1$.
We prove this by contradiction.
 If it does not
hold, then there exists a sequence of
$\{(\xi_n,\lambda_n,U_n)\}$ satisfying $\eps|\xi_n|\in[r_0,r_1]$, $U_n=(f_n,E_n,B_n)\in
L^2_{\xi_n}(\R^3)\times \mathbb{C}^3_{\xi_n}\times \mathbb{C}^3_{\xi_n}$ with $\|U_n\|_{\xi_n} =1$, and $\lambda_nU_n=
 \tilde{\AA}_{\eps}(\xi_n)U_n$ with $|{\rm Im}\lambda_n|\le y_1$ and ${\rm Re}\lambda_n\to0$ as $n\to\infty$. That is,
 $$
 \left\{\bal
 \lambda_nf_n=(L_1-i\eps v\cdot\xi_n-i\eps\frac{v\cdot\xi_n}{|\xi_n|^2}P_{ d} )f_n-\eps v\chi_0\cdot( \omega_n\times E_n),\\
 \lambda_n E_n=-\eps\omega_n\times  (f_n,v\chi_0)+i\eps^2\xi_n\times B_n,\\
 \lambda_n B_n=-i\eps^2 \xi_n\times E_n.
\ea\right.
$$
Rewrite the first  equation as
 $$
 (\lambda_n+\nu+i\eps v\cdot\xi_n)f_n=K_1f_n-i\eps\frac{v\cdot\xi_n}{|\xi_n|^2}P_{ d} f_n- \eps v\chi_0\cdot( \omega_n\times E_n).
 $$
Since $K_1$ is a compact operator  on $L^2(\R^3)$, there exists a
subsequence $\{f_{n_j}\}$ of $\{f_n\}$ and $g_1\in L^2(\R^3)$ such that
$$
K_1f_{n_j}\rightarrow g_1 \quad \mbox{as}\quad j\to\infty.
$$
By using  the fact that  $\eps|\xi_n|\in[r_0,r_1]$ and  $P_{ d} f_n=C_0^n\sqrt{M}$ with
  $|C_0^n|^2+|E_n|^2+|B_n|^2\le 1$, there exists a subsequence of (still denoted by) $\{( \xi_{n_j},f_{n_j},E_{n_j},B_{n_j})\}$, and $( \xi_0,C_0,E_0, B_0)$ with $\eps|\xi_0|\in[r_0,r_1]$ and $|C_0|^2+|E_0|^2+|B_0|^2\leq1$
such that $(\xi_{n_j},C^{n_j}_0,E_{n_j},B_{n_j})\to (\xi_0,C_0,E_0,B_0)$ as $j\to\infty$. In particular,
$$
 \frac{v\cdot\xi_{n_j}}{|\xi_{n_j}|^{2}}P_{ d} f_{n_j}
 \rightarrow \frac{v\cdot\xi_0}{|\xi_0|^{2}}C_0\sqrt{M}=:g_2, \quad \frac{ \xi_{n_j}}{|\xi_{n_j}|}\times E_{n_j} \to \frac{ \xi_0}{|\xi_0|}\times E_0=:Y_0 \ \ \mbox{as} \ \ j\to\infty.
$$
Since $|\text{Im}\lambda_n|\leq y_1$ and ${\rm Re}\lambda_n\to 0$,
we can extract a subsequence of (still denoted by) $\{\lambda_{n_j}\}$
such that $\lambda_{n_j}\rightarrow \lambda_0$ with ${\rm
Re}\lambda_0=0$. Then
$$
 \lim_{j\rightarrow\infty}f_{n_j}
 =\lim_{j\rightarrow\infty}\frac{g_1-\eps g_2-\eps(v\cdot Y_0)\chi_0}{\lambda_{n_j}+\nu+i\eps(v\cdot\xi_{n_j})}
 =\frac{g_1-\eps g_2-\eps(v\cdot Y_0)\chi_0}{\lambda_0+\nu+i\eps(v\cdot\xi_0)}=:f_0 \ \ {\rm in} \ \ L^2.
$$
It follows that $\tilde{\AA}_{\eps}(\xi_0) U_0=\lambda_0 U_0$ with $U_0=(f_0,E_0,B_0)\in
L^2_{\xi_0}(\R^3)\times \mathbb{C}^3_{\xi_0}\times \mathbb{C}^3_{\xi_0}$ and $\lambda_0$ is an eigenvalue of $\tilde{\AA}_{\eps}(\xi_0)$ with ${\rm Re}\lambda_0=0$. This
 contradicts to the fact that  ${\rm Re} \lambda <0$
for $\eps \ne 0$ as stated in  Lemma~\ref{Egn}. Thus, the proof the lemma is completed.
\end{proof}

We now investigate the spectrum and resolvent sets of $\tilde{\AA}_{\eps}(\xi)$ for $\eps(1+|\xi|)$ small. Based on macro-micro decomposition,  we can split $\tilde\BB_{\eps}(\xi)$ into
\be\label{Qxi}
\left\{\bal
 \tilde\BB_{\eps}(\xi)=Q_{\eps}(\xi)+\eps \mathbb{B}_3(\xi),
\\
Q_{\eps}(\xi)=L_1-i\eps P_r(v\cdot\xi)P_r,\\
\mathbb{B}_3(\xi)=i P_d(v\cdot\xi)P_r+i (1+\frac1{|\xi|^2})P_r(v\cdot\xi)P_d.
\ea\right.
\ee
Thus, we can decompose  $\lambda-\tilde{\AA}_{\eps}(\xi)$ into
\bq
\lambda-\tilde{\AA}_{\eps}(\xi)=\lambda-G^3_{\eps}(\xi)-G^4_{\eps}(\xi),
\label{Bd3}\eq
where \bma
G^3_{\eps}(\xi)&=\left(\ba Q_{\eps}(\xi) & 0 & 0\\ 0 & 0 & i\eps^2 \xi \times \\ 0 & -i\eps^2\xi\times & 0 \ea\right),\\
G^4_{\eps}(\xi)&=\left(\ba \eps \mathbb{B}_3(\xi) & -\eps v\chi_0\cdot \omega\times & 0\\ -\eps\omega\times P_m & 0 & 0 \\ 0 & 0 & 0 \ea\right).
\ema

\begin{lem}\label{LP}
Let $\xi\neq0$ and $Q_{\eps}(\xi)$ defined by \eqref{Qxi}. We have
 \begin{enumerate}
\item[\rm (1)] If $\lambda\ne0$, then
\be
\bigg\|\lambda^{-1}\(1+\frac1{|\xi|^2}\)P_r(v\cdot\xi)P_{d}\bigg\|_{\xi}\le C(|\xi|+1)|\lambda|^{-1}.\label{S_2}
\ee

\item[\rm (2)] If $\mathrm{Re}\lambda>-\mu $, then the operator $\lambda P_{r}-Q_{\eps}(\xi)$ is invertible on $N_1^\bot$ and satisfies
\bma
\|(\lambda P_r-Q_{\eps}(\xi))^{-1}\|&\leq(\mathrm{Re}\lambda+\mu )^{-1},\label{S_3}\\
\|P_{d}(v\cdot\xi) (\lambda P_r-Q_{\eps}(\xi))^{-1}P_{r}\|_{\xi}
&\leq C(\mathrm{Re}\lambda+\mu )^{-1} (1+|\xi|)[1+ (1+\eps|\xi|)^{-1} |\lambda| ]^{-1} ,\label{S_5}\\
\|P_m(\lambda P_r-Q_{\eps}(\xi))^{-1}P_r\|&\leq C(\mathrm{Re}\lambda+\mu )^{-1} [1+(1+\eps|\xi|)^{-1}|\lambda|]^{-1}.\label{S_5a}
\ema
\end{enumerate}
\end{lem}

\begin{proof}
The estimates \eqref{S_2} and \eqref{S_3} are proved in Lemma 3.5 of \cite{Li3}.
By \eqref{S_3} and the fact that $\|P_{d}(v\cdot\xi) P_{r}f\|_{\xi}\le C(|\xi|+1)\|P_{r}f\|$, we have
 \be
 \| P_{d}(v\cdot\xi) (\lambda  P_{r}-Q_{\eps}(\xi))^{-1} P_{r}f\|_{\xi} \leq C(|\xi|+1)(\mathrm{Re}\lambda+\mu )^{-1}\|P_rf\|.   \label{2.33a}
 \ee
We now  decompose the operator $ P_{d}(v\cdot\xi) (\lambda
 P_{r}-Q_{\eps}(\xi))^{-1} P_{r}$ as
 $$
  P_{d}(v\cdot\xi) (\lambda  P_{r}-Q_{\eps}(\xi))^{-1} P_{r}=\frac1\lambda  P_{d}(v\cdot\xi) P_{r}+\frac1\lambda  P_{d}(v\cdot\xi) Q_{\eps}(\xi)(\lambda
 P_{r}-Q_{\eps}(\xi))^{-1} P_{r}.
 $$
This together with \eqref{S_3} and the fact that
 $ \| P_{d}(v\cdot\xi)  Q_{\eps}(\xi)f\|_{\xi}\leq C(|\xi|+1)(1+\eps|\xi|) \|P_{r}f\| $
give
 \be
 \| P_{d}(v\cdot\xi) (\lambda  P_{r}-Q_{\eps}(\xi))^{-1} P_{r}f\|_{\xi}
 \leq
C (|\xi|+1)(1+\eps|\xi|)|\lambda|^{-1}[(\mathrm{Re}\lambda+\mu )^{-1}+1] \|f\|. \label{2.33}
 \ee
The combination of the two cases \eqref{2.33a} and \eqref{2.33} yields \eqref{S_5}.
\eqref{S_5a} can be proved similiarly. This completes the proof of the lemma.
\end{proof}

\begin{lem}\label{spectrum2}For fixed $\eps\in (0,1)$, the following  holds.
 \begin{enumerate}
\item[\rm (1)]  For any  $\delta>0$, there are two constants
 $r_1=r_1(\delta),\,y_1=y_1(\delta)>0$ such that for all $|\xi|\ne 0$,
\bq \label{rb1}
 \rho(\tilde{\AA}_{\eps}(\xi))\supset
 \left\{\bln
 &\{\lambda\in\mathbb{C}\,|\,
     \mathrm{Re}\lambda\ge-\frac{\nu_0}{2},\, |\lambda\pm\eps^2 i|\xi||\ge \eps^2\delta\}
 \cup \C_+, \quad  \eps|\xi|\ge r_1; \\
 &\{\lambda\in\mathbb{C}\,|\,
     \mathrm{Re}\lambda\ge-\frac{\mu}{2},\,|\mathrm{Im}\lambda|\geq y_1\}
 \cup\C_+, \qquad\quad~~ \, \eps|\xi|\le r_1,
 \eln\right.
\eq where $\C_+=\{\lambda\in\mathbb{C}\,|\,\mathrm{Re}\lambda>0\}$.

\item[\rm (2)]  For any $\delta>0$, there exists $r_0=r_0(\delta)>0$ such that for $\eps(1+|\xi|)\leq r_0$,
 \bq
 \sigma(\tilde{\AA}_{\eps}(\xi))\cap\{\lambda\in\mathbb{C}\,|\,\mathrm{Re}\lambda\ge-\frac{\mu}{2}\}
 \subset
 \{\lambda\in\mathbb{C}\,|\,|\lambda|\le\delta\}.   \label{sg4a}
 \eq
\end{enumerate}
\end{lem}

\begin{proof}
By \eqref{rb2}, there exists $r_1=r_1(\delta)>0$ such that the first part of \eqref{rb1} holds. Thus, we only need to prove the second part
of \eqref{rb1}.
By Lemma \ref{LP}, we have for $\rm{Re}\lambda>-\mu $ and
$\lambda\neq0$  that the operator
$\lambda-Q_{\eps}(\xi)=\lambda P_{ d}+\lambda P_r-Q_{\eps}(\xi)$ is invertible on
$L^2_{\xi}(\R^3_v)$ and it satisfies
$$
  (\lambda -Q_{\eps}(\xi))^{-1} =\lambda^{-1}P_{ d}+(\lambda P_r-Q_{\eps}(\xi))^{-1}P_r.
$$
Here, we have used the fact that the operator $\lambda P_{ d}$ is orthogonal to $\lambda
P_r-Q_{\eps}(\xi)$. Thus, for ${\rm Re}\lambda \ge -\mu/2$ and $\lambda\ne0,\pm\eps^2i|\xi|$, the operator $\lambda-G^3_{\eps}(\xi)$ is invertible on
$L^2_{\xi}(\R^3_v)\times \mathbb{C}^3_{\xi}\times \mathbb{C}^3_{\xi}$ and satisfies
$$
(\lambda-G^3_{\eps}(\xi))^{-1}= \left(\ba \lambda^{-1} P_{ d}+(\lambda P_r-Q_{\eps}(\xi))^{-1}P_r & 0 \\  0 & (\lambda -\eps^2\mathbb{B}_2(\xi))^{-1} \ea\right)_{7\times7}.
$$
Therefore, we can rewrite \eqref{Bd3} as
$$
\lambda-\tilde{\AA}_{\eps}(\xi)=\(I-G^4_{\eps}(\xi)(\lambda-G^3_{\eps}(\xi))^{-1}\)\(\lambda-G^3_{\eps}(\xi)\),
$$
where
\be \label{Y_1}
\left\{\bln
&G^4_{\eps}(\xi)(\lambda-G^3_{\eps}(\xi))^{-1}= \left(\ba X^4_{\eps}(\lambda,\xi) & X^2_{\eps}(\lambda,\xi) \\  X^5_{\eps}(\lambda,\xi) & 0 \ea\right)_{7\times7},
\\
&X^4_{\eps}(\lambda,\xi)=i\eps P_{ d}(v\cdot\xi)(\lambda P_r-Q_{\eps}(\xi))^{-1}P_r+i\eps \lambda^{-1}\(1+\frac1{|\xi|^2}\)P_r(v\cdot\xi)P_{ d},
\\
 &X^5_{\eps}(\lambda,\xi)=\left(\ba -\eps\omega\times P_m(\lambda P_r-Q_{\eps}(\xi))^{-1}P_r \\ 0_{3\times1} \ea\right)_{6\times1}.
 \eln\right.
\ee
For  $\eps |\xi| \leq r_1$, by \eqref{X_2}, \eqref{Y_1} and \eqref{S_2}--\eqref{S_5a} we can choose $y_1=y_1(r_1)>0$ such that it holds
for $\mathrm{Re}\lambda\ge-\mu/2$ and $|\mathrm{Im}\lambda|\geq y_1$ that
\be
\|X^4_{\eps}(\lambda,\xi)\|_{\xi}+\|X^2_{\eps}(\lambda,\xi)\| +\|X^5_{\eps}(\lambda,\xi)\| \leq 1/2. \label{bound_1}
\ee
This implies that the operator $I-G^4_{\eps}(\xi)(\lambda-G^3_{\eps}(\xi))^{-1}$ is invertible on
$L^{2}_{\xi}(\R^3_v)\times \mathbb{C}^3_{\xi}\times \mathbb{C}^3_{\xi}$ and thus  $\lambda-\tilde{\AA}_{\eps}(\xi)$ is invertible on $L^{2}_{\xi}(\R^3_v)\times \mathbb{C}^3_{\xi}\times \mathbb{C}^3_{\xi}$ and it satisfies
$$
 (\lambda-\tilde{\AA}_{\eps}(\xi))^{-1}
 =\(\lambda-G^3_{\eps}(\xi)\)^{-1}\(I-G^4_{\eps}(\xi)(\lambda-G^3_{\eps}(\xi))^{-1}\)^{-1}.
$$
Therefore, $\rho(\tilde{\AA}_{\eps}(\xi))\supset \{\lambda\in\mathbb{C}\,|\,{\rm
Re}\lambda\ge-\mu/2, |{\rm Im}\lambda|\ge y_1\}$ for $\eps|\xi|\le
r_1$. This and Lemma \ref{SG_1} prove \eqref{rb1}.

Assume that $ |\lambda|>\delta$ and $\mathrm{Re}\lambda\ge-\mu/2$. Then, by \eqref{S_2}--\eqref{S_5a} we can choose $r_0=r_0(\delta)>0$ small enough so that \eqref{bound_1} still holds for $\eps(1+|\xi|)\leq r_0$.
This
implies that the operator
$\lambda-\tilde{\mathbb{A}}_{\eps}(\xi)$ is invertible on $L^{2}_{\xi}(\R^3)\times \mathbb{C}^3_{\xi}\times \mathbb{C}^3_{\xi}$.
Therefore, we have
 $\rho(\tilde{\AA}_{\eps}(\xi))\supset\{\lambda\in\mathbb{C}\,|\, |\lambda|>\delta,\mathrm{Re}\lambda\ge-\mu/2\}$
for $\eps(1+|\xi|)\leq r_0$, which gives \eqref{sg4a}. And this completes the proof of the lemma.
\end{proof}

\subsubsection{Eigenvalues in $\eps (1+|\xi|)\le r_0$}

Now we prove the existence and establish  the asymptotic expansions of the eigenvalues of $\tilde{\AA}_{\eps}(\xi)$ for $\eps(1+|\xi|)$ being small.
 In terms of \eqref{Axi}, the eigenvalue problem $\tilde{\AA}_{\eps}(\xi)U=\lambda  U$ for $U=(f,X,Y)\in L^{2}_{\xi}(\R^3_v)\times \mathbb{C}^3_{\xi}\times \mathbb{C}^3_{\xi}$
can be written as
 \bma
  \lambda f  &=\(L_1-i\eps v\cdot\xi  -i\eps\frac{ v\cdot\xi}{|\xi|^2}P_{d}\)f-\eps v\chi_0\cdot(\omega\times X),\label{L_2}\\
  \lambda X&=-\eps\omega\times (f,v\chi_0)+i\eps^2\xi\times Y,\label{L_2a}\\
  \lambda Y&=-i\eps^2\xi\times X,\quad |\xi|\ne0. \nnm
 \ema

We rewrite $f$ in the
form of $f=f_0+f_1$, where $f_0=P_{ d}f=C_0\sqrt M$ and $f_1=(I-P_{ d})f=P_rf$.
Then  \eqref{L_2} gives
 \bma
 &\lambda f_0=- i\eps P_{d}(v\cdot\xi)(f_0+f_1),\label{A_2}
\\
&\lambda f_1=L_1f_1- i\eps P_r(v\cdot\xi)(f_0+f_1)-i\eps\frac{v\cdot\xi}{|\xi|^2}f_0-\eps v\chi_0\cdot(\omega\times X).\label{A_3}
 \ema
By Lemma \ref{LP} and \eqref{A_3}, the microscopic part $f_1$ can be represented  by
 \bq
 f_1=i\eps R(\lambda,\eps \xi)(v\cdot\xi)\(1+\frac{1}{|\xi|^2}\)f_0+\eps R(\lambda,\eps \xi)v\chi_0\cdot(\omega\times X),  \quad   \text{Re}\lambda>-\mu, \label{A_4}
 \eq
 where
$$
R(\lambda,\xi)=(L_1-\lambda -i P_r(v\cdot\xi))^{-1}.
$$

Substituting \eqref{A_4} into \eqref{A_2} and \eqref{L_2a}, we obtain the eigenvalue problem  for  $(\lambda,C_0,X,Y)$ as
\bma
 \lambda C_0=&\eps^2(1+|\xi|^{-2})(R(\lambda,\eps\xi)(v\cdot\xi)\chi_0,(v\cdot\xi)\chi_0)C_0\nnm\\
 &-i\eps^2(R(\lambda,\eps\xi)v\chi_0\cdot(\omega\times X),(v\cdot\xi)\chi_0),  \label{A_6}
 \\
  \lambda X=&-i\eps^2\omega\times (1+ |\xi|^{-2} )(R(\lambda,\eps\xi)(v\cdot\xi)\chi_0,v\chi_0)C_0\nnm\\
  &-\eps^2\omega\times (R(\lambda,\eps\xi)v\chi_0\cdot(\omega\times X),v\chi_0)+i\eps^2\xi\times Y, \label{A_7}
  \\
  \lambda Y=&-i\eps^2\xi\times X. \label{A_8}
 \ema

Let $\mathbb{O}$ be a rotation in $\R^3$ satisfying $\mathbb{O}^T \xi=(|\xi|,0,0).$ By changing variable $v\to \mathbb{O} v$ and using the rotational invariance of the operator $L_1$, we have the following transformation:
\be
(R(\lambda,\xi)\chi_i,\chi_j)= \omega_i\omega_j(R(\lambda, se_1)\chi_1,\chi_1)+(\delta_{ij}-\omega_i\omega_j)(R(\lambda, se_1)\chi_2,\chi_2),\label{B_1a}
\ee where  $e_1=(1,0,0)$, $\xi=s\omega$ with $s=|\xi|,\, \omega\in \S^2$.

Substituting  \eqref{B_1a} into  \eqref{A_6} and \eqref{A_7}, we obtain
 \bma
 \lambda C_0=&\eps^2(1+s^2)(R(\lambda,\eps se_1)\chi_1,\chi_1)C_0,\label{A_9}
 \\
 \lambda X=&\eps^2(R(\lambda,\eps se_1)\chi_2,\chi_2)X+i\eps^2\xi\times Y.   \label{A_10}
  \ema
Multiplying \eqref{A_10} by $\lambda$ and using \eqref{A_8} and \eqref{rotat}, we obtain
$$ (\lambda^2-\eps^2(R(\lambda,\eps se_1)\chi_2,\chi_2)\lambda+\eps^4s^2)X=0.$$

Denote
 \bma
 D_0(z,s,\eps)&=:z- (1+s^2)(R(\eps^2z,\eps s)\chi_1,\chi_1),\label{D0}\\
 D_1(z,s,\eps)&=:z^2- (R(\eps^2z,\eps s)\chi_2,\chi_2)z+ s^2.   \label{D1}
 \ema

The eigenvalues $\lambda=\eps^2z$ can be obtained by solving $D_0(z,s,\eps)=0$ and $D_1(z,s,\eps)=0$.
The following two lemmas are about the solutions to the equations
$D_0(z,s,\eps)=0$ and $D_1(z,s,\eps)=0$ respectively. 

\begin{lem}\label{eigen} There are  two constants $r_0,r_1>0$ such that the equation $D_0(z,s,\eps)=0$ has a unique solution $z=z_0(s,\eps)$: $I\to J_0$ for  $I=\{(s,\eps)\in \R^2\,|\,\eps(1+|s|)\le r_0\}$ and $J_0=\{z\in\C\,|\,|z+\eta(1+s^2)|\le r_1(1+s^2)\} $, which is a $C^{\infty}$ function of $s$, $\eps$ and satisfies
\be
z_0(s,0)=-\eta(1+s^2) ,\quad \pt_\eps z_0(s,0)=0, \label{z1}
\ee
where $\eta>0$ is a constant given by \eqref{coe}.
In particular, $z_0(s,\eps)$ satisfies the following expansion for $ \eps(1+|s|) \le  r_0 $:
\be
z_0(s,\eps)=-\eta(1+s^2)+O(\eps^2(1+s^2)^2). \label{z2}
\ee
\end{lem}
\begin{proof}
By \eqref{D0}, the equation
\be D_0(z,s,0)=z+\eta(1+s^2)=0  \label{a}\ee
has a unique solution $b_0(s)=-\eta(1+s^2)$.

For any fixed $s$ and $\eps$, we define
$$D(z,s,\eps)=(1+s^2)R_{11}(\eps^2z,\eps s),$$
where $R_{11}(x,y)= ((L_1-x-iyP_rv_1)^{-1}\chi_1,\chi_1)$.
 It is straightforward to check that for any fixed $s$ and $\eps$, a solution of $D_0(z,s,\eps)=0$ is a fixed point of $D(z,s,\eps)$.

Since $R_{11}(x,y)$ is smooth for $(x,y)\in \mathbb{C}\times\R $ and satisfies
\be \label{a1b}
\left\{\bln
\partial_1 R_{11}(0,0)&=(L^{-2}_1\chi_1,\chi_1)>0 ,\\
 \partial_2 R_{11}(0,0)&=i(v_1L^{-1}_1\chi_1,L^{-1}\chi_1)=0 ,
 \eln\right.
\ee
it follows that
\bmas
|D(z,s,\eps)-b_0(s)|&= (1+s^2) \Big|R_{11}(\eps^2z,\eps s )-R_{11}(0,0 )\Big|\nnm\\
&\le C(1+s^2)(|\eps^2z|+\eps^2 s^2)\le r_1(1+s^2),
\\
|D(z_1,s,\eps)-D(z_2,s,\eps)|&\le C\eps^2(1+s^2)|z_1-z_2|\le \frac12|z_1-z_2|,
\emas
for $|z-b_0(s)|\le r_1(1+s^2)$ and $\eps(1+|s|)\le r_0$ with $r_0,r_1>0$ being sufficiently small.

Hence, by the contraction mapping theorem, there exists a unique fixed point $z_0(s,\eps):  I\to J_0$ such that $D(z_0(s,\eps),s, \eps)=z_0(s,\eps)$ for $(s,\eps)\in I$ and $z_0(s,0)=b_0(s)$. This is equivalent to  $D_0(z(s,\eps),s,\eps)=0$. Since $D_0(z,s,\eps)$ is $C^{\infty}$ with respect to $z\in J_0$ and $(s,\eps)\in I$, it follows that $z_0(s,\eps)$ is a $C^{\infty}$ function with respect to $(s,\eps)\in I$.
In particular, we obtain
$$
\pt_z D_0(z,s,0)=1,\quad \pt_\eps D_0(z,s,0)=-(1+s^2)s\partial_2 R_{11}(0,0)=0,
$$
which gives
\be
 \pt_\eps z_0(s,0)=- \frac{{\partial_\eps}D_0(b_0(s),s,0)}{{\partial_z}D_0(b_0(s),s,0)} =0. \label{b}
\ee
Combining \eqref{a} and \eqref{b} yields \eqref{z1}.

For \eqref{z2}, by \eqref{a1b}, we can obtain that for $|z-b_0(s)|\le r_1(1+s^2)$ and $\eps(1+|s|)\le r_0$,
$$
R_{11}(\eps^2z,\eps s)=-\eta+ O(1)(\eps^2|z|+\eps^2 s^2).
$$ Thus
\bmas
z_0(s,\eps)&=(1+s^2)R_{11}(\eps^2z_0(s,\eps),\eps s)\\
&=-\eta(1+s^2) +O(1)[\eps^2(|z_0(s,\eps)|+ s^2)(1+s^2)]\\
&=-\eta(1+s^2)+O(\eps^2(1+s^2)^2), \quad \eps(1+|s|)\le r_0.
\emas
Hence, the proof of  the lemma is completed.
\end{proof}

\begin{lem}\label{eigen_2a}
There are small constants $r_0,r_1>0$ such that the equation $D_1(z,s,\eps)=0$  has two continuous solutions $z_j=z_j(s,\eps): I\to J_1$, $j=\pm1$ for $I=\{(s,\eps)\in \R^2\,|\, \eps(1+|s|)\le r_0\}$ and $ J_1=\{z\in \C\,|\,  |z-b_j(s)|\le r_1|b_j(s)| \}$. In particular,   $z_j(s,\eps)$ are $C^{\infty}$ functions in $s, \eps$ for $ \eps(1+|s|) \le  r_0$ and $|\eta^2-4s^2|\ge r_0$, which satisfy
\be
z_j(s,0)=b_j(s),\quad \pt_\eps z_j(s,0)=0, \quad j=\pm 1,\label{z1a}
\ee
where
$$b_j(s)=-\frac{\eta}2+\frac{j\sqrt{\eta^2-4s^2}}{2}.$$
Moreover, $z_j(s,\eps)$ satisfies the following expansion for  $ \eps(1+|s|) \le  r_0 $:
\be\label{z2a}
z_j(s,\eps)=b_j(s)+\left\{\bal
O( \eps^2|b_j(s)|), & |\eta^2-4s^2|\ge r_0,\\
O(\eps),& |\eta^2-4s^2|< r_0.
\ea\right.
\ee
In addition, there exists a continuous real function $\vartheta_0(\eps): (-r_0,r_0)\to B(\eta/2,r_1)$ such that $z_1(s,\eps)=z_{-1}(s,\eps)$ if and only if $(s,\eps)=(\vartheta_0(\eps),\eps)$ and $\vartheta_0(0)=\eta/2$.
\end{lem}

\begin{proof}
By \eqref{D1} and noting that $\eta=-(L^{-1}_1\chi_2,\chi_2)$, the equation
$$ D_1(z,s,0)=z^2+\eta z+s^2=0
$$
 have two solutions $b_j(s)=-\eta/2+j \sqrt{\eta^2-4s^2}/2$ for $j=\pm1$.

For any fixed $s$ and $\eps$, we define
\bq G_j(z,s,\eps)=\frac12\Big( R_{22}(\eps^2z,\eps s)+j\sqrt{R_{22}(\eps^2z,\eps s)^2-4s^2}\Big),\quad j=-1,1,\label{fp1}\eq
where $ R_{22}(x,y)=((L_1-x-iyP_rv_1)^{-1}\chi_2,\chi_2). $
It is straightforward to check that for any fixed $s$ and $\eps$, a solution of $D_1(z,s,\eps)=0$ is a fixed point of $G_j(z,s,\eps)$.
We consider the existence of the solutions of $D_1(z,s,\eps)=0$ for two case:  $(s,\eps)\in I_1$ and $(s,\eps)\in I_2$ with
     \bmas
I_1&=\big\{(s,\eps)\,|\,\eps(1+|s|) \le r_0, |\eta^2-4s^2|\ge r_0\big\}, \\
I_2&=\big\{(s,\eps)\,|\,\eps(1+|s|) \le r_0, |\eta^2-4s^2|\le r_0 \big\}.
\emas

First, we will study the existence of the solutions of $D_1(z,s,\eps)=0$ for $(s,\eps)\in I_1$.
Since $R_{22}(x,y)$ is smooth for $(x,y)\in \mathbb{C}\times\R $ and satisfies
\be \label{a2}
\left\{\bln
\pt_1 R_{22}(0,0)&=(L^{-2}_1\chi_2,\chi_2)>0,\\
 \pt_2 R_{22}(0,0)&=i(  v_1L^{-1}_1 \chi_2,L^{-1}_1\chi_2)=0,
 \eln\right.
\ee
it follows that for $\eps^2|z|\le r_0$ and $\eps|s|\le r_0$ with $r_0\ll1$,
\be R_{22}(\eps^2z,\eps s) =-\eta+O(1)(\eps^2|z|+ \eps^2s^2). \label{a2a}\ee
Thus, it holds that for $|z -b_j(s)|\le r_1|b_j(s)| $, $\eps(1+|s|) \le r_0 $ and $|\eta^2-4s^2|\ge r_0$ with $r_0,r_1\ll1$,
\be
|R_{22}(\eps^2z,\eps s)^2-4s^2|\ge |\eta^2-4s^2|-C\eps^2|z|-C\eps^2 |s|^2 \ge \frac12 r_0 .\label{a3}
\ee

From \eqref{fp1}--\eqref{a3}, we obtain that for $|z-b_j(s)|\le r_1|b_j(s)|$, $\eps(1+|s|) \le r_0 $ and $|\eta^2-4s^2|\ge r_0$  with $r_0,r_1\ll1$,
\bmas
|G_j(z,s,\eps)-b_j(s)|&\le \frac12\big|R_{22}(\eps^2z,\eps s)-R_{22}(0,0)\big|\\
&\quad+ \frac{ | R_{22}(\eps^2z,\eps s)^2 -R_{22}(0,0)^2 |}{ 2|\sqrt{R_{22}(\eps^2z,\eps s)^2-4s^2}|+2|\sqrt{R_{22}(0,0)^2-4s^2}|} \\
&\le C\eps^{2}(|z|+ s^{ 2}) \le r_1|b_j(s)| ,
\\
|G_j(z_1,s,\eps)-G_j(z_2,s,\eps)|&\le \frac12\big|R_{22}(\eps^2z_1,\eps s)-R_{22}(\eps^2z_2,\eps s)\big|\\
&\quad+\frac{ | R_{22}(\eps^2z,\eps s)^2 -R_{22}(\eps^2z_1,\eps s)^2 |}{ 2|\sqrt{R_{22}(\eps^2z,\eps s)^2-4s^2}|+2|\sqrt{R_{22}(\eps^2z_1,\eps s)^2-4s^2}|} \\
&\le C \eps^{2}|z_1-z_2|  \le \frac12|z_1-z_2|.
\emas
Hence by contraction mapping theorem, there exists a unique fixed point $z_j(s,\eps): I_1\to J_1$, $j=\pm1$
such that $G_j(z_j(s,\eps),s, \eps)=z_j(s,\eps)$ for $(s,\eps)\in  I_1$ and $z_j(s,0)=b_j(s)$.
This is equivalent to  $D_1(z_j(s,\eps),s,\eps)=0$. Since $G_j(z,s,\eps)$ is $C^{\infty}$ with respect to $z\in J_1$ and $(s,\eps)\in I_1$,
it follows that $z_j(s,\eps)$ is a $C^{\infty}$ function with respect to $(s,\eps)\in I_1$. In particular, it holds that
$$ \pt_z D_1(z,s,0)=2z-\eta, \quad \pt_{\eps}D_1(z,s,0)=0, $$
which leads to
$$ \pt_{\eps}z_j(s,0)=\frac{\pt_{\eps}D_1(b_j(s),s,0)}{\pt_z D_1(b_j(s),s,0)}=0, \quad 2|s|\ne \eta. $$
Thus, we obtain \eqref{z1a}. 
By \eqref{a2a},  we obtain that for $(s,\eps)\in I_1$,
\bma
2z_j(s,\eps)&= R_{22}(\eps^2z_j(s,\eps),\eps s)+j\sqrt{R_{22}(\eps^2z_j(s,\eps),\eps s)^2-4s^2} \nnm\\
&=-\eta+O(1)(\eps^2|z_j(s,\eps)|+\eps^2s^2 )+j\sqrt{\eta^2-4s^2}\nnm\\
&\quad+O(1)(\eps^{2}|z_j(s,\eps)|+ \eps^{2}s^{2} )/\sqrt{\eta^2-4s^2}\nnm\\
&=2b_j(s)+O\(\frac{\eps^{2}(|b_j(s)|+s^2)}{\sqrt{\eta^2-4s^2}}\), \quad j=-1,1, \label{zj}
\ema which gives \eqref{z2a} for $|\eta^2-4s^2|\ge r_0$.

Next, we will study the existence of the solutions of $D_1(z,s,\eps)=0$ for $(s,\eps)\in I_2$.
For any  $z,z_1,z_2\in J_1$ and  $(s,\eps)\in I_2$, we obtain
\bma
|G_j(z,s,\eps)-b_j(s)|&\le \frac12\big|R_{22}(\eps^2z,\eps s)-R_{22}(0,0)\big| \nnm\\
&\quad+ \frac12\sqrt{|R_{22}(\eps^2z,\eps s)^2-R_{22}(0,0)^2|} \nnm\\
&\le C\eps  \sqrt{|z|}+C\eps |s|  , \label{xxy}
\\
|G_j(z_1,s,\eps)-G_j(z_2,s,\eps)|&\le \frac12\big|R_{22}(\eps^2z_1,\eps s)-R_{22}(\eps^2z_2,\eps s)\big|\nnm\\
&\quad+\frac12\sqrt{|R_{22}(\eps^2z_1,\eps s)^2-R_{22}(\eps^2z_2,\eps s)^2|} \nnm\\
&\le C\eps \sqrt{|z_1-z_2|} , \label{xxy1}
\ema
where we have used the inequality $|\sqrt x-\sqrt y|\le \sqrt{|x-y|}$.
This implies that for any fixed  $(s,\eps)\in I_2$,  $G_j(z,s,\eps)$ is a continuous mapping in $z\in J_1$.
Hence by Brouwer fixed point theorem, there exists at least a fixed point $z_j(s,\eps): I_2\to J_1$, $j=\pm1$ such that $G_j(z_j(s,\eps),s, \eps)=z_j(s,\eps)$ for $(s,\eps)\in I_2$ and $z_j(s,0)=b_j(s)$. This is equivalent to  $D_1(z_j(s,\eps),s,\eps)=0$. Moreover, by \eqref{xxy} and \eqref{xxy1},   $z_j(s,\eps)$ is a continuous function with respect to $(s,\eps)\in I_2$ and satisfies
$$
2z_j(s,\eps) =2b_j(s)+ O(1)(\eps \sqrt{|z_j(s,\eps)|}+\eps |s|)=2b_j(s)+ O(\eps), \quad j=-1,1,
$$
which gives \eqref{z2a} for $|\eta^2-4s^2|\le r_0$.

We claim that there are at most two differential solutions $z_j=z_j(s,\eps)$, $j=\pm 1$ of the equation $D_1(z,s,\eps)=0$  for $(s,\eps)\in I_2$. Indeed, it is easy to check that $D_1(z,s,\eps)$ is $C^\infty$ with respect to $(z,s,\eps)\in \C^2\times (-1,1)$ and satisfies
$$  D_1(-\eta/2, \eta/2,0)= 0, \quad \pt_s D_1(-\eta/2, \eta/2,0)= \eta.$$
Thus by the implicit function theorem, there exists a unique $C^\infty$ function $s(z,\eps): (z,\eps)\in B( -\eta/2,r_0)\times (-r_0,r_0)\to B( \eta/2,r_1)$ such that $D_1(z,s(z,\eps), \eps)=0$  and $s( -\eta/2,0)=\eta/2$, where $B(x,r)$ is a ball in $\C$ given by
 $$B(x,r)=\{y\in \C\,|\, |y-x|<r\}, \quad (x,r)\in \C\times \R_+. $$
 In particular, it holds that
$$  \pt_z D_1(-\eta/2, \eta/2,0)=0,  \quad \pt_{zz} D_1(-\eta/2, \eta/2,0)=2,$$
which gives
$$ \pt_z s( -\eta/2,0)=0,\quad \pt_{zz} s(-\eta/2,0)=-2/\eta. $$
Moreover, if $s(z,\eps)$ is real, then $z$ must be a solution of $D_1(z,s, \eps)=0$.

If there are three differential solution $z_j=z_j(s,\eps)$, $j=1,2,3$ of $D_1(z_j,s, \eps)=0$ for $(s,\eps)\in I_2$,  then  there exists a point $(s,\eps)\in B(\eta/2,r_0)\times (-r_0,r_0)$ such that $z_1(s,\eps)\ne z_2(s,\eps)\ne z_3(s,\eps)\in B( -\eta/2,r_0).$ This implies that
$s(z_1,\eps)=s(z_2,\eps)=s(z_3,\eps).$  By mean value theorem, there exist $z_4,z_5\in B( -\eta/2,r_0)$ such that
$
\pt_z s(z_4,\eps)=\pt_z s(z_5,\eps) $, and then there exist $z_6\in B( -\eta/2,r_0)$ such that $\pt_{zz} s(z_6,\eps)=0$. This is a contradiction to $\pt_{zz} s(-\eta/2,0)=-2/\eta.$
Thus, the equation $D_1(z,s,\eps)=0$ admits exactly two continuous solutions $z_j(s,\eps): I_2\to J_1$ for $j=\pm 1$.

Furthermore, it is straightforward to check that $z_1(s,\eps)=z_{-1}(s,\eps)$ if and only if
$$z_{\pm1}=\frac12R_{22}(\eps^2z_{\pm1},\eps s), \quad R_{22}(\eps^2z_{\pm1},\eps s)^2-4s^2=0,$$
which is equivalent to
$$z_{\pm1}=-s, \quad R_{22}(-\eps^2s,\eps s)+2s=0, \quad s>0.$$
Denote
$D_2(s,\eps)=R_{22}(-\eps^2s,\eps s)+2s $ with $(s,\eps)\in \C\times (-1,1)$.
It holds that
$$D_2(\eta/2,0)=0,\quad \pt_s D_2(\eta/2,0)=2.$$
By the implicit function theorem,  there exist small constants $r_0,r_1>0$ such that the equation $D_2(s,\eps)=0$ has a unique solution $s=\vartheta_0(\eps): (-r_0,r_0)\to B(\eta/2,r_1)$. Since $\overline{D_2(s,\eps)}=D_2(\overline{s},\eps)$, it follows that $\vartheta_0(\eps)=\overline{\vartheta_0}(\eps)$.
Thus, the solutions $z_1(s,\eps)=z_{-1}(s,\eps)$ if and only if $s=\vartheta_0(\eps)$ and $\eps\in (-r_0,r_0)$.
This proves the lemma.
\end{proof}

With  Lemmas \ref{eigen} and \ref{eigen_2a}, we have the following lemma about
 the eigenvalue $\lambda_j(|\xi|,\eps)$ and the corresponding eigenfunction $\mathcal{U}_j(\xi,\eps)$
of the operator $\tilde{\AA}_{\eps}(\xi)$  for $\eps(1+|\xi|)\le r_0$.

\begin{lem}\label{eigen_4a}
{\rm (1)} There exists a small constant $r_0>0$ such that $ \sigma(\tilde{\AA}_{\eps}(\xi))\cap \{\lambda\in \mathbb{C}\,|\, \mathrm{Re}\lambda>-\mu /2\} $ consists of five points $\{\lambda_j(s,\eps),\ j=0,1,2,3,4\}$ for   $\eps(1+|s|)\le  r_0$ and $s=|\xi|$. The eigenvalues $\lambda_j(s,\eps)$  are $C^\infty$ functions of $s$ and $\eps $, and admit the following asymptotic expansions for $\eps(1+|s|) \le r_0 $:
 \bma\label{specr0}
 \lambda_{0}(s,\eps) &= \eps^2b_{0}(s) +O(\eps^4 (1+s^2)^2), \\
 \lambda_k(s,\eps)&=\eps^2b_k(s)+\left\{\bal
O( \eps^4|b_k(s)|), & |\eta^2-4s^2|\ge r_0,\\
O(\eps^3),& |\eta^2-4s^2|< r_0,
\ea\right. \label{specr1}
 \ema
where $k=1,2,3,4,$  $\lambda_1=\lambda_2$ and $\lambda_3=\lambda_4$,  and
 \be \label{bj}
 \left\{\bln
 &b_0=-\eta(1+s^2), \quad b_{1}=b_{2}=-\frac{\eta}2-\frac{\sqrt{\eta^2-4s^2}}2, \\
&b_{3}=b_{4}=-\frac{\eta}2+\frac{\sqrt{\eta^2-4s^2}}2.
\eln\right.
 \ee
Moreover, $\lambda_1(s,\eps)=\lambda_3(s,\eps)$ if and only if $(s,\eps)=(\vartheta_0(\eps),\eps)$ with $\vartheta_0(\eps)$ being a real continuous function given in Lemma \ref{eigen_2a}.

{\rm (2)} The eigenfunctions $\mathcal{U}_j(\xi,\eps)=(u_j(\xi,\eps),X_j(\xi,\eps),Y_j(\xi,\eps))$, $j=0,1,2,3,4$ are $C^\infty$  in $s$ and $\eps$, and satisfy for $\eps(1+|s|) \le r_0$ and $(s,\eps)\ne (\vartheta_0(\eps),\eps)$:
  \be            \label{eigf2}
  \left\{\bln
  &( \mathcal{U}_i ,\mathcal{U}^*_j )_{\xi}=:(u_i,\overline{u_j})_{\xi}-(X_i,\overline{X_j})-(Y_i,\overline{Y_j})=\delta_{ij}, \ \ 0\le i,j\le 4,  \\
  &u_0=P_du_0+P_ru_0,\quad (X_0,Y_0)\equiv(0,0),\\
&P_{ d}u_0 = \frac{s}{\sqrt{1+s^2}}\[ 1+O(\eps^2(1+s^2))\]\chi_0,\\
&P_ru_0 =  i\eps \sqrt{1+s^2} L^{-1}_1(v\cdot\omega)\chi_0+O(\eps^2(1+s^{2})),\\
&u_k=\frac{\eps \lambda_k \Theta_k}{\sqrt{\lambda_1\lambda_3-\lambda_k^2}}\[L^{-1}_1  (v\cdot  e_k)\chi_0 +O(\eps(1+|s|)) \],\\
&(X_k,Y_k) =\frac{ \lambda_k \Theta_k}{\sqrt{\lambda_1\lambda_3-\lambda_k^2}} \(\omega\times e_k, \frac{i\eps^2s e_k}{\lambda_k}\),\quad k=1,2,3,4,
  \eln\right.
  \ee
where $   \mathcal{U}^*_j=(\overline{u_j},-\overline{X_j},-\overline{Y_j})$, 
and $e_k$, $k=1,2,3,4$ are normal vectors satisfying $e_1=e_3$, $e_2=e_4$, and $e_k\cdot\omega=e_1\cdot e_2=0$, and
\be \label{Ak1}
\Theta_k=1+\left\{\bal
O( \eps^2), & |\eta^2-4s^2|\ge r_0,\\
O(\eps ),& |\eta^2-4s^2|< r_0.
\ea\right.
\ee
\end{lem}

\begin{proof}
The eigenvalue $\lambda_j(s,\eps)$ and the eigenfunction $\mathcal{U}_j(\xi,\eps)$ for $j=0,1,2,3,4$ can be constructed as follows. For $j=0,$ we take $\lambda_0=\eps^2 z_0(s,\eps)$ with $ z_0(s,\eps)$ being the solution of the equation  $D_0(z,s,\eps)=0$ given in Lemma \ref{eigen}, and choose $X=Y=0$ in \eqref{A_6}--\eqref{A_8}. And  the corresponding eigenfunction $\mathcal{U}_0(\xi,\eps)=(u_0(\xi,\eps),X_0(\xi,\eps),Y_0(\xi,\eps))$ is defined by
 \be  \label{C_2}
 \left\{\bln
  u_0(\xi,\eps)&=a_0(s,\eps)\chi_0   +i \eps\(s+\frac1s\) a_0(s,\eps)(L_1-\lambda_0-i\eps s P_r(v\cdot\omega))^{-1}(v\cdot\omega)\chi_0,\\
  X_0(\xi,\eps)&=Y_0(\xi,\eps)\equiv 0,
  \eln\right.
\ee
where $a_0(s,\eps)$ is a complex value function determined later.

For $j=1,2,3,4,$ we take $\lambda_1=\lambda_2=\eps^2 z_{-1}(s,\eps)$ and $\lambda_3=\lambda_4=\eps^2 z_{1}(s,\eps)$ with $ z_{\pm1}(s,\eps)$ being the solution of the equation  $D_1(z,s,\eps)=0$ given in Lemma \ref{eigen_2a}, and choose $C_0=0$ in \eqref{A_6}--\eqref{A_8}. Then we define the corresponding eigenfunction $\mathcal{U}_j(\xi,\eps)=(u_j(\xi,\eps),X_j(\xi,\eps),Y_j(\xi,\eps))$   by
 \be  \label{C_3}
 \left\{\bln
  u_j(\xi,\eps)&=  \eps a_j(s,\eps)(L_1-\lambda_j-i \eps s P_r(v\cdot\omega))^{-1} (v\cdot e_j)\chi_0,  \\
   X_j(\xi,\eps)&=a_j(s,\eps)\omega\times e_j, \quad  Y_j(\xi,\eps)=i\frac{\eps^2sa_j(s,\eps)}{\lambda_j(s,\eps)}e_j,
  \eln\right.
\ee
where $e_j$, $j=1,2,3,4$ are normal vectors satisfying $e_1=e_3$, $e_2=e_4$ and $e_k\cdot\omega=e_1\cdot e_2=0$,  and $a_j(s,\eps)$ is a complex value function determined later. It's easy to verify that $(\mathcal{U}_1,\mathcal{U}^*_2)_{\xi}=(\mathcal{U}_3,\mathcal{U}^*_4)_{\xi}=0, $  where $ \mathcal{U}^*_j=(\overline{u_j},-\overline{X_j},-\overline{Y_j})$ is the eigenvector of $\tilde{\AA}^*_{\eps}(\xi)$ corresponding to the eigenvalue $\overline{\lambda_{j}(s,\eps)}$.

Rewrite the eigenvalue problem as
$$\tilde{\AA}_{\eps}(\xi)\mathcal{U}_j(\xi,\eps)=\lambda_j(s,\eps)\mathcal{U}_j(\xi,\eps), \quad  j=0,1,2,3, 4.$$
Taking the inner product $(\cdot,\cdot)_{\xi}$ of the above equation with $\mathcal{U}^*_k(\xi,\eps)$ and using the facts that
 \bgrs
 (\tilde{\AA}_{\eps}(\xi)U,V)_{\xi} =(U,\tilde{\AA}^*_{\eps}(\xi)V)_{\xi},\quad U,V\in D(\tilde{\AA}_{\eps}(\xi)),
\\
 \tilde{\AA}^*_{\eps}(\xi)\mathcal{U}^*_j(\xi,\eps) =\overline{\lambda_j(s,\eps)}\mathcal{U}^*_j(\xi,\eps) ,
 \egrs
we have
$$
(\lambda_j(s,\eps)-\lambda_k(s,\eps))(\mathcal{U}_j(\xi,\eps),\mathcal{U}^*_k(\xi,\eps))_{\xi}=0,\quad 0\le j, k\le 4.
$$
For $\eps(1+|s|)\le r_0$ and $(s,\eps)\ne (\vartheta_0(\eps),\eps)$, we have $\lambda_j(s,\eps)\neq \lambda_k(s,\eps)$ for
$j, k\in\{0,1,3\}$ and $j\ne k$.  Thus,
$$
\(\mathcal{U}_j(\xi,\eps),\mathcal{U}^*_k(\xi,\eps)\)_{\xi}=0,\quad 0\leq j\neq k\leq 4.
$$
We can normalize $\mathcal{U}_j(\xi,\eps)$ by taking
$$\(\mathcal{U}_j(\xi,\eps),\mathcal{U}^*_j(s,\eps)\)_{\xi} =1,\quad 0\le j\le 4.$$

The coefficients $a_j(s,\eps) $, $0\le j\le 4$ satisfy the normalization conditions $(\mathcal{U}_j,\mathcal{U}^*_j)_{\xi}=1$ as
 \bma
 &a_0(s,\eps)^2\bigg(1+\frac1{s^2}+\eps^2\bigg(s+\frac1s\bigg)^2D_0(s,\eps)\bigg)=1, \label{C_4}\\
 &a_k(s,\eps)^2\bigg(-1+\frac{\eps^4s^2}{\lambda_k(s,\eps)^2}+\eps^2 D_{kk}(s,\eps)\bigg)=1,\quad k=1,2,3,4,\label{C_4a}
 \ema
 where
 $$
 \left.\bln
 D_0(s,\eps)&=(R(\lambda_0,\eps s e_1) \chi_1, R(\overline{\lambda_0},-\eps s e_1) \chi_1), \\
 D_{jk}(s,\eps)&=(R(\lambda_j,\eps s e_1)  \chi_2, R(\overline{\lambda_k},-\eps se_1) \chi_2), \ \ j,k=1,2,3,4.
 \eln\right.
 $$

 Substituting \eqref{specr0} into \eqref{C_4}, we obtain
 \bma
 a_0(s,\eps)&=\frac{s}{\sqrt{1+s^2}}\Big(1+\eps^2(1+s^2)D_0(s,\eps)\Big)^{-\frac12}\nnm\\
 &=\frac{s}{\sqrt{1+s^2}}\[1+O(\eps^2 (1+s^2))\]. \label{C_4b}
 \ema
Due to the fact that $\lambda_1=\lambda_2$ and $\lambda_3=\lambda_4$, we have $a_1=a_2$ and $a_3=a_4$. Note that
$$
\(\mathcal{U}_1(\xi,\eps),\mathcal{U}^*_3(s,\eps)\)_{\xi}=0 \Longrightarrow  -1+\frac{\eps^4s^2}{\lambda_1\lambda_3}+\eps^2D_{13}(s,\eps)=0.
$$
Thus, it follows that for $\lambda_1\ne \lambda_3$,
\bma
-1+\frac{\eps^4s^2}{\lambda_1(s,\eps)^2}+\eps^2 D_{11}(s,\eps)&=\frac{\eps^4s^2}{\lambda_1^2}-\frac{\eps^4s^2}{\lambda_1\lambda_3}+\eps^2(D_{11}-D_{13})\nnm\\
&= \frac{\eps^4s^2}{\lambda_1^2\lambda_3}(\lambda_3-\lambda_1)+\eps^2D_{113}(\lambda_3-\lambda_1)\nnm\\
&=\frac{\lambda_3-\lambda_1}{\lambda_1}\(\frac{\eps^4s^2}{\lambda_1\lambda_3} +\eps^2\lambda_1 D_{113}\),\label{C_5}
\ema
where $D_{ijk}(s,\eps)=(R(\lambda_i,\eps s)  \chi_2, R(\overline{\lambda_j},-\eps s) R(\overline{\lambda_k},-\eps s) \chi_2)$. Similarly,
\be
-1+\frac{\eps^4s^2}{\lambda_3(s,\eps)^2}+\eps^2 D_{33}(s,\eps)=\frac{\lambda_1-\lambda_3}{\lambda_3}\(\frac{\eps^4s^2}{\lambda_1\lambda_3} +\eps^2\lambda_3 D_{313}\). \label{C_5a}
\ee
Thus, it follows from \eqref{C_4a}, \eqref{C_5} and \eqref{C_5a} that
\be
a_k(s,\eps)= \frac{\lambda_k}{\sqrt{\lambda_1\lambda_3-\lambda^2_k}}\(\frac{\eps^4s^2}{\lambda_1\lambda_3} +\eps^2\lambda_k D_{k13}\)^{-\frac12},\quad k=1,3. \label{C_5b}
\ee
Let
\be \label{Ak}
\Theta_k=\(\frac{\eps^4s^2}{\lambda_1\lambda_3} +\eps^2\lambda_k D_{k13}\)^{-\frac12}, \quad k=1,2,3,4.
\ee
By \eqref{specr1}, it holds for $|\eta^2-4s^2|\ge r_0$  that
\bma
\Theta_k&=\(\frac{ s^2}{(b_1+O(\eps^2|b_1|))(b_3+O(\eps^2|b_3|))} +\eps^4b_k D_{k13}\)^{-\frac12} \nnm\\
&=\(1+\frac1{b_1}O(\eps^2|b_1|)+\frac1{b_3}O(\eps^2|b_3|) +O(\eps^4|b_k|)\)^{-\frac12} \nnm\\
&=1+O(\eps^2), \label{C_6}
\ema
and for $|\eta^2-4s^2|\le r_0$ it holds that
\bma
\Theta_k&=\(\frac{ s^2}{(b_1+O(\eps ))(b_3+O(\eps ))} +\eps^4b_k D_{k13}\)^{-\frac12} \nnm\\
&=\(1+\frac1{b_1}O(\eps )+\frac1{b_3}O(\eps ) +O(\eps^4|b_k|)\)^{-\frac12} \nnm\\
&=1+O(\eps). \label{C_6a}
\ema
By combining \eqref{C_2}, \eqref{C_3}, \eqref{C_4b}, \eqref{C_6} and \eqref{C_6a}, and using the fact that
\bmas
R(\lambda_j,\eps s)&=L^{-1}_1+ \lambda_j R(\lambda_j,\eps s)L^{-1}_1+ i\eps s R(\lambda_j,\eps s)(v\cdot\omega)L^{-1}_1\\
&=L^{-1}_1+O(\eps^2|b_k|+\eps|s|),
\emas
we obtain the expansion of $\mathcal{U}_j(\xi,\eps)$ given in \eqref{eigf2}. This completes the proof of the lemma.
\end{proof}

\subsubsection{Eigenvalues in $\eps |\xi|\ge r_1$}

We now turn to study the asymptotic expansions of the eigenvalues
and eigenvectors in the high frequency regime. Firstly, recalling
the eigenvalue problem
\bma
  \lambda f
  &=\tilde\BB_{\eps}(\xi)f-\eps v\chi_0\cdot(\omega\times X),\label{L_3}\\
  \lambda X&=-\eps\omega\times (f,v\chi_0)+i\eps^2\xi\times Y,\label{L_3a}\\
  \lambda Y&=-i\eps^2\xi\times X,\quad |\xi|\ne0.\label{A_13}
\ema

By Lemma \ref{LP03}, there exists a large constant $R_0>0$ such that the operator $\lambda-\tilde\BB_{\eps}(\xi)$ is invertible on $L^2_{\xi}(\R^3)$ for ${\rm Re}\lambda\ge-\nu_0/2$ and $\eps|\xi|>R_0$. Then it follows from \eqref{L_3} that
\bq f=\eps (\tilde\BB_{\eps}(\xi)-\lambda)^{-1}v\chi_0\cdot(\omega\times X),\quad \eps |\xi|>R_0.\label{L_5}\eq

By a similar argument as \eqref{B_1a}, it holds that
\bma
\((\tilde\BB_{\eps}(\xi)-\lambda)^{-1}\chi_i,\chi_j\)=&\omega_i\omega_j\((\tilde\BB_{\eps}(se_1)-\lambda)^{-1}\chi_1,\chi_1\) \nnm\\
&+(\delta_{ij}-\omega_i\omega_j)\((\tilde\BB_{\eps}(se_1)-\lambda)^{-1}\chi_2,\chi_2\), \label{B_2a}
\ema
where $e_1=(1,0,0)$, $s=|\xi|$ and $\omega=\xi/|\xi|$. Substituting \eqref{L_5} and \eqref{B_2a} into \eqref{L_3a} gives
\be
\lambda X =\eps^2((\tilde\BB_{\eps}(se_1)-\lambda)^{-1}\chi_2,\chi_2)X+i\eps^2\xi\times Y, \quad \eps s>R_0. \label{A_12}
\ee
By multiplying \eqref{A_12} by $\lambda$ and using \eqref{A_13}, we obtain
$$ \(\lambda^2-\eps^2((\tilde\BB_{\eps}(se_1)-\lambda)^{-1}\chi_2,\chi_2)\lambda+\eps^4s^2\)X=0, \quad \eps s>R_0.$$
Denote
\bq D_2(z,s,\eps)=z^2-((\tilde\BB_{\eps}(se_1)-\eps^2z)^{-1}\chi_2,\chi_2)z+s^2,\quad \eps s>R_0.\label{D3a}\eq

The eigenvalues $\lambda=\eps^2z$ can be obtained by solving $D(z,s,\eps)=0$.
Firstly, similar to Lemma \ref{LP03}, we can prove the following lemma.
\begin{lem}\label{lem1}
For  $\eps\in(0,1)$ and  any $\delta>0$, if ${\rm Re}\lambda\ge -\nu_0+\delta$, then
 \bma
\|(\lambda-D_{\eps}(\xi))^{-1}K_1\|  &\leq C\delta^{-\frac12}(1+\eps|\xi|)^{-\frac12},\label{T_2}
 \\
\|(\lambda-D_{\eps}(\xi))^{-1}\chi_j\|  &\leq C\delta^{-\frac12}(1+\eps|\xi|)^{-\frac12}, \label{T_4}
 \ema
where $j=0,1,2,3,4.$ Furthermore, there exists  a sufficiently large $R_0>0$  such that $\lambda-\tilde\BB_{\eps}(\xi)$ is invertible for ${\rm Re}\lambda\ge -\nu_0+\delta$ and $\eps| \xi|>R_0$, and
\bq \|(\lambda-\tilde\BB_{\eps}(\xi))^{-1}\chi_j\|   \leq C\delta^{-\frac12}(1+\eps|\xi|)^{-\frac12}. \label{T_5}\eq
\end{lem}
We now study the equation \eqref{D3a} as follows.

\begin{lem}\label{eigen_5}
There exists a large constant $r_1>0$ such that the equation $D_2(z,s,\eps)=0$ has two solutions $z_j(s,\eps)=ji s+y_j(s,\eps) $, $j=\pm1$ for $\eps s>r_1$, where $y_j(s,\eps)$ is a $C^{\infty}$ function in $s$ and $\eps$ for   $\eps s>r_1$  satisfying
\be
  \frac{C_1}{\eps s}\le -{\rm Re}y_j(s,\eps)\le \frac{C_2}{\eps s}, \quad |{\rm Im}y_j(s,\eps)|\le C_3\frac{\ln \eps s}{\eps s} ,\label{eigen_h1}
\ee
where $C_1,C_2,C_3>0$ are some constants.
\end{lem}

\begin{proof}
For any fixed $(s,\eps)$ satisfying $\eps s>R_0$, we define a function of $z$ as
$$ G_j(z,s,\eps)=\frac12\Big(R_{0}(z,s,\eps)+j\sqrt{R_{0}(z,s,\eps)^2-4s^2}\Big),\quad j=\pm1,\,\,\, \eps |s|>R_0,$$
where $R_{0}(z,s,\eps)=((\tilde\BB_{\eps}(se_1)-\eps^2z)^{-1}\chi_2,\chi_2)$.
It is straightforward to check that a solution of $D_2(z,s,\eps)=0$ for any fixed $s,\,\eps$ is a fixed point of $G_j(z,s,\eps)$.

By \eqref{T_5}, when $R_1>0$ is large enough and $\delta>0$ is
small enough, it holds for $\eps s>R_1$ and $z,z_1,z_2\in B (jis,\delta)$ that
\bma
|G_j(z,s,\eps)-jis|&\le \frac12|R_{0}(z,s,\eps)|+\frac{|R_0(z,s,\eps)^2|}{2|\sqrt{R_0(z,s,\eps)^2-4s^2}+2i s|}\le \delta, \label{Y_3a}\\
|G_j(z_1,s,\eps)-G_j(z_2,s,\eps)|&\le \frac12|R_{0}(z_1,s,\eps)-R_{0}(z_2,s,\eps)| \nnm\\
&\quad+\frac{|R_0(z_1,s,\eps)^2-R_0(z_2,s,\eps)^2|}{2|\sqrt{R_0(z_1,s,\eps)^2-4s^2}|+2|\sqrt{R_0(z_2,s,\eps)^2-4s^2}|} \nnm\\
&\le C\eps^2|z_1-z_2|\(1+\frac1s\)\le \frac12|z_1-z_2|.\nnm
\ema
Thus,  by the  contraction mapping theorem,  there is a unique function $z_j(s,\eps)$ satisfying that  $z_j(s,\eps)=G_j(z,s,\eps)$, namely,  $z_j(s,\eps) $ is the solution of $D_2(z,s,\eps)=0$. Set $y_j(s,\eps)=z_j(s,\eps)-jis.$ By \eqref{Y_3a} and \eqref{T_5}, we have
\be |y_j(s,\eps)|\le C|R_0(z_j,s,\eps)|\(1+\frac1s\)\le C|\eps s|^{-\frac12}\to 0,\quad \eps|s|\to \infty. \label{beta}\ee

We now turn to prove \eqref{eigen_h1}. We obtain from \eqref{fp1} that
\be
y_j(s,\eps)=\frac12 R_j(y_j,s,\eps)+\frac12\frac{jR_j(y_j,s,\eps)^2}{\sqrt{R_j(y_j,s,\eps)^2-4s^2}+2i s} =:I_1+I_2,\label{ddd}
\ee
where
$$R_j(y_j,s,\eps)=R_0(z_j,s,\eps)=((\tilde\BB_{\eps}(se_1)-\eps^2jis-\eps^2y_{j} )^{-1}\chi_2,\chi_2). $$
First, we estimate $I_1$.  For this, we decompose
\be
-\(L_1-i(v_1+j \eps)\eps s-i \eps\frac{v_1}s P_{ d}-\eps^2y_{j} \)^{-1}=X_j(s,\eps)+Z_j( s,\eps), \label{decompose}
\ee
where
\be \label{AYZ}
\left\{\bln
X_j(s,\eps)&=(\nu(v)+i(v_1+j \eps)\eps s)^{-1},\\
Z_j(s,\eps)&=(I-Y_j(s,\eps))^{-1}Y_j(s,\eps)X_j(s,\eps),\\
Y_j(s,\eps)&=X_j(s,\eps)\(K_1-i \eps\frac{v_1}s P_{ d}-\eps^2y_{j}\).
\eln\right.
\ee
By \eqref{decompose}, we divide $I_1$ into
\be
I_1= -\(X_j(s,\eps)\chi_2,\chi_2\)- \(Z_{j}(s,\eps)\chi_2,\chi_2\)=:I_3+I_4. \label{I_1}
\ee
It holds that
\bma
I_3
=&-\intr \frac{\nu}{\nu^2+(v_1\pm \eps)^2\eps^2s^2}v^2_2M dv-i\intr \frac{(v_1\pm \eps)\eps s}{\nu^2+(v_1\pm \eps)^2\eps^2s^2}v^2_2M dv\nnm\\
=&-{\rm Re}I_3-i {\rm Im}I_3. \label{aaa}
\ema
By changing variable $(u_1,u_2,u_3)=((v_1\pm\eps)\eps s,v_2,v_3)$, we obtain for $\eps s>1$ that
\bma
{\rm Re}I_3&\le   C\intr \frac{1}{\nu_0^2+(v_1\pm\eps)^2\eps^2s^2} e^{-\frac{|v|^2}4}dv\nnm\\
&\le \frac C{\eps s}\intr \frac{1}{\nu_0^2+u_1^2} e^{-\frac14(\frac{u_1}{\eps s}\mp\eps)^2}e^{-\frac{u_2^2+u_3^2}4}du\nnm\\
&\le \frac C{\eps s}\intr \frac{1}{\nu_0^2+u_1^2} e^{-\frac{u_1^2}4}du_1\le \frac{C_1}{\eps s}, \label{bbb}
\ema
\bma
{\rm Re}I_3&\ge \intr \frac{\nu_0}{\nu_1^2(1+|v|^2)+(v_1\pm\eps)^2\eps^2s^2}v_2^2Mdv\nnm\\
&\ge \frac 1{\eps s}\intr \frac{\nu_0}{\nu_1^2(1+(\frac{u_1}{\eps s}\mp\eps)^2+u^2_2+u^2_3)+u_1^2}u_2^2e^{-\frac12(\frac{u_1}{\eps s}\mp\eps)^2}e^{-\frac{u_2^2+u_3^2}2} du\nnm\\
&\ge \frac C{\eps s}\intr \frac{1}{\nu_1^2(3+u_1^2+u^2_2+u^2_3)+u_1^2}u_2^2e^{-\frac{u_1^2+u_2^2+u_3^2}2}du\ge \frac{C_2}{\eps s},
\ema
and
\bma
|{\rm Im}I_3|\le& \frac{C}{\eps s}\intr \frac{|u_1|}{\nu_0^2+u_1^2}e^{-\frac12(\frac{u_1}{s}\mp\eps)^2}e^{-\frac{u_2^2+u_3^2}2}du\nnm\\
\le& \frac{C}{\eps s}\int^{\eps s}_0 \frac{u_1}{\nu_0^2+u_1^2}du_1+\frac{ C}{\eps^2 s^2}\int^\infty_{\eps s} e^{-\frac12(\frac{u_1}{\eps s}\mp\eps)^2}du_1\nnm\\
\le& C_3\frac{\ln \eps s}{\eps s}. \label{ddd2}
\ema

We now consider  $I_4$. By changing variable $v_2\to -v_2$, we obtain that $(X_j(s,\eps) \chi_2,\chi_0) =0$. This together with \eqref{AYZ} implies that
$$Z_j(s,\eps)\chi_2=(I-Y_{j} )^{-1}X_j (K_1 -\eps^2y_{j})X_j \chi_2 .$$
Since
$$ |k_1(v,u)|\le C\frac{1}{|\bar{v}-\bar{u}|}e^{-\frac{|v-u|^2}8} ,\quad \bar{u}=(u_2,u_3), $$
we can obtain from \eqref{bbb} and \eqref{ddd2} that
\bmas
|K_1X_{\pm 1} \chi_2|&\le C\int_{\R} e^{-\frac{|v_1-u_1|^2}{8}}\frac{\nu+|(u_1\pm \eps)\eps s|}{\nu_0^2+|(u_1\pm \eps)\eps s|^2}e^{-\frac{u_1^2}{4}}du_1 \int_{\R^2}\frac{1}{|\bar{v}-\bar{u}|}e^{-\frac{|\bar{v}-\bar{u}|^2}{8}}e^{-\frac{|\bar{u}|^2}{4}}d\bar{u}\nnm\\
&\le Ce^{-\frac{|v|^2}{8}}\int_{\R} \frac{1+|(u_1\pm \eps)\eps s|}{\nu_0^2+|(u_1\pm \eps)\eps s|^2}e^{-\frac{u_1^2}{8}}du_1\le C\frac{\ln \eps s}{\eps s}e^{-\frac{|v|^2}{8}}.
\emas
This and \eqref{bbb} lead to
\be
\|X_{\pm 1}K_1X_{\pm 1}\chi_2 \|^2 \le C\frac{\ln^2 \eps s}{\eps^2 s^2}\intr \frac{1}{\nu_0^2+(v_1\pm \eps)^2\eps^2s^2 }e^{-\frac{|v|^2}{4}}dv\le C\frac{\ln^2\eps s}{|\eps s|^3}. \label{fff}
\ee
By \eqref{beta} and Lemma \ref{lem1}, it holds that for $\eps |s|>R_1$,
\be \|(I-Y_j(s,\eps))^{-1}\|\le 2, \quad j=\pm 1. \label{Yj}\ee
Thus, it follows from \eqref{beta}, \eqref{bbb}, \eqref{fff} and \eqref{Yj} that
\bma
|I_4|\le &\left|\(X_j(K_1 -\eps^2y_{j}) X_j\chi_2, (I-Y_j^* )^{-1}\chi_2\)\right| \nnm\\
\le &C(\|X_jK_1X_j\chi_2\|  +\eps^2|y_{j}| \| X_j^2\chi_2\|  )\le  C|\eps s|^{-\frac32} \ln |\eps s| . \label{I4}
\ema

Next, we estimate $I_2$ as follows:
\bq |I_2|\le \frac{C}{s}|R_{j}(y_{j},s,\eps)|^2\le  \frac{C}{s}|I_1|^2\le C\frac{\ln^2 \eps s}{\eps^2s^3}.\label{abb}\eq
Combining \eqref{I_1}--\eqref{ddd2}, \eqref{I4} and \eqref{abb}, we obtain \eqref{eigen_h1}.
The proof  of the lemma is then completed.
\end{proof}

\def \z {\mathbb{z}}
\def \e {\mathbb{e}}

With  Lemma \ref{eigen_5}, we have the following lemma about the eigenvalues $\beta_j(|\xi|,\eps)$
and the corresponding eigenvectors $\mathcal{V}_j(\xi,\eps)$ of the operator $\tilde{\AA}_{\eps}(\xi)$ for $\eps|\xi|\ge r_1$.

\begin{lem}\label{eigen_4}
{\rm (1)} There exists a constant $r_1>0$ such that the spectrum $\sigma(\tilde{\AA}_{\eps}(\xi))\cap \{\lambda\in\mathbb{C}\,|\,\mathrm{Re}\lambda>-\mu/2\}$  consists of four eigenvalues $\{\beta_j(s,\eps),\ j=1,2,3,4\}$ for  $\eps s>r_1$ and $s=|\xi|$. In particular, the eigenvalues $\beta_j(s,\eps)$ are $C^{\infty}$ functions in $s$ and $\eps $ and satisfy the following expansion for $\eps s>r_1$:
 \be   \label{specr1a}
 \left\{\bln
 &\beta_1(s,\eps) =\beta_2(s,\eps) = -\eps^2is+\eps^2\zeta_{-1}(s,\eps), \\
  &\beta_3(s,\eps) =\beta_4(s,\eps) = \eps^2is+\eps^2\zeta_{1}(s,\eps),
 \eln\right.
\ee
where $\zeta_{\pm1}(s,\eps)$ is a $C^{\infty}$ function in $s$ and $\eps $ for $\eps s>r_1$ satisfying
\be \label{specr3}
 \frac{C_1}{\eps s}\le -{\rm Re}\zeta_{\pm1}(s,\eps)\le \frac{C_2}{\eps s},\quad  |{\rm Im}\zeta_{\pm1}(s,\eps)|\le C_3 \frac{\ln \eps s}{\eps s},
\ee
with positive constants  $C_1,C_2$ and $C_3$.

{\rm (2)} The eigenvectors $\mathcal{V}_j(\xi,\eps)=\(w_j(\xi,\eps),X_j(\xi,\eps),Y_j(\xi,\eps)\)$, $j=1,2,3,4$ are $C^{\infty}$ in $s$ and $\eps$,  and satisfy for $\eps s>r_1$:
 \bq   \label{eigf2a1}
  \left\{\bln
  &(\mathcal{V}_i ,\mathcal{V}^*_j )=(w_i,\overline{w_j})-(X_i,\overline{X_j})-(Y_i,\overline{Y_j})=\delta_{ij}, \quad 1\le i,j\le 4,\\
&w_j(\xi,\eps)=\eps c_j(s,\eps)\(\beta_j(s,\eps)-L_1 + i \eps s (v\cdot\omega)+i\eps\frac{  v\cdot\omega}{s} P_d\)^{-1}(v\cdot e_j)\chi_0,\\
&X_j(\xi,\eps)=c_j(s,\eps)\omega\times e_j,\quad Y_j(\xi,\eps)=\frac{ i\eps^2s c_j(s,\eps)}{\beta_{j}(s,\eps)}e_j,
  \eln\right.
  \eq
where  $ \mathcal{V}^*_j=(\overline{w_j},-\overline{X_j},-\overline{Y_j})$, and $e_k$, $k=1,2,3,4$ are normal vectors satisfying $e_1=e_3$, $e_2=e_4$, and $e_k\cdot\omega=e_1\cdot e_2=0$, and $c_j(s,\eps)$ are $C^{\infty}$ functions of $s$ and $\eps$ for $\eps s>r_1$ satisfying
\be \label{eigf3a1}
c_j(s,\eps) = i\frac1{\sqrt{2}}+\eps^2 O\(\frac{1}{\eps s}\)+\eps O\(\frac{ \ln \eps s}{\eps^2 s^2}\).
\ee
\end{lem}

\begin{proof}
The eigenvalues $\beta_j(s,\eps)$ and the eigenvectors $\mathcal{V}_j(\xi,\eps) $  for $j=1,2,3,4$ can be constructed as follows. We take $\beta_1=\beta_3=z_{-1}(s,\eps)$ and $\beta_2=\beta_4=z_{1}(s,\eps)$ to be the solution of the equation  $D_2(z,s,\eps)=0$ defined in Lemma \ref{eigen_5}.
The corresponding eigenvectors $\mathcal{V}_j(\xi,\eps)=(w_j(\xi,\eps),X_j(\xi,\eps),Y_j(\xi,\eps))$, $1\le j\le 4$  are given by
$$
 \left\{\bln
  &w_j(\xi,\eps)=\eps c_j(s,\eps)\(\beta_j(s,\eps)-L_1 + i \eps s (v\cdot\omega)+i\eps\frac{ v\cdot\omega}{s} P_d\)^{-1}(v\cdot e_j\chi_0) \\
  &X_j(\xi,\eps)=c_j(s,\eps)\omega\times e_j,\quad Y_j(\xi,\eps)= \frac{ i \eps^2s c_j(s,\eps)}{\beta_j(s,\eps)} e_j ,
  \eln\right.
$$
where  $e_j$, $j=1,2,3,4$ are normal vectors satisfying $e_1=e_3$, $e_2=e_4$, and $e_j\cdot\omega=e_1\cdot e_2=0$.  It is straightforward to check  that   $(\mathcal{V}_1 ,  \mathcal{V}_2^* )=(\mathcal{V}_3 ,  \mathcal{V}_4^* )=0,$ where $ \mathcal{V}^*_j=(\overline{w_j},-\overline{X_j},-\overline{Y_j})$ is the eigenvector of $\tilde{\AA}^*_{\eps}(\xi)$ corresponding to the eigenvalue $\overline{\beta_{j}}$.

Rewrite the eigenvalue problem as
$$
 \tilde{\AA}_{\eps}(\xi)\mathcal{V}_j(\xi,\eps) =\beta_j(s,\eps)\mathcal{V}_j(\xi,\eps), \quad j=1,2,3,4.
 $$
By taking the inner product
$(\cdot,\cdot)_{\xi}$ of above equation with $\mathcal{V}^*_j(s,\eps) $, and by using the fact that
 \bgrs
(\tilde{\AA}_{\eps}(\xi) U,V)_{\xi}=(U,\tilde{\AA}^*_{\eps}(\xi)V)_{\xi},\quad U,V\in  D(\tilde{\AA}_{\eps}(\xi)) ,\\
\tilde{\AA}^*_{\eps}(\xi)\mathcal{V}^*_j(\xi,\eps)=\overline{\beta_j(s,\eps)} \mathcal{V}^*_j(\xi,\eps),
\egrs
we have
 $$
(\beta_i(s,\eps)-\beta_{j}(s,\eps))\(\mathcal{V}_i(\xi,\eps),\mathcal{V}^*_j(\xi,\eps)\)_{\xi}=0,\quad 1\le i,j\le 4.
$$
Since  $\beta_k(s,\eps)\neq \beta_{j}(s,\eps)$ for  $k=1,3,\, j=2,4$ and $P_dw_j(\xi,\eps)=0$,
we have the orthogonal relation
 $$
\(\mathcal{V}_k(\xi,\eps),\mathcal{V}^*_j(\xi,\eps)\)_{\xi}=\(\mathcal{V}_k(\xi,\eps),\mathcal{V}^*_j(\xi,\eps)\)=0,\quad 1\leq k\neq j\leq 4.
 $$
 By nomalization, we have
$$\(\mathcal{V}_j(\xi,\eps),\mathcal{V}^*_j(\xi,\eps)\)=1, \quad j=1,2,3,4.$$
Precisely, the coefficients $c_j(s,\eps)$  is determined by the normalization condition:
$$
 c_j(s,\eps)^2\(-1+\frac{\eps^4s^2}{\beta_{j}(s,\eps)^2}+\eps^2D_j(s,\eps)\)=1,\quad j=1,2,3,4,
$$
 where $D_j(s,\eps)=((\tilde\BB_{\eps}(se_1)-\beta_j)^{-1}\chi_2, (\tilde\BB_{\eps}(-se_1)-\overline{\beta_j})^{-1} \chi_2) $.
Since
\bmas
\frac{\beta_{j}^2(s,\eps)}{\eps^4s^2}&=\(i+O\(\frac{\ln \eps s}{ \eps s^2 }\)\)^2=-1 +O\(\frac{\eps\ln \eps s}{ \eps^2 s^2 }\),  \\
D_j(s,\eps)&=O(1) \|(\tilde\BB_{\eps}(se_1)-\beta_j)^{-1}\chi_2\|^2=O\(\frac{1}{\eps s}\),
 \emas
it follows that
\bmas
c_j^2(s,\eps)=&-\bigg(2-\bigg(\frac{s^2}{\beta_j^2(s,\eps)}+1\bigg)-\eps^2D_j(s,\eps)\bigg)^{-1}\\
=&-\frac12 +O\(\frac{\eps^2}{\eps s}\)+O\(\frac{\eps\ln \eps s}{\eps^2 s^2}\).
\emas
Thus, we obtain \eqref{eigf2a1} and \eqref{eigf3a1} so that
the proof of the theorem is completed.
\end{proof}

With Lemmas~\ref{LP01}, \ref{spectrum2}, \ref{eigen_4a} and \ref{eigen_4},
similar to Theorem~3.4 in \cite{Li4}, we have the following decomposition of the semigroup $ e^{\frac{t}{\eps^2}\mathcal{\tilde{\AA}}_{\eps}(\xi)}$.

\begin{thm}\label{rate1}
The semigroup $e^{\frac{t}{\eps^2}\tilde{\AA}_{\eps}(\xi)}$ with $\xi\neq0$ can be decomposed into
\bq  e^{\frac{t}{\eps^2}\tilde{\AA}_{\eps}(\xi)}U=S_1(t,\xi,\eps)U+S_2(t,\xi,\eps)U+S_3(t,\xi,\eps)U,\quad \forall\,U\in L^2_{\xi}(\R^3_v)\times \C^3_{\xi}\times \C^3_{\xi},  \label{B_0}\eq
where
\bma
S_1(t,\xi,\eps)U&=\sum^4_{j=0}e^{\frac{t}{\eps^2}\lambda_j(|\xi|,\eps)}\(U, \mathcal{U}^*_j(\xi,\eps) \)_{\xi} \mathcal{U}_j(\xi,\eps)1_{\{\eps (1+|\xi|)\le r_0\}}, \label{S1}\\
S_2(t,\xi,\eps)U&=\sum^4_{k=1}e^{\frac{t}{\eps^2}\beta_k(|\xi|,\eps)}\(U, \mathcal{V}^*_k(\xi,\eps) \) \mathcal{V}_k(\xi,\eps)1_{\{\eps |\xi|\ge r_1\}}, \label{S2}
\ema
with $(\lambda_j(|\xi|,\eps),\mathcal{U}_j(\xi,\eps))$ and $(\beta_k(|\xi|,\eps),\mathcal{V}_k(\xi,\eps))$ being the eigenvalue and eigenvector of the operator $\tilde{\AA}_{\eps}(\xi)$ for $\eps(1+|\xi|)\le r_0$ and $\eps|\xi|\ge r_1$ respectively.
And $S_3(t,\xi,\eps)U $ satisfies
\be
\|S_3(t,\xi,\eps)U\|_{\xi}\le Ce^{-\frac{bt}{\eps^2}}\|U\|_{\xi},\label{S3}
\ee
 where the  two constants $b>0$ and $C>0$ are independent of $\xi$ and $\eps$.
\end{thm}

\subsection{Spectral structure of $\BB_{\eps}(\xi)$}

The spectrum structure of $\BB_{\eps}(\xi)$ is now well known. We just list the results in the following to be  self-contained.

\begin{lem}[\cite{Ellis}]\label{SG_1a}
The operator $\BB_{\eps}(\xi)$ generates a strongly continuous contraction semigroup on
$L^2(\R^3_v)$, which satisfies
$$
\|e^{t\BB_{\eps}(\xi)}f\| \le\|f\| , \quad \forall\, t>0,\,f\in
L^2(\R^3_v).
$$
\end{lem}

\begin{lem}[\cite{FL-3,Ukai1}]\label{spectrum} The following statements hold.
\begin{enumerate}
  \item[{\rm (1)}] For any $\delta>0$ and all $\xi\in\R^3$, there exists $y_1=y_1(\delta)>0$ such that
  $$
\rho(\BB_{\eps}(\xi))\supset\{\lambda\in\mathbb{C}\,|\,\mathrm{Re}\lambda\ge-\nu_0
+\delta,\,|\mathrm{Im}\lambda|\geq y_1\}\cup\{\lambda\in\mathbb{C}\,|\,\mathrm{Re}\lambda>0\}.
$$
  \item[{\rm (2)}] For any $r_1>0$, there exists a constant $\alpha=\alpha(r_1)>0$ such that for $\eps|\xi|\ge r_1$,
$$ \sigma(\BB_{\eps}(\xi))\subset \{\lambda\in \C\,|\, {\rm Re}\lambda<-\alpha\}. $$

\end{enumerate}
\end{lem}

\begin{thm}[\cite{FL-3,Ukai1}]\label{spect3a}
{\rm (1)} There exists a constant $r_0>0$ such that for $\eps|\xi|\le r_0$,
$$\sigma(\BB_{\eps}(\xi))\cap \{\lambda\in \C\,|\, {\rm Re}\lambda\ge -\frac{\nu_0}{2}\}=\{\gamma_j(|\xi|,\eps),\, j=-1,0,1,2,3\}.$$
In particular, the eigenvalues $\gamma_j(|\xi|,\eps)$, $j=-1,0,1,2,3$  are analytic functions of $\eps |\xi|$ and satisfy the following expansion for $\eps|\xi|\leq r_0$:
 \be
 \gamma_{j}(|\xi|,\eps) =  i\mu_j\eps|\xi|-a_{j}\eps^2|\xi|^2+O\( \eps^3|\xi|^3\),
\ee
 where
 \bq \label{hj}
\left\{\bal
\mu_{\pm1}=\pm \sqrt{\frac53}, \quad \mu_k=0,\quad k=0,2,3,\\
a_{j}=-(L^{-1} P_1(v\cdot\omega)h_j,(v\cdot\omega)h_j)>0 ,\\
h_0(\xi)=\sqrt{\frac25}\chi_0-\sqrt{\frac35}\chi_4 ,\\
h_{\pm1}(\xi) =\sqrt{\frac3{10}}\chi_0\mp\frac{\sqrt2}2(v\cdot\omega)\chi_0+\sqrt{\frac15}\chi_4 ,\\
h_{k}(\xi)= (v\cdot W^k)\chi_0 ,\quad k=2,3,
\ea\right.
\eq
 and $W^j$ $(j=2,3)$ are orthonormal vectors satisfying
$W^j\cdot\omega=0$.

{\rm (2)} The corresponding eigenfunctions $\psi_j(\xi,\eps)=\psi_j(\eps|\xi|,\omega)$, $j=-1,0,1,2,3$ satisfy
 \bq
  \left\{\bln                      \label{eigf1}
 & (\psi_i(\xi,\eps) ,\overline{\psi_j(\xi,\eps)} ) =\delta_{ij}, \quad -1\le i,j\le 3,\\
 & \psi_j(\xi,\eps) =P_0\psi_0(\xi,\eps)+P_1\psi_0(\xi,\eps),\\
&P_0\psi_j(\xi,\eps) =h_j(\xi)+O(\eps|\xi|),\\
&P_1\psi_j(\xi,\eps) =i\eps|\xi| L^{-1}P_1(v\cdot\omega)h_j(\xi)+O(\eps^2|\xi|^2).
  \eln\right.
  \eq
\end{thm}

\begin{thm}[\cite{FL-3,Ukai1}] \label{rate2}
 The semigroup $ e^{\frac{t}{\eps^2}\BB_{\eps}(\xi)}$ with $\xi\in \R^3$ satisfies
\be
e^{\frac{t}{\eps^2}\BB_{\eps}(\xi)}f=S_4(t,\xi,\eps)f+S_5(t,\xi,\eps)f,
     \quad \forall\,f\in L^2(\R^3_v),  \label{E_3}
\ee
 where
$$
 S_4(t,\xi,\eps)f=\sum^3_{j=-1}e^{\frac{t}{\eps^2}\gamma_j(|\xi|,\eps)}
              \(f,\overline{\psi_j(\xi,\eps)}\)\psi_j(\xi,\eps) 1_{\{\eps|\xi|\leq r_0\}},
 $$
and  $S_5(t,\xi,\eps)f$ satisfies
\be
 \|S_5(t,\xi,\eps)f\| \leq Ce^{-\frac{bt}{\eps^2}}\|f\| \label{E_4}
\ee
with  two constants $b>0$ and $C>0$ independent of $\xi$ and $\eps$.
\end{thm}

\section{Fluid approximations}
\setcounter{equation}{0}
\label{sect3}

In this section,  we will present  the first and second order of  fluid approximations to the semigroups $e^{\frac{t}{\eps^2}\AA_\eps}$ and  $e^{\frac{t}{\eps^2}\BB_\eps}$  that is a  key step to show the convergence  of the solution of the VMB system to the solution of the  NSMF system.

  For any $U_0=(g_0,E_0,B_0)\in H^l$, set
 \bq
 e^{\frac{t}{\eps^2}\AA_{\eps}(\xi)}\hat U_0=\( g ,-\frac{i\xi}{|\xi|^2}(g,\chi_0)-\frac{\xi}{|\xi|}\times X,-\frac{\xi}{|\xi|}\times Y\),
\label{solution1}
  \eq
where
\bmas
e^{\frac{t}{\eps^2}\tilde{\AA}_{\eps}(\xi)}\hat V_0
 &=( g,X,Y)\in L^2_{\xi}(\R^3_v)\times \mathbb{C}^3_{\xi}\times \mathbb{C}^3_{\xi},\\
\hat V_0&=\(\hat g_0,\frac{\xi}{|\xi|}\times \hat E_0,\frac{\xi}{|\xi|}\times \hat B_0\).
\emas

For any $f_0\in H^l $ and any $U_0=(g_0,E_0,B_0) \in H^l$, set
 \be
 \left\{\bln
  e^{\frac{t}{\eps^2}\BB_\eps}f_0&=(\mathcal{F}^{-1}e^{\frac{t}{\eps^2}\BB_\eps(\xi)}\mathcal{F})f_0,\\
  e^{\frac{t}{\eps^2}\AA_\eps}U_0&=(\mathcal{F}^{-1}e^{\frac{t}{\eps^2}\AA_\eps(\xi)}\mathcal{F})U_0.
  \eln\right.
  \ee
Then $e^{\frac{t}{\eps^2}\BB_\eps}f_0$ and $e^{\frac{t}{\eps^2}\AA_{\eps}}U_0$ are the solutions of the systems \eqref{LVMB0} and \eqref{LVMB1} respectively. By Lemmas \ref{SG_1} and \ref{SG_1a}, it holds that
 \bmas
  \|e^{\frac{t}{\eps^2} \BB_{\eps}} f_0\|_{H^l}&=\intr (1+|\xi|^2)^l\|e^{\frac{t}{\eps^2} \BB_{\eps}(\xi)} \hat f_0\|^2 d\xi\le \intr (1+|\xi|^2)^l\|\hat f_0\|^2 d\xi
=\|f_0\|_{H^l},\\
 \|e^{\frac{t}{\eps^2} \AA_{\eps}} U_0\|_{H^l}&=\intr (1+|\xi|^2)^l\|e^{\frac{t}{\eps^2}\tilde{\AA}_{\eps}(\xi)}\hat V_0\|^2_{\xi} d\xi\le \intr (1+|\xi|^2)^l\|\hat V_0\|^2_{\xi} d\xi
=\|U_0\|_{H^l},
\emas
where we have used $\|\hat V_0\|^2_{\xi}=\|\hat U_0\|^2.$

\subsection{Semigroup of the linear NSMF system}

In this subsection, we will  study the solution to the linear bipolar NSMF system. Firstly, we consider  the following linearized bipolar NSMF system to \eqref{NSM_2} for $U_1=(n,m,q)$ and
$U_2=(\rho,E,B)$:
\be \label{LNSM1}
\left\{\bal
\Tdx\cdot m=0,\quad n+\sqrt{\frac23}q=0,\\
\dt m-\kappa_0\Delta_x m +\Tdx p=G_{1}, \\
 \dt q - \kappa_1\Delta_x q= \frac35G_{2},
\ea\right.
\ee
and
\be\label{LNSM2}
\left\{\bal
 \dt E-\Tdx\times B=\eta(\Tdx \rho-E)+G_{3},\\
 \dt B+\Tdx\times E=0, \\
\Tdx\cdot E=\rho,\quad \Tdx\cdot B=0,
\ea\right.
\ee
where $G_{1},\, G_{3}  \in \R^3$ and $G_{2}\in \R$ are given functions, $p$ is the pressure satisfying $p=\Delta^{-1}_x\divx G_{1}$, and the initial data $(n,m,q)(0)$ and $(\rho,E,B)(0)$ satisfy   \eqref{NSP_5i}.

For any $\hat{f}_0=\hat{f}_0(\xi,v)\in N_0$ and $\hat{V}_0=(\hat{\rho}_0(\xi)\chi_0,\hat{E}_0(\xi),\hat{B}_0(\xi))\in N_1\times \C^3_{\xi}\times \C^3_{\xi}$, set
\bma
Y_1(t,\xi)\hat{f}_0&=\sum_{j=0,2,3}e^{-a_j|\xi|^2t}\(\hat{f}_0,h_j(\xi)\) h_j(\xi),\label{v1}\\
\tilde{Y}_2(t,\xi)\hat{V}_0&=\sum^4_{j=0}e^{b_j(|\xi|)t}\(\hat{V}_0,\overline{\mathcal{X}_j(\xi)}\)_{\xi}\mathcal{X}_j(\xi),\quad |\xi|\ne \frac{\eta}{2}, \label{v1a}
\ema
where $b_k(|\xi|)$ $(k=0,1,2,3,4)$ and $(a_j$, $h_j(\xi))$  $(j=0,2,3)$ are defined by \eqref{bj} and \eqref{hj} respectively, and $\mathcal{X}_j(\xi)$, $j=0,1,2,3,4$ are given by
\be \label{X_3}
\left\{\bln
&\mathcal{X}_0(\xi)=\bigg(\frac{|\xi|}{\sqrt{1+|\xi|^2}}\chi_0,0,0\bigg),\\
&\mathcal{X}_k(\xi)=\frac{b_k}{\sqrt{b^2_k-|\xi|^2 }}\(0,\frac{\xi}{|\xi|}\times e_k, \frac{i |\xi| e_k}{b_k}\),\ \ k=1,2,3,4.
\eln\right.
\ee
Here,  $e_k$, $k=1,2,3,4$ are normal vectors satisfying $e_1=e_3$, $e_2=e_4$, and $e_k\cdot\omega=e_1\cdot e_2=0$.

  For any $\hat{U}_0=(\hat{\rho}_0 \chi_0,\hat{E}_0,\hat{B}_0)\in N_1\times \C^3 \times \C^3 $ with $\hat{\rho}_0=i\xi\cdot \hat{E}_0$, set
 \bq
 Y_2(t,\xi)\hat U_0=\( \hat \rho \chi_0,-\frac{i\xi}{|\xi|^2}\hat \rho-\frac{\xi}{|\xi|}\times \hat X,-\frac{\xi}{|\xi|}\times \hat Y\),
\label{solution1a}
  \eq
  with
$$
\left.\bln
\tilde{Y}_2(t,\xi)\hat V_0
 &=(\hat \rho\chi_0,\hat X,\hat Y)\in L^2_{\xi}(\R^3_v)\times \mathbb{C}^3_{\xi}\times \mathbb{C}^3_{\xi},\\
\hat V_0&=\(\hat \rho_0\chi_0,\frac{\xi}{|\xi|}\times \hat E_0,\frac{\xi}{|\xi|}\times \hat B_0\).
\eln\right.
$$
Denote
\be
\left\{\bln
Y_1(t)f_0&=(\mathcal{F}^{-1}Y_1(t,\xi)\mathcal{F})f_0, \label{v2}\\
Y_2(t)U_0&=(\mathcal{F}^{-1}Y_2(t,\xi)\mathcal{F})U_0.
\eln\right.
\ee
It is straightforward to check that
\be
\left\{\bln
&f=P_{||}f+P_{\bot}f, \quad P_{||}Y_1(t)=0, \quad P_{\bot}Y_1(t)=Y_1(t),\\
&\|P_{||}f\|^2=|(f,v\chi_0)_{||}|^2+|(f,\tilde{h}_1)|^2,\\
&\|P_{\bot}f\|^2=|(f,v\chi_0)_{\bot}|^2+|(f,\tilde{h}_0)|^2+\|P_1f\|^2.
\eln\right.
\ee

Then, we can represent the solutions to the NS system \eqref{LNSM1} and NSM type system \eqref{LNSM2} by the semigroups $Y_1(t)$ and $Y_2(t)$ respectively.

\begin{lem} \label{sem}
For any $f_0\in L^2$, $U_0=(g_0,E_0,B_0)\in L^2$ and $G_j\in L^1_t(L^2_x)$, $j=1,2,3$, we define
\bma
u(t,x,v)&=Y_1(t)P_0f_0+\intt Y_1(t-s)H_1(s)ds, \label{U_1}\\
U(t,x,v)&=Y_2(t)P_2U_0+\intt Y_2(t-s)H_2(s)ds, \label{U_2}
\ema
 where 
$$
\left.\bln
H_1(t,x,v)&=G_{1}(t,x)\cdot v\chi_0+G_{2}(t,x)\chi_4,\\
H_2(t,x,v)&=(\Tdx\cdot G_{3}(t,x)\chi_0,G_{3},0).
\eln\right.
$$
Let $(n,m,q)=((u,\chi_0),(u,v\chi_0),(u,\chi_4))$ and  $U=(\rho\chi_0,E,B) $. Then $(n,m,q)(t,x)\in L^\infty_t(L^2_x)$ and $(\rho,E,B)(t,x)\in L^\infty_t(L^2_x)$ are the unique global solutions to the linear NS system \eqref{LNSM1} and NSM type system \eqref{LNSM2} with the initial datas $(n,m,q)(0)$ and $(\rho,E,B)(0)$ satisfying \eqref{NSP_5i}.
\end{lem}

\begin{proof}
To show \eqref{U_1}, by taking Fourier transform to \eqref{LNSM1}, we have
\bma
&i\xi\cdot \hat{m}=0,\quad \hat{n} +\sqrt{\frac23}\hat{q}=0,\label{LNSP_1a}\\
&\dt \hat{m}-\kappa_0|\xi|^2 \hat{m}+i\xi\hat{p}_1 =\hat{G}_{1},\label{LNSP_3a}\\
&\dt \hat{q} - \kappa_1|\xi|^2 \hat{q}=\frac35\hat{G}_{2},\label{LNSP_4a}
\ema
where the initial data $(\hat{n},\hat{m},\hat{q})(0)$ satisfies
$$
 \hat{m}(0)= (P_0\hat{f}_0,v\chi_0)_{\bot},\quad  \sqrt{\frac32}\hat{n}(0) =-\hat{q}(0)=\sqrt{\frac35}\bigg(P_0\hat{f}_0,\sqrt{\frac25}\chi_0-\sqrt{\frac35}\chi_4\bigg).
$$

Then, it follows from \eqref{LNSP_4a} and \eqref{LNSP_1a} that
\bma
\hat{q}(t,\xi)&=e^{-\kappa_1|\xi|^2t} \hat{q}(0)  +\intt e^{-\kappa_1|\xi|^2(t-s)}\hat{G}_{2}(s)ds\nnm\\
&=e^{-\kappa_1|\xi|^2t} \(P_0\hat{f}_0,h_0(\xi)\) (h_0(\xi),\chi_4)\nnm\\
&\quad +\intt e^{-\kappa_1|\xi|^2(t-s)} \(\hat{H}_1(s),h_0(\xi)\) (h_0(\xi),\chi_4)ds,\label{n3}
\ema
and
\bma
\hat{n}(t,\xi)&=e^{-\kappa_1|\xi|^2t} \(P_0\hat{f}_0,h_0(\xi)\)(h_0(\xi),\chi_0)\nnm\\
&\quad +\intt e^{-\kappa_1|\xi|^2(t-s)} \(\hat{H}_1(s),h_0(\xi)\)(h_0(\xi),\chi_0)ds. \label{n4}
\ema
By \eqref{LNSP_1a} and \eqref{LNSP_3a} and noting that $ \hat m=\hat m_{\bot}$, we have
\bma
\hat{m}(t,\xi)&=e^{-\kappa_0|\xi|^2t} \hat{m}(0)  +\intt e^{-\kappa_0|\xi|^2(t-s)} \hat{G}_1(s)_{\bot}ds\nnm\\
&=\sum_{j=2,3}e^{-\kappa_0|\xi|^2t} \(P_0\hat{f}_0,h_j(\xi)\)(h_j(\xi),v\chi_0) \nnm\\
&\quad+\sum_{j=2,3}\intt e^{-\kappa_0|\xi|^2(t-s)} \(\hat{H}_1(s),h_j(\xi)\)(h_j(\xi),v\chi_0)ds. \label{n5}
\ema
Noting that $\kappa_0=a_2=a_3$, $\kappa_1=a_0$, $(h_0(\xi),v\chi_0)=0$ and $(h_j(\xi),\chi_0)=(h_j(\xi),\chi_4)=0$, $j=2,3$, we obtain \eqref{U_1} by using \eqref{n3}--\eqref{n5}.

Next, we prove \eqref{U_2} as follows.  Taking Fourier transform to \eqref{LNSM2} gives  the system for $(\hat{\rho},\hat{E},\hat{B})$:
\bma
&\dt \hat{\rho}+ \eta(1+|\xi|^2 )\hat{\rho}=i\xi\cdot\hat{G}_{3}, \label{n1}\\
 &\dt \hat{E}=i\xi\times B+\eta(i\xi \hat{\rho}-\hat{E})+\hat{G}_{3}, \label{E-1}\\
 &\dt \hat{B}=-i\xi\times \hat{E},\label{B-2} \\
 &i\xi\cdot \hat{E}=\hat{\rho}, \quad i\xi\cdot \hat{B}=0,\nnm
\ema
where the initial data $ (\hat{\rho},\hat{E},\hat{B})(0)$ satisfies
$$
 \hat{\rho}(0)=i\xi\cdot \hat{E}_0, \quad  \hat{E}(0)= \hat{E}_0,\quad  \hat{B}(0)= \hat{B}_0.
$$
Taking $\omega\times $ to \eqref{E-1} and \eqref{B-2} yields
\bma
 \dt (\omega\times \hat{E})&=i\xi\times (\omega\times \hat{B})-\eta (\omega\times \hat{E})+\omega\times \hat{G}_{3},\label{E-2}\\
 \dt (\omega\times \hat{B})&=-i\xi\times (\omega\times \hat{E}). \label{B-3}
\ema
Let $\hat{V}=(\hat{\rho}\chi_0 ,\omega\times \hat{E},\omega\times \hat{B})^T\in L^2_{\xi}(\R^3)\times \C^3_{\xi}\times \C^3_{\xi}$. Then, the system \eqref{n1}, \eqref{E-2} and \eqref{B-3} can be written as
$$ \dt \hat{V}=A(\xi)\hat{V}+\hat H_3, $$
where $\hat{H}_3=(i\xi\cdot \hat{G}_3\chi_0 ,\omega\times \hat{G}_3,0)^T$, and
$$
A(\xi)=\left(\ba
-\eta(1+|\xi|^2) &0 &0\\
0 &-\eta &i\xi\times\\
0 &-i\xi\times & 0
\ea\right).
$$
 It is straightforward to check  that $A^*(\xi)=\overline{A(\xi)}$ and $A(\xi) $ admits five eigenvalues $b_j(|\xi|)$ given by \eqref{bj} with  eigenfunctions $ \mathcal{X}_j(\xi)$ given by  \eqref{X_3}.  Note that $b_1(|\xi|)=b_3(|\xi|)=-\eta/2$ when $ |\xi|= \eta/2$, and $ \mathcal{X}_j(\xi)$ satisfy the orthnormal relation
$(\mathcal{X}_i(\xi),\overline{\mathcal{X}_j(\xi)})_{\xi}=\delta_{ij}$ for $ |\xi|\ne \eta/2$.
Thus
\bmas
\hat{V}(t,\xi)=&\sum^4_{j=0}e^{b_j(|\xi|)t}\(\hat V(0),\overline{\mathcal{X}_j(\xi)}\)_{\xi} \mathcal{X}_j(\xi)\\
&+\sum^4_{j=0}\intt e^{b_j(|\xi|)(t-s)}\(\hat H_3(s),\overline{\mathcal{X}_j(\xi)}\)_{\xi} \mathcal{X}_j(\xi)ds,
\quad |\xi|\ne \frac{\eta}2.
\emas
This proves \eqref{U_2} and completes the proof of the lemma.
\end{proof}

\subsection{Fluid approximation of $e^{\frac{t}{\eps^2}\AA_{\eps}}$}

The following lemma will be used to study the  fluid dynamical approximations of the semigroups $e^{\frac{t}{\eps^2}\AA_\eps}$ and $e^{\frac{t}{\eps^2}\BB_\eps}$.

\begin{lem} \label{S2a}
For any $V_0\in N_1\times \C^3_{\xi}\times \C^3_{\xi}$ and $f_0\in N_0$, we have
\bma
\|S_{3}(t,\xi,\eps)V_0\|_{\xi}&\le C\(\eps(1+|\xi|)1_{\{\eps(1+|\xi|)\le r_0\}}+1_{\{\eps(1+|\xi|)\ge r_0\}}\)e^{-\frac{bt}{\eps^2}}\|V_0\|_{\xi},\label{S5}\\
\|S_{5}(t,\xi,\eps)f_0\| &\le C\(\eps |\xi| 1_{\{\eps |\xi|\le r_0\}}+1_{\{\eps |\xi| \ge r_0\}}\)e^{-\frac{bt}{\eps^2}}\|f_0\|  ,\label{S6}
\ema
where $S_3(t,\xi,\eps)$ and $S_5(t,\xi,\eps)$ are given in Theorems \ref{rate1} and \ref{rate2} respectively.
\end{lem}

\begin{proof}
For any $ V\in L^2_{\xi}(\R^3)\times \C^3_{\xi}\times \C^3_{\xi},$ define a projection $P_{\eps}(\xi)$ by
$$P_{\eps}(\xi)V= \sum^4_{j=0}\(V,\mathcal{U}^*_j(\xi,\eps)\)_{\xi} \mathcal{U}_j(\xi,\eps),\quad \eps(1+|\xi|)\le r_0,$$
where $ \mathcal{U}_j(\xi,\eps)$, $j=0,1,2,3,4$ are the eigenfunctions of $\tilde{\AA}_{\eps}(\xi)$ for $\eps(1+|\xi|)\le r_0$ given in  \eqref{eigf2}.

By Theorem \ref{rate1}, we claim that
\be S_1(t,\xi,\eps)=e^{\frac{t}{\eps^2}\tilde{\AA}_{\eps}(\xi)}P_{\eps}(\xi)1_{\{\eps(1+|\xi|)\le r_0\}}. \label{S1a}
\ee
Indeed, it follows from semigroup theory that for $\kappa>0$,
 \bmas
  e^{\frac{t}{\eps^2}\tilde{\AA}_{\eps}(\xi)}P_{\eps}(\xi)V =&\frac1{2\pi i}\int^{\kappa+ i\infty}_{\kappa- i\infty}
   e^{ \frac{\lambda t}{\eps^2}}(\lambda-\tilde{\AA}_{\eps}(\xi))^{-1}P_{\eps}(\xi)Vd\lambda\\
   =&\frac1{2\pi i}\sum^4_{j=0}\int^{\kappa+ i\infty}_{\kappa- i\infty}
   e^{ \frac{\lambda t}{\eps^2}}(\lambda-\lambda_j(|\xi|,\eps))^{-1}\(V,\mathcal{U}^*_j\)_{\xi} \mathcal{U}_jd\lambda\\
   =&\sum^4_{j=0}e^{\frac{t}{\eps^2}\lambda_j(|\xi|,\eps) }\(V,\mathcal{U}^*_j\)_{\xi} \mathcal{U}_j=S_1(t,\xi,\eps)V   .
 \emas
Thus, by Theorem \ref{rate1} we  can decompose $S_3(t,\xi,\eps)$ into
\be
S_3(t,\xi,\eps)=S_{31}(t,\xi,\eps)+S_{32}(t,\xi,\eps),\label{S3b}
\ee
where
$$
\left\{\bln
S_{31}(t,\xi,\eps)&=e^{\frac{t}{\eps^2}\tilde{\AA}_{\eps}(\xi)}(I- P_{\eps}(\xi))1_{\{\eps(1+|\xi|)\le r_0\}},\\
S_{32}(t,\xi,\eps)&=e^{\frac{t}{\eps^2}\tilde{\AA}_{\eps}(\xi)}1_{\{\eps(1+|\xi|)\ge r_0\}}.
\eln\right.
$$
Moreover, $S_{3k}(t,\xi,\eps)$, $k=1,2$ satisfy
\be
\|S_{3k}(t,\xi,\eps)V\|_{\xi}\le Ce^{-\frac{bt}{\eps^2}}\| V\|_{\xi},\quad k=1,2.
\ee

For any $V_0\in N_1\times \C^3_{\xi}\times \C^3_{\xi}$,  we can obtain by \eqref{Y_4} and \eqref{Y_3} that
$$
\|V_0-P_{\eps}(\xi)V_0\|_{\xi} =\|\tilde{Y}_2(0,\xi)V_0 -S_1(0,\xi,\eps)V_0\|_{\xi}\le C\eps \|V_0\|_{\xi}.
$$
The above estimate  and the the fact $S_{31}(t,\xi,\eps)=S_{31}(t,\xi,\eps)(I- P_{\eps}(\xi))$ imply that
\be
\|S_{31}(t,\xi,\eps)V_0\|_{\xi}\le C\eps(1+|\xi|)1_{\{\eps(1+|\xi|)\le r_0\}}e^{-\frac{bt}{\eps^2}}\|V_0\|_{\xi}.\label{S4}
\ee
By combining \eqref{S3b}--\eqref{S4}, we obtain \eqref{S5}. \eqref{S6} can be proved similarly. And this completes the proof of the lemma.
\end{proof}

The following lemma gives the first order fluid approximation of the semigroup $e^{\frac{t}{\eps^2}\AA_\eps}$.

\begin{lem} \label{fl1}
For any $\eps\ll 1$,  any integer $k,m\ge 0$ and $U_0=(g_0,E_0,B_0)\in L^2\cap L^1 $, it holds that
\bma
 \left\|e^{\frac{t}{\eps^2}\AA_\eps}U_0-Y_2(t)P_2U_0 \right\|_{H^k}
 &\le C \( \eps(1+t)^{-\frac34 }+e^{-\frac{bt}{\eps^2}}\)(\|U_0\|_{H^{k+1} }+\|U_0\|_{L^{1} })\nnm\\
 &\quad +C\eps^m (1+t)^{-m}\|\Tdx^{m}U_0\|_{H^k}, \label{limit1}
\ema
where $Y_2(t)$ is defined by \eqref{solution1a}, $P_2U_0=(P_dg_0 ,E_0,B_0)$,  and $b>0$ is a constant given by \eqref{S3}.
Moreover, if  $U_0=(g_0,E_0,B_0)\in L^2\cap L^1 $ satisfying $P_rg_0=0$, then
\bma
 \left\|e^{\frac{t}{\eps^2}\AA_\eps}U_0-Y_2(t)P_2U_0 \right\|_{H^k}
 &\le C \eps (1+t)^{-\frac34 } (\|U_0\|_{H^{k+1} }+\|U_0\|_{L^{1} })\nnm\\
 &\quad +C\eps^m (1+t)^{-m}\|\Tdx^{m}U_0\|_{H^k}. \label{limit1a}
\ema
\end{lem}

\begin{proof}
For brevity, we only prove the case when $k=0$ because the proof for $k>0$ is similar.
 By \eqref{B_0} and by taking $\eps\le r_0/2$ with $r_0>0$ given in Lemma \ref{eigen_4a}, we have
\bma
\left\|e^{\frac{t}{\eps^2}\AA_\eps}U_0-Y_2(t)P_2U_0\right\|^2_{L^2}&=\intr \left\|e^{\frac{t}{\eps^2}\tilde{\AA}_\eps(\xi)}\hat{V}_0-\tilde{Y}_2(t,\xi)P_2\hat{V}_0\right\|^2_{\xi}d\xi \nnm\\
&\le 4\int_{1+|\xi|\le \frac{r_0}{\eps} } \left\|S_1(t,\xi,\eps)\hat{V}_0-\tilde{Y}_2(t,\xi)P_2\hat{V}_0\right\|^2_{\xi}d\xi\nnm\\
&\quad+4\int_{|\xi|\ge \frac{r_1}{\eps}} \left\|S_2(t,\xi,\eps)\hat{V}_0\right\|^2_{\xi}d\xi+4\intr \left\|S_3(t,\xi,\eps)\hat{V}_0\right\|^2_{\xi}d\xi\nnm\\
&\quad +4\int_{1+|\xi|\ge \frac{r_0}{\eps}}  \left\|\tilde{Y}_2(t,\xi)P_2\hat{V}_0\right\|^2_{\xi}d\xi \nnm\\
&=:I_1+I_2+I_3+I_4, \label{S_4aa}
\ema
where $\hat V_0=(\hat g_0,\omega\times \hat E_0,\omega\times \hat B_0)$.

First,  we decompose $I_1$ into
\bma
I_1&=  4 \(\int_{1+|\xi|\le \frac{r_0}{\eps},|\eta^2-4|\xi|^2|\ge r_0}+\int_{ |\eta^2-4|\xi|^2|\le r_0} \) \left\|S_1(t,\xi,\eps)\hat{V}_0-\tilde{Y}_2(t,\xi)P_2\hat{V}_0 \right\|^2_{\xi}d\xi\nnm\\
&=:I_{11}+I_{12} .
\ema
We estimate $I_{11}$ and $I_{12}$ separately as follows.
By Lemma \ref{eigen_4a}, it holds that for $|\eta^2-4|\xi|^2| \ge r_0$ and $\eps(1+|\xi|)\le r_0$ with $r_0\ll1$,
\be \label{bk}
\frac1{\eps^2}\lambda_k=b_k(1+O(\eps^2)),\quad
\frac{\lambda_k\Theta_k}{\sqrt{\lambda_1\lambda_3-\lambda_k^2}}=\frac{ib_k}{\sqrt{b_k^2-|\xi|^2}}(1+O(\eps^2)), \quad k=1,2,3,4.
\ee
Thus, we can obtain by \eqref{bk} and \eqref{S1} that for $\eps(1+|\xi|)\le r_0$ and  $|\eta^2-4|\xi|^2| \ge r_0$,
\bma
S_1(t,\xi,\eps)\hat{V}_0
=&e^{b_0(|\xi|)t+ O(\eps^2(1+|\xi|^2)^2)t}\[\(\hat{V}_0, \overline{\mathcal{X}_0(\xi)} \)_{\xi}\mathcal{X}_0(\xi)+T_0(\xi,\eps)\]\nnm\\
&+\sum^4_{k=1}e^{b_k(|\xi|)t+ O(\eps^2|b_k(|\xi|)|)t}\[\(\hat{V}_0, \overline{\mathcal{X}_k(\xi)} \) \mathcal{X}_k(\xi)+T_k(\xi,\eps)\], \label{S1b}
\ema
where
\be \label{S1c}
\left\{\bln
\|T_0(\xi,\eps)\|_{\xi}&=O(1) \|\mathcal{U}_0-\mathcal{X}_0\|_{\xi}\|\hat{V}_0\|_{\xi}=O(\eps\sqrt{1+|\xi|^2})\|\hat{U}_0\| ,\\
\|T_k(\xi,\eps)\|_{\xi}&=O(1)\|\mathcal{U}_k-\mathcal{X}_k\| \|\hat{V}_0\| =  O(\eps) \|\hat{U}_0\|, \ \ k=1,2,3,4.
\eln\right.
\ee
Note that
\be \label{bka}
\left\{\bln
&{\rm Re}b_1(|\xi|)\le -c_1|\xi|^2,\,\,\,  |\xi|\le r_0;\quad {\rm Re}b_1(|\xi|)\le -c_2, \,\,\,  |\xi|\ge r_0,\\
&{\rm Re}b_3(|\xi|)\le -\eta/2, \,\,\,   \xi\in \R^3; \quad  b_1(|\xi|)=b_2(|\xi|),  \,\,\, b_3(|\xi|)=b_4(|\xi|),
\eln\right.
\ee
where $c_1,c_2$ are two positive constants. Thus, we have
\bma
I_{11}&\le C \int_{1+|\xi|\le \frac{r_0}{\eps}}e^{b_0t} \[r_0^2\eps^2(1+|\xi|^2)^3 t^2+\eps^2(1+|\xi|^2)\]\|\hat{U}_0\|^2 d\xi \nnm\\
&\quad+C\sum^4_{j=1}\int_{1+|\xi|\le \frac{r_0}{\eps}}e^{{\rm Re}b_jt} \(\eps^4|b_j|^2 t^2+\eps^2\)\|\hat{U}_0\|^2  d\xi \nnm\\
&\le C\int_{ |\xi|\le r_0}\eps^2e^{-c_1|\xi|^2t}(1+\eps^2|\xi|^4t^2)\|\hat{U}_0\|^2d\xi\nnm\\
&\quad +C\int_{1+|\xi|\le \frac{r_0}{\eps} }\eps^2e^{-c_2t}(1+|\xi|^2)\|\hat{U}_0\|^2 d\xi \nnm\\
&\le C\eps^2\(\sup_{|\xi|\le r_0}\|\hat{U}_0\|^2\int_{ |\xi|\le r_0}e^{-c_1|\xi|^2t}d\xi+e^{-c_2t}\intr (1+|\xi|^2)\|\hat{U}_0\|^2 d\xi\) \nnm\\
&\le C\eps^2\[(1+t)^{-3/2}\|U_0\|_{L^1}^2+e^{-c_2t}\|U_0\|_{H^1}^2\], \label{Y_4}
\ema
where we have used
$$\sup_{|\xi|\le r_0}\|\hat{U}_0\|^2\le C \intr \|g_0\|^2_{ L^1_x }dv +C\|(E_0,B_0) \|^2_{L^1_x} \le C\|U_0\|^2_{L^1}.$$

Note that when $ |\eta^2-4|\xi|^2| \to 0$, it holds that $\lambda_1\to \lambda_3$ and $b_1\to b_3$. This implies that $\|\mathcal{X}_k\|_{\xi}\to \infty$ and $\|\mathcal{U}_k\|_{\xi}\to \infty$ for $k=1,2,3,4$. In this situation,  the expansions \eqref{S1b}--\eqref{S1c} do not hold. To overcome this difficulty, we rewrite $S_1(t,\xi,\eps) $ and $\tilde{Y}_2(t,\xi)$ in  other forms. Indeed, since
  \bmas
&e^{\frac{t}{\eps^2}\lambda_1}\frac{\lambda_1}{\lambda_3-\lambda_1}\Theta^2_1+e^{\frac{t}{\eps^2}\lambda_3}\frac{\lambda_3}{\lambda_1-\lambda_3}\Theta^2_3 \\
=&e^{z_1t}\Theta^2_1+ e^{z_1t}\frac{z_3}{z_3-z_1}(\Theta^2_1-\Theta^2_3)
+\(e^{z_1t}-e^{z_3t}\)\frac{z_3}{z_3-z_1} \Theta^2_3,
 \emas
 where $\lambda_j=\eps^2z_j $, and $\Theta_j$, $j=1,2,3,4$ are defined by \eqref{Ak},  we can rewrite $S_1(t,\xi,\eps)$ for $ |\eta^2-4|\xi|^2| < r_0$  as
\bmas
S_1(t,\xi,\eps)\hat{V}_0&=e^{z_0t} \tilde{V}_0+e^{z_1t}\Theta^2_1(\tilde{V}_1+\tilde{V}_2)+(e^{z_1t}-e^{z_3t})\frac{z_3}{z_3-z_1} \Theta^2_3(\tilde{V}_1+\tilde{V}_2)\\
&\quad+e^{z_3t} \frac{z_3}{z_3-z_1}\Theta^2_3\[(\tilde{V}_1-\tilde{V}_3)+(\tilde{V}_2-\tilde{V}_4)\]\\
&\quad+ e^{z_1t}\frac{z_3}{z_3-z_1}(\Theta^2_1-\Theta^2_3) (\tilde{V}_1+\tilde{V}_2),
\emas
where
$$
\left\{\bln
\tilde{V}_0&=(\hat{V}_0,\mathcal{U}^*_0)_{\xi}\mathcal{U}_0, \ \ \tilde{V}_k=(\hat{V}_0,\tilde{\mathcal{U}}^*_k)\tilde{\mathcal{U}}_k,\ \  k=1,2,3,4,\\
\tilde{\mathcal{U}}_k&=\(\eps R(\lambda_k,\eps \xi)(v\cdot e_k)\chi_0, \omega\times e_k, \frac{i se_k}{z_k}\)
\eln\right.
$$ with $R(\lambda,\eps \xi)=(L_1-\lambda-i \eps  P_r(v\cdot\xi))^{-1}$.
Note that
\bmas
&(e^{z_1t}-e^{z_3t})\frac{z_3}{z_3-z_1} =-z_3e^{z_3t}\int^t_0e^{\tau(z_1-z_3)}d\tau,\\
&e^{z_3t} \frac{z_3}{z_3-z_1} \[(\tilde{V}_1-\tilde{V}_3)+(\tilde{V}_2-\tilde{V}_4)\]=z_3e^{z_3t} ( \tilde{V}_{13} + \tilde{V}_{24}),\\
& e^{z_1t}\frac{z_3}{z_3-z_1}(\Theta^2_1-\Theta^2_3)=O(1)\eps^4z_3 e^{z_1t},
\emas where
\be \label{ujk}
\left\{\bln
\tilde{V}_{jk}&=(\hat{V}_0,\tilde{\mathcal{U}}^*_{jk})\tilde{\mathcal{U}}_j+(\hat{V}_0,\tilde{\mathcal{U}}^*_{k})\tilde{\mathcal{U}}_{jk}, \ \  j,k=1,2,3,4,  \\
\tilde{\mathcal{U}}_{jk}&=\( \eps^3 R(\lambda_j,\eps \xi)R(\lambda_k,\eps \xi)(v\cdot e_j)\chi_0, 0, \frac{i se_j}{z_jz_k}\).
\eln\right.
\ee
We have
\bma
S_1(t,\xi,\eps)\hat{V}_0&=e^{z_0t} \tilde{V}_0+e^{z_1t}\Theta^2_1(\tilde{V}_1+\tilde{V}_2)+z_3e^{z_3t}\Theta^2_3 ( \tilde{V}_{13}+ \tilde{V}_{24})\nnm\\
&\quad-z_3e^{z_3t}\int^t_0e^{\tau(z_1-z_3)}d\tau \Theta^2_3 (\tilde{V}_1+\tilde{V}_2)+O(1)\eps^4z_3 e^{z_1t} (\tilde{V}_1+\tilde{V}_2). \label{S_1}
\ema

Similarly, we rewrite $\tilde{Y}_2(t,\xi)$ as
\bma
\tilde{Y}_2(t,\xi)P_2\hat{V}_0
&= e^{b_0t} \tilde{W}_0+e^{b_1t}(\tilde{W}_1+\tilde{W}_2) -b_3e^{b_3t}\int^t_0e^{\tau(b_1-b_3)}d\tau (\tilde{W}_1+\tilde{W}_2) \nnm\\
&\quad+b_3e^{b_3t} (\tilde{W}_{13} + \tilde{W}_{24}), \label{Y_2}
\ema
where
\be \label{xjk}
\left\{\bln
&\tilde{W}_0=(P_2\hat{V}_0,\overline{\mathcal{X}_0})_{\xi}\mathcal{X}_0,\quad \tilde{W}_k=(P_2\hat{V}_0,\overline{X_k})X_k, \\
&\tilde{W}_{jk}=(P_2\hat{V}_0,\overline{X_{jk}})X_j+(P_2\hat{V}_0,\overline{X_{k}})X_{jk}, \ \  j,k=1,2,3,4,\\
 &X_k=\(0, \omega\times e_k, \frac{ise_k}{b_k}\), \ \ X_{jk}=\(0, 0, \frac{i se_j}{b_jb_k}\).
\eln\right.
\ee

Since it follows from Lemma \ref{eigen_4a} that for $|\eta^2-4|\xi|^2|\le r_0$, 
\be \label{bkb}
\left\{\bln
&z_0=b_0+O(\eps^{2}), \quad z_k=b_k+O(\eps), \quad \Theta_k=1+O(\eps), \\
&|\tilde{V}_k-\tilde{W}_k|+|\tilde{V}_{jk}-\tilde{W}_{jk}|=O(\eps)\|\hat{U}_0\|, \quad  j,k=1,2,3,4,
\eln\right.
\ee
we obtain by  \eqref{S_1} and \eqref{Y_2} that
\bma
I_{12}\le & C\int_{ |\eta^2-4|\xi|^2|\le r_0} e^{-\frac{\eta}2t}\(|z_0-b_0|^2 +
\|\mathcal{U}_0-\mathcal{X}_0\|^2_{\xi}\) \|\hat{U}_0\|^2 d\xi \nnm\\
& +C\sum_{k=1,3}\int_{ |\eta^2-4|\xi|^2|\le r_0}e^{-\frac{\eta}2t}\(\eps^2+|z_k-b_k|^2  +|\Theta^2_k-1|^2 \) \|\hat{U}_0\|^2 d\xi \nnm\\
\le&C\eps^2e^{-\frac{\eta}2t} \|U_0\|^2_{L^2} . \label{Y_3}
\ema

Thus, it follows from \eqref{Y_4} and \eqref{Y_3} that
\be
I_{1}\le  C\eps^2  \( e^{-\frac{\eta}2t}\|U_0\|^2_{L^2}+ (1+t)^{-\frac32}\|U_0\|^2_{L^1}\). \label{I1b}
\ee

By \eqref{S2} and Lemma \ref{eigen_4}, we have
$$
S_2(t,\xi,\eps)\hat{V}_0=\sum^4_{k=1}e^{\frac{t}{\eps^2}\beta_k(|\xi|,\eps) } \(\hat{V}_0,  \mathcal{V}^*_k(\xi,\eps)  \) \mathcal{V}_k(\xi,\eps), \quad \eps|\xi|\ge r_1,
$$
which gives
\bma
I_2&=4\int_{|\xi|\ge \frac{r_1}{\eps}} \left\|S_2(t,\xi,\eps)\hat{V}_0\right\|^2_{\xi}d\xi \le C\int_{ |\xi|\ge \frac{r_1}{\eps}} e^{-\frac{ct}{\eps|\xi|}}\|\hat{V}_0\|^2  d\xi \nnm\\
&\le C\sup_{|\xi|\ge \frac{r_1}{\eps}}\frac1{|\xi|^{2m}}e^{-\frac{c t}{\eps|\xi|}}\int_{ |\xi|\ge \frac{r_1}{\eps}} |\xi|^{2m}\|\hat{U}_0\|^2  d\xi
 \le C \eps^{2m}(1+t)^{-2m}\|\Tdx^mU_0\|^2_{L^2 } . \label{I2b}
\ema

By \eqref{S3}, we have
\be
I_3=4\intr \left\|S_3(t,\xi,\eps)\hat{V}_0\right\|^2_{\xi}d\xi\le C\intr e^{-2\frac{bt}{\eps^2}}\|\hat{V}_0\|^2_{\xi} d\xi \le Ce^{-2\frac{bt}{\eps^2}}\|U_0\|^2_{L^2}.
\ee

For $I_4$, it holds that
\bma
I_4&=4\int_{1+|\xi|\ge \frac{r_0}{\eps}}  \left\|\tilde{Y}_2(t,\xi)P_2\hat{V}_0\right\|^2_{\xi}d\xi\le C\int_{1+|\xi|\ge \frac{r_0}{\eps}}e^{-\eta t }\|\hat{V}_0\|^2_{\xi} d\xi \nnm\\
&\le  C\frac{\eps^2}{r_0^2} e^{-\eta t }\int_{1+|\xi|\ge \frac{r_0}{\eps}}(1+|\xi|)^2\|\hat{U}_0\|^2 d\xi
\le  C\eps^2 e^{-\eta t } \|U_0\|^2_{H^1} . \label{S_6}
\ema
Therefore, it follows from \eqref{S_4aa} and  \eqref{I1b}--\eqref{S_6} that
\be
\left\|e^{\frac{t}{\eps^2}\AA_\eps}U_0-Y_2(t)P_2U_0\right\|^2_{L^2 }
\le C\(\eps^2(1+t)^{-\frac32}+e^{-2\frac{bt}{\eps^2}}\)(\|U_0\|^2_{H^1 }+\|U_0\|^2_{L^1}). \label{aaa}
\ee

We now turn to \eqref{limit1a}. Since $f_0\in N_0,$ by Lemma \ref{S2a} we have
\bma
I_3&\le C\int_{1+|\xi|\le \frac{r_0}{\eps}}\eps^2 (1+|\xi|^2)e^{-2\frac{bt}{\eps^2}}\| \hat{V}_0\|^2_{\xi}d\xi+C\int_{1+|\xi|\ge \frac{r_0}{\eps}}e^{-2\frac{bt}{\eps^2}}\| \hat{V}_0\|^2_{\xi}d\xi \nnm\\
&\le C\eps^2 e^{-2\frac{bt}{\eps^2}}\(\int_{1+|\xi|\le \frac{r_0}{\eps}} (1+|\xi|^2) \| \hat{U}_0\|^2 d\xi +\int_{1+|\xi|\ge \frac{r_0}{\eps}} |\xi|^2\| \hat{U}_0\|^2 d\xi\)\nnm\\
&\le C\eps^2 e^{-2\frac{bt}{\eps^2}}\|U_0\|^2_{H^1 }.\label{S_7a-1}
\ema
Thus, by \eqref{I1b}, \eqref{I2b}, \eqref{S_6} and \eqref{S_7a-1} we obtain \eqref{limit1a} for $k=0$.
And this completes the proof of  the lemma.
\end{proof}

The following lemma gives the second order fluid approximation of the semigroup $e^{\frac{t}{\eps^2}\AA_\eps}$.

\begin{lem}\label{fl2}
For any $\eps\ll 1$, any integer $k,m\ge 0$ and  $U_0=(g_0,0,0)\in H^{k+2}\cap L^1 $ satisfying $P_{ d}g_0=0$, we have
\bma
\bigg\|\frac1{\eps}e^{\frac{t}{\eps^2}\AA_\eps}U_0-Y_2(t)Z_0\bigg\|_{H^k}
&\le C \( \eps(1+t)^{-\frac34 }+ \frac1{\eps}e^{-\frac{bt}{\eps^2}}\)(\|U_0\|_{H^{k+2} }+\|U_0\|_{L^{1} })\nnm\\
 &\quad +C \eps^{m}(1+t)^{-m}\|\Tdx^{m}U_0\|_{H^k}, \label{limit2}
\ema
where $Y_2(t)$ is defined in \eqref{solution1a}, $Z_0=( P_d(v\cdot\Tdx L^{-1}_1g_0), (v L^{-1}_1g_0,\chi_0),0),$ and $b>0$ is a constant given by \eqref{S3}.
\end{lem}

\begin{proof} Again we only prove the case when $k=0$ because the proof for $k>0$ is similar.
By \eqref{B_0} and by taking $\eps\le r_0/2$ with $r_0>0$ given in Lemma \ref{eigen_4a}, we have
\bma
\left\|\frac1{\eps}e^{\frac{t}{\eps^2}\AA_\eps}U_0-Y_2(t)Z_0\right\|^2_{L^2}
&\le 4\int_{1+|\xi|\le \frac{r_0}{\eps}} \left\|\frac1{\eps}S_1(t,\xi,\eps)\hat{U}_0-\tilde{Y}_2(t,\xi)\hat{Z}_1\right\|^2_{\xi}d\xi\nnm\\
&\quad+4\int_{|\xi|\ge \frac{r_1}{\eps}} \left\|\frac1{\eps}S_2(t,\xi,\eps)\hat{U}_0\right\|^2_{\xi}d\xi+4\intr \left\|\frac1{\eps}S_3(t,\xi,\eps)\hat{U}_0\right\|^2_{\xi}d\xi\nnm\\
&\quad+4\int_{1+|\xi|\ge \frac{r_0}{\eps}}  \left\|\tilde{Y}_2(t,\xi)\hat{Z}_1\right\|^2_{\xi}d\xi\nnm\\
&=:I_1+I_2+I_3+I_4, \label{S_7}
\ema
where  $\hat{Z}_1=(i P_d(v\cdot\xi L^{-1}_1\hat{g}_0), \omega\times (v L^{-1}_1\hat{g}_0,\chi_0),0)$.

Firstly,  we decompose $I_1$ into
\bma
I_1&=  4 \(\int_{1+|\xi|\le \frac{r_0}{\eps},|\eta^2-4|\xi|^2|\ge r_0}+\int_{ |\eta^2-4|\xi|^2|\le r_0} \)
\left\|\frac1{\eps}S_1(t,\xi,\eps)\hat{V}_0-\tilde{Y}_2(t,\xi)\hat{Z}_1\right\|^2_{\xi}d\xi\nnm\\
&=:I_{11}+I_{12} ,
\ema
By Lemma \ref{eigen_4a} and Theorem \ref{rate1}, it holds for any $U_0=(g_0,0,0) $ with $P_{ d}g_0=0$ that
\be
S_1(t,\xi,\eps)\hat{U}_0=\sum^4_{j=0}e^{\frac{t}{\eps^2}\lambda_j(|\xi|,\eps) } \eps\big(\hat{g}_0,  \tilde{u}_j(\xi,\eps) \big) \mathcal{U}_j(\xi,\eps), \label{S_1c}
\ee
where
$$
\left\{\bln
\tilde{u}_0&=\eps^{-1}P_r\overline{u_0}=i \sqrt{1+s^2}\[1+O(\eps^2(1+s^2))\]R(\overline{\lambda_0},-\eps \xi)(v\cdot \omega)\chi_0, \\
\tilde{u}_k&=\eps^{-1}P_r\overline{u_k}= \frac{\lambda_k\Theta_k}{\sqrt{\lambda_1\lambda_3-\lambda_k^2}}R(\overline{\lambda_k},-\eps \xi)(v\cdot e_k)\chi_0,\quad k=1,2,3,4
\eln\right.
$$  with $R(\lambda,\eps \xi)=(L_1-\lambda-i \eps  P_r(v\cdot\xi))^{-1}$.

Thus, we can obtain by \eqref{S_1c} and \eqref{bk} that for $\eps(1+|\xi|)\le r_0$ and  $|\eta^2-4|\xi|^2| \ge r_0$,
\bma
\frac1{\eps}S_1(t,\xi,\eps)\hat{U}_0
&= e^{\frac{t}{\eps^2}\lambda_0(|\xi|,\eps) }\[i\sqrt{1+|\xi|^2}(v\cdot\omega L^{-1}_1 \hat{g}_0,\chi_0)\mathcal{X}_0(\xi)+R_0(\xi,\eps)\]\nnm\\
&\quad+ \sum^4_{k=1}e^{\frac{t}{\eps^2}\lambda_k(|\xi|,\eps) }\bigg[\frac{b_k}{\sqrt{|\xi|^2-b_k^2}}(v\cdot e_k L^{-1}_1 \hat{g}_0,\chi_0)\mathcal{X}_k(\xi)+R_k(\xi,\eps)\bigg]\nnm\\
&= e^{b_0(|\xi|)t+ O(\eps^2(1+|\xi|^2)^2)t}\[\(\hat{Z}_1, \overline{\mathcal{X}_0(\xi)} \)_{\xi}\mathcal{X}_0(\xi)+R_0(\xi,\eps)\]\nnm\\
&\quad+ \sum^4_{k=1}e^{b_k(|\xi|)t+ O(\eps^2b_k(|\xi|))t} \[\(\hat{Z}_1,\overline{\mathcal{X}_k(\xi)}\) \mathcal{X}_k(\xi)+R_k(\xi,\eps)\], \label{S1a}
\ema
where
$$
\left\{\bln
\|R_0(\xi,\eps)\|_{\xi}&=O(\eps(1+|\xi|^2))\|\hat{U}_0\| ,\\
\|R_k(\xi,\eps)\|_{\xi}&= O(\eps\sqrt{1+|\xi|^2})\|\hat{U}_0\|, \quad k=1,2,3,4
\eln\right.
$$
Thus, it follows form \eqref{bka} and \eqref{S1a} that
\bma
I_{11}&\le C \int_{1+|\xi|\le \frac{r_0}{\eps}}e^{b_0t} \[r_0^2\eps^2(1+|\xi|^2)^4 t^2+\eps^2(1+|\xi|^2)^2\]\|\hat{U}_0\|^2 d\xi \nnm\\
&\quad+C\sum^4_{j=1}\int_{1+|\xi|\le \frac{r_0}{\eps}}e^{{\rm Re}b_jt} \(\eps^4|b_j|^2 t^2+\eps^2\)(1+|\xi|^2)\|\hat{U}_0\|^2  d\xi \nnm\\
&\le C\int_{ |\xi|\le r_0}\eps^2e^{-c_1|\xi|^2t} \|\hat{U}_0\|^2d\xi +C\int_{1+|\xi|\le \frac{r_0}{\eps} }\eps^2e^{-c_2t}(1+|\xi|^2)^2\|\hat{U}_0\|^2 d\xi \nnm\\
&\le C\eps^2\[(1+t)^{-3/2}\|U_0\|_{L^1}^2+e^{-c_2t}\|U_0\|_{H^2}^2\]. \label{I11}
\ema

For $ |\eta^2-4|\xi|^2| < r_0$, we rewrite $\frac1{\eps}S_1(t,\xi,\eps)$  as
\bma
\frac1{\eps}S_1(t,\xi,\eps)\hat{U}_0&= e^{z_0t} \tilde{V}_0+ e^{z_1t}\Theta^2_1(\tilde{V}_1+\tilde{V}_2)+ z_3e^{z_3t} ( \tilde{V}_{13}+ \tilde{V}_{24})\nnm\\
&\quad- z_3e^{z_3t}\int^t_0e^{\tau(z_1-z_3)}d\tau \Theta^2_3 (\tilde{V}_1+\tilde{V}_2)+O(1)\eps^4 z_3e^{z_1t} (\tilde{V}_1+\tilde{V}_2), \label{S_1b}
\ema
where
$$
\left\{\bln
&\tilde{V}_0= (\hat{g}_0, \tilde{u}_0) \mathcal{U}_0,\quad \tilde{V}_k= (\hat{g}_0, \tilde{u}_k)\tilde{\mathcal{U}}_k, \\
&\tilde{V}_{jk}= (\hat{g}_0,\tilde{u}_{jk})\tilde{\mathcal{U}}_j+ (\hat{g}_0,\tilde{u}_{k})\tilde{\mathcal{U}}_{jk}, \quad j,k=1,2,3,4,  \\
&\tilde{u}_{jk} = \eps^2 R(\overline{\lambda_j},-\eps \xi)R(\overline{\lambda_k},-\eps \xi)(v\cdot e_j)\chi_0
\eln\right.
$$ with $\tilde{\mathcal{U}}_k$, $\tilde{\mathcal{U}}_{jk}$ defined by \eqref{ujk}.
Similarly, we rewrite $\tilde{Y}_2(t,\xi)$ as
\bma
\tilde{Y}_2(t,\xi)\hat{Z}_1&= e^{b_0t} \tilde{W}_0+e^{b_1t}(\tilde{W}_1+\tilde{W}_2) -b_3e^{b_3t}\int^t_0e^{\tau(b_1-b_3)}d\tau (\tilde{W}_1+\tilde{W}_2) \nnm\\
&\quad+b_3e^{b_3t} (\tilde{W}_{13} + \tilde{W}_{24}), \label{Y_2a}
\ema
where
$$
\left\{\bln
&\tilde{W}_0=(\hat{Z}_1,\overline{\mathcal{X}_0})_{\xi}\mathcal{X}_0,\quad \tilde{W}_k=(\hat{Z}_1,\overline{X_k})X_k, \\
&\tilde{W}_{jk}=(\hat{Z}_1,X_{jk})X_j+(\hat{Z}_1,X_{k})X_{jk}, \quad j,k=1,2,3,4
\eln\right.
$$   with $X_k$, $X_{jk}$ defined by \eqref{xjk}.

Thus, it follows from \eqref{bkb}, \eqref{S_1b} and \eqref{Y_2a} that
\be
I_{12}\le  C\int_{ |\eta^2-4|\xi|^2|\le r_0} e^{-\frac{\eta}2t}\eps^2 \|\hat{U}_0\|^2 d\xi\le C\eps^2e^{-\frac{\eta}2t} \|U_0\|^2_{L^2} . \label{I12}
\ee

By combining \eqref{S1a}, \eqref{I11} and \eqref{I12}, we obtain
\be
I_{1}\le  C\eps^2 (1+t)^{-\frac32} ( \|U_0\|^2_{H^2}+\|U_0\|^2_{L^1}). \label{I_1a}
\ee

By \eqref{S2} and Lemma \ref{eigen_4}, it holds for $U_0=(g_0,0,0)$ with $P_{ d}g_0=0$ that
$$
S_2(t,\xi,\eps)\hat{U}_0=\sum^4_{k=1}e^{\frac{t}{\eps^2}\beta_k(|\xi|,\eps)} \(\hat{g}_0,  \overline{w_k(\xi,\eps)}  \) \mathcal{V}_k(\xi,\eps) , \quad \eps|\xi|\ge r_1,
$$
which gives
\bma
I_2&\le C\int_{ |\xi|\ge \frac{r_1}{\eps}} \frac{1}{\eps|\xi|}e^{-\frac{ct}{\eps|\xi|}}\|\hat{U}_0\|^2  d\xi
\le C\frac{1}{r_1}\sup_{|\xi|\ge \frac{r_1}{\eps}}\frac1{|\xi|^{2m}}e^{-\frac{c t}{\eps|\xi|}}\int_{ |\xi|\ge \frac{r_1}{\eps}} |\xi|^{2m}\|\hat{U}_0\|^2  d\xi \nnm\\
&\le C \eps^{2m}(1+t)^{-2m}\|\Tdx^mU_0\|^2_{L^2 } . \label{I_2}
\ema

By \eqref{S3} and $P_{ d}g_0=0$, we have
\be
I_3 \le C\intr \frac1{\eps^2} e^{-2\frac{bt}{\eps^2}}\|\hat{U}_0\|^2 d\xi
\le C\frac1{\eps^2}e^{-2\frac{bt}{\eps^2}}\|U_0\|^2_{L^2}. \label{I_3}
\ee

For $I_4$, it holds that
\bma
I_4&\le C\int_{1+|\xi|\ge \frac{r_0}{\eps}}e^{-\eta t }\|\hat{Z}_1\|^2 d\xi\le C \frac{\eps^2}{r_0^2} e^{-\eta t }\int_{1+|\xi|\ge \frac{r_0}{\eps}}(1+|\xi|)^4\|\hat{g}_0\|^2 d\xi \nnm\\
&\le  C\eps^2 e^{-\eta t } \|U_0\|^2_{H^2} . \label{S_6b}
\ema
By combining \eqref{I_1a}, \eqref{I_2},  \eqref{I_3} and \eqref{S_6b}, we obtain \eqref{limit2}
for $k=0$.
And this completes the proof of the lemma.
\end{proof}

In the following lemma, we will present the time decay rates of the semigroup $e^{\frac{t}{\eps^2}\AA_\eps}$.

\begin{lem}\label{time}
For any $\eps\ll 1$, $\alpha\in \mathbb{N}^3$, any integer $k\ge 0$ and $U_0=(g_0,E_0,B_0) $, we have
\bma
 \left\| \dxa P_2e^{\frac{t}{\eps^2}\AA_\eps}U_0 \right\|_{L^{2} }
 &\le C (1+t)^{-\frac34-\frac{m}2}\(\|\dx^{\alpha'} U_0\|_{L^{1} }+\|\dxa U_0\|_{L^2}\)\nnm\\
 &\quad+ C\eps^k(1+t)^{-k}\|\Tdx^{|\alpha|+k} U_0\|_{L^{2} }, \label{time1a}
 \\
  \left\| \dxa P_3e^{\frac{t}{\eps^2}\AA_\eps}U_0 \right\|_{L^{2} }
 &\le C\(\eps (1+t)^{-\frac34-\frac{m}2}+ e^{-\frac{bt}{\eps^2}}\)\(\|\dx^{\alpha'} U_0\|_{L^{1} }+\|\dxa U_0\|_{H^1}\)\nnm\\
 &\quad+ C\eps^{k+1}(1+t)^{-k}\|\Tdx^{|\alpha|+k} U_0\|_{L^{2} }, \label{time2a}
\ema
where $P_2,P_3$ are defined by  \eqref{P23}, $\alpha'\le \alpha$, $m=|\alpha-\alpha'|$,  and $b>0$ is a constant given by \eqref{S3}.

Moreover, if $U_0=(g_0,0,0)$ satisfies that $P_dg_0=0$, then
\bma
 \left\| \dxa P_2e^{\frac{t}{\eps^2}\AA_\eps}U_0 \right\|_{L^{2} }
 &\le C\(\eps (1+t)^{-\frac34-\frac{m}2}+ e^{-\frac{bt}{\eps^2}}\)\(\|\dx^{\alpha'} U_0\|_{L^{1} }+\|\dxa U_0\|_{H^1}\)\nnm\\
 &\quad+ C\eps^{k+1}(1+t)^{-k}\|\Tdx^{|\alpha|+k} U_0\|_{L^{2} }, \label{time3a}
 \\
 \left\| \dxa P_3e^{\frac{t}{\eps^2}\AA_\eps}U_0 \right\|_{L^{2} }
 &\le C\(\eps^2 (1+t)^{-\frac34-\frac{m}2}+ e^{-\frac{bt}{\eps^2}}\)\(\|\dx^{\alpha'} U_0\|_{L^{1} }+\|\dxa U_0\|_{H^2}\)\nnm\\
 &\quad+ C\eps^{k+2}(1+t)^{-k}\|\Tdx^{|\alpha|+k} U_0\|_{L^{2} }. \label{time4a}
\ema
\end{lem}

\begin{proof}
By Theorem \ref{rate2}, we have for $j=2,3$ that
\bma
\left\|\dxa P_je^{\frac{t}{\eps^2}\AA_\eps}U_0\right\|^2_{L^2 }
=& \intr  \left\|\xi^{\alpha}P_je^{\frac{t}{\eps^2}\tilde{\AA}_{\eps}(\xi)}\hat{V}_0 \right\|^2_{\xi}d\xi  \nnm\\
\le &  \int_{1+|\xi|\le \frac{r_0}{\eps}} \left\|\xi^{\alpha}P_jS_1(t,\xi,\eps)\hat{V}_0\right\|^2_{\xi}d\xi +\int_{|\xi|\ge \frac{r_1}{\eps}} \left\|\xi^{\alpha}P_jS_2(t,\xi,\eps)\hat{V}_0\right\|^2_{\xi}d\xi\nnm\\
 &+\intr \left\|\xi^{\alpha}S_3(t,\xi,\eps)\hat{V}_0\right\|^2_{\xi}d\xi, \label{D1c}
\ema
where $\hat V_0=(\hat g_0,\omega\times \hat E_0,\omega\times \hat B_0)$.
By noting
$\|\hat V_0\|^2_{\xi}=\|\hat U_0\|^2,$
 we can estimate the third term on the right hand side of \eqref{D1c} as follows:
\be
\intr (\xi^\alpha)^2 \|S_3(t,\xi,\eps)\hat{V}_0\|^2_{\xi}d\xi \le C e^{-2\frac{bt}{\eps^2}}\intr(\xi^\alpha)^2\|\hat{V}_0\|^2_{\xi}d\xi \le Ce^{-2\frac{bt}{\eps^2}} \|\dxa U_0\|^2_{L^{2}} .\label{t3b}
\ee

By \eqref{S1b} and \eqref{S_1}, we have
\bma
&\quad \int_{1+|\xi|\le \frac{r_0}{\eps}}\|\xi^{\alpha} P_2S_1(t,\xi,\eps) \hat{V}_0\|^2_{\xi}d\xi\nnm\\
&\le  C\int_{ |\xi|\le r_0} e^{-c_1|\xi|^2t}(\xi^\alpha)^2\|\hat{V}_0\|^2_{\xi}d\xi +C\int_{ 1+|\xi|\le \frac{r_0}{\eps}} e^{-c_2t}(\xi^\alpha)^2\|\hat{V}_0\|^2_{\xi}d\xi\nnm\\
&\le C\sup_{ |\xi|\le r_0}\|(\xi)^{\alpha'}\hat{U}_0\|^2 \int_{ |\xi|\le r_0} e^{-c_1|\xi|^2t}|\xi|^{2|\alpha-\alpha'|}d\xi \nnm\\
&\quad + C e^{-c_2 t}\int_{|\xi|\le \frac{r_0}{\eps}}(\xi^\alpha)^2\|\hat{U}_0\|^2 d\xi\nnm\\
&\le C(1+t)^{-\frac32-m} \(\|\dx^{\alpha'} U_0\|^2_{L^{1} }+\|\dxa U_0\|^2_{L^2}\), \label{t1a}
\ema
where $\alpha'\le \alpha$, $m=|\alpha-\alpha'|$, and $c_1,c_2>0$ are some generic constants. Combining \eqref{D1c}--\eqref{t1a} and \eqref{I2b} gives \eqref{time1a}.

By \eqref{S1a} and \eqref{S_1b}, it holds that for $V_0=(g_0,0,0)$ with $P_dg_0=0$,
\bma
&\quad \int_{1+|\xi|\le \frac{r_0}{\eps}}\|\xi^{\alpha} P_2S_1(t,\xi,\eps)\hat{V}_0\|^2 d\xi \nnm\\
&\le  C\eps^2\int_{ |\xi|\le r_0} e^{-c_1|\xi|^2t}(\xi^\alpha)^2\|\hat{V}_0\|^2 d\xi
 +C\eps^2\int_{1+|\xi|\le \frac{r_0}{\eps}} e^{-c_2t}(\xi^\alpha)^2(1+|\xi|^2)\|\hat{V}_0\|^2 d\xi\nnm\\
&\le C\eps^2 (1+t)^{-\frac32-m}\(\|\dx^{\alpha'} U_0\|^2_{L^{1} }+\|\dxa U_0\|^2_{H^1}\),\label{t3a}
\ema
Combining \eqref{D1c}, \eqref{t3a} and \eqref{I_2} yields  \eqref{time3a}.  \eqref{time2a} and \eqref{time4a} can be proved similarly. And this completes the proof of the lemma.
\end{proof}

\begin{lem}\label{timey}
For any $1\le q\le 2$, $\alpha\in \mathbb{N}^3$,  and any  $U_0=(\rho_0 \chi_0,E_0,B_0)$ with $\rho_0=\Tdx\cdot E_0$, we have
\be
 \left\|  \dxa Y_2(t) U_0 \right\|_{L^{2} }
 \le C(1+t)^{-\frac32(\frac1q-\frac12)-\frac{m}2}\(\|\dx^{\alpha'} U_0\|_{L^{q} }+\|\dxa U_0\|_{L^2}\), \label{time5a}
\ee where  $\alpha'\le \alpha$ and $m=|\alpha-\alpha'|$.  Moreover, if  $U_0=(\Tdx\cdot E_0 \chi_0,E_0,0)$, then we have
\be
 \left\|  \dxa Y_2(t) U_0 \right\|_{L^{2} }
 \le C\((1+t)^{-\frac32(\frac1q-\frac12)-\frac{m+1}2}+t^{-\frac12}e^{-\frac{\eta}{2} t}\)\(\|\dx^{\alpha'} E_0\|_{L^{q} }+\|\dxa E_0\|_{L^2}\). \label{time5b}
\ee
\end{lem}

\begin{proof}
By  \eqref{solution1a}, we have
\be
\left\|\dxa Y_2(t) U_0\right\|^2_{L^2 }
= \intr  \left\|\xi^{\alpha}\tilde{Y}_2(t,\xi) \hat{V}_0 \right\|^2_{\xi}d\xi, \label{D1d}
\ee
where $\hat V_0=(\hat \rho_0\chi_0,\omega\times \hat E_0,\omega\times \hat B_0)$.
Since
$$\|\mathcal{X}_0(\xi)\|^2_{\xi}=1, \quad \|\mathcal{X}_k(\xi)\|^2_{\xi}\le \frac{|b_k|^2+|\xi|^2}{|b_k^2-|\xi|^2|},\,\,\, k=1,2,3,4,$$
it follows that for  $|\eta^2-4|\xi|^2| \ge r_0$ with $r_0\ll 1$,
\bma
\int_{|\eta^2-4|\xi|^2| \ge r_0}\|\xi^{\alpha}\tilde{Y}_2(t,\xi) \hat{V}_0\|^2_{\xi}d\xi
&\le  C\sum^4_{j=0}\int_{|\eta^2-4|\xi|^2| \ge r_0}(\xi^\alpha)^2e^{2{\rm Re}b_j(|\xi|)t } \|\hat{V}_0\|^2_{\xi}  \|\mathcal{X}_j(\xi)\|^4_{\xi}d\xi\nnm\\
&\le  C\int_{|\xi|\le r_0} e^{-c_1|\xi|^2t } (\xi^\alpha)^2\|\hat{U}_0\|^2 d\xi+ C\intr e^{-c_2t } (\xi^\alpha)^2\|\hat{U}_0\|^2 d\xi\nnm\\
&\leq C\(\int_{|\xi|\leq r_0} |\xi|^{2pm} e^{-c_1 p|\xi|^2t}d\xi\)^{1/p}\(\int_{|\xi|\leq r_0} \|\xi^{\alpha'}\hat{U}_0\|^{2p'}
d\xi\)^{1/p'} \nnm\\
&\quad+ C\intr e^{-c_2t } (\xi^\alpha)^2\|\hat{U}_0\|^2 d\xi\nnm\\
&\le C(1+t)^{-\frac32(\frac2{q}-1)-m} \(\|\dx^{\alpha'} U_0\|^2_{L^{q} }+\|\dxa U_0\|^2_{L^2}\), \label{t5a}
\ema
where
$1/q+1/p'=1$, $1/(2p')+1/q=1$,  $\alpha'\le \alpha$, $m=|\alpha-\alpha'|$ and $c_1,c_2>0$ are
constants.

For  $|\eta^2-4|\xi|^2| \le r_0$, it follows from \eqref{Y_2} that
\bma
\int_{|\eta^2-4|\xi|^2| \le r_0}\|\xi^{\alpha}\tilde{Y}_2(t,\xi) \hat{V}_0\|^2_{\xi}d\xi
&\le  C\int_{|\eta^2-4|\xi|^2| \le r_0}(\xi^\alpha)^2 e^{-\frac{\eta}{2}t } \|\hat{V}_0\|^2_{\xi}  d\xi\nnm\\
&\le Ce^{-\frac{\eta}{2}t } \|\dxa U_0\|^2_{L^2}. \label{t6a}
\ema
Combining \eqref{D1d}  and \eqref{t5a}--\eqref{t6a} gives \eqref{time5a}.

For $\hat V_0=(i (\hat{E}_0\cdot\xi)\chi_0, \omega\times \hat E_0,0)$, we obtain
\bmas
\tilde{Y}_2(t,\xi) \hat{V}_0&=e^{-\eta (1+|\xi|^2)t} (i (\hat{E}_0\cdot\xi)\chi_0, 0 ,0 ) \\
&\quad+e^{b_1t}\frac{b_1}{b_3-b_1}\sum^2_{k=1} (\omega\times \hat E_0,\omega\times e_k) (0,\omega\times e_k,\frac{i|\xi|e_k}{b_1})\\
&\quad+e^{b_3t}\frac{b_3}{b_1-b_3}\sum^4_{k=3} (\omega\times \hat E_0,\omega\times e_k) (0,\omega\times e_k,\frac{i|\xi|e_k}{b_3}).
\emas
Thus, it follows that for $|\eta^2-4|\xi|^2| \ge r_0$ with $r_0\ll 1$,
\bma
\int_{|\eta^2-4|\xi|^2| \ge r_0}\|\xi^{\alpha}\tilde{Y}_2(t,\xi) \hat{V}_0\|^2_{\xi} d\xi
&\le  C\intr e^{-2\eta (1+|\xi|^2)t}(\xi^\alpha)^2(1+|\xi|^2) |\hat{E}_0 |^2 d\xi \nnm\\
&\quad+  C\int_{|\xi|\le r_0} e^{-c_1|\xi|^2t } (\xi^\alpha)^2|\xi|^2|\hat{E}_0 |^2 d\xi+ C\intr e^{-c_2t } (\xi^\alpha)^2|\hat{E}_0 |^2 d\xi\nnm\\
&\le C\((1+t)^{-\frac32(\frac2{q}-1)-m-1}+t^{-1}e^{-\eta t}\) \(\|\dx^{\alpha'} E_0\|^2_{L^{q} }+\|\dxa E_0\|^2_{L^2}\).\label{t5b}
\ema
For  $|\eta^2-4|\xi|^2| \le r_0$, it follows from \eqref{Y_2} that
\be
\int_{|\eta^2-4|\xi|^2| \le r_0}\|\xi^{\alpha}\tilde{Y}_2(t,\xi) \hat{V}_0\|^2_{\xi} d\xi
 \le Ce^{-\frac{\eta}{2}t } \|\dxa E_0\|^2_{L^2}. \label{t6b}
\ee
Combining \eqref{D1d}  and \eqref{t5b}--\eqref{t6b} yields \eqref{time5b}.
And this completes the proof of the lemma.
\end{proof}

\subsection{Fluid approximation of $e^{\frac{t}{\eps^2}\BB_{\eps} }$}

The following preliminary lemma is for the study of the  fluid dynamic approximation of the semigroup $e^{\frac{t}{\eps^2}\BB_\eps}$.

\begin{lem} \label{S3a}
For any function $\phi(r)$ satisfying $| \phi^{(k)}(r)|\le C(1+|r|)^{-2-k-\delta}$ for any $\delta>0$  and $k=0,1$, we have
$$
\bigg|\intr e^{ix\cdot\xi} e^{i\vartheta |\xi|}\alpha(\omega)\phi(|\xi|)d\xi\bigg| \le C|\vartheta|^{-1},
$$
where $\alpha(\omega)$ is a smooth function for $\omega=\xi/|\xi|\in\S^2$.
\end{lem}

\begin{proof}
	Firstly, note that
$$
\intr e^{ix\cdot\xi} e^{i\vartheta |\xi|}\alpha(\omega)\phi(|\xi|)d\xi =\int^{\infty}_0e^{i\vartheta r}g(x,r)\phi(r)r^2 dr,
$$
where
$$  g(x,r)=  \int_{\S^2}e^{i r x\cdot\omega} \alpha(\omega) d\omega.
$$

For any function $\alpha(\omega)\in C^{\infty}(\S^2)$, by change of variable $\omega\to O_x\omega$,  where $O_x=(a_{ij})_{3\times 3}$ is an orthogonal matrix satisfying $O^T_xx=(0,0,|x|)$, we obtain
\bmas
g(x,r)&=\int^{2\pi}_0 d \varphi\int^{\pi}_0e^{i |x|r\cos\theta}\sin\theta\alpha(O_x\omega)  d\theta\\
&=-\frac1{i|x|r}\int^{2\pi}_0 d \varphi\int^{\pi}_0 \alpha(O_x\omega)  de^{i|x|r\cos\theta}\\
&=-\frac1{i|x|r}\int^{2\pi}_0\alpha(O_x\omega) e^{ir|x|\cos\theta}\bigg|^{\pi}_0d \varphi\\
&\quad+\frac1{i|x|r}\int^{2\pi}_0 d \varphi\int^{\pi}_0 e^{ir|x|\cos\theta}\nabla\alpha(O_x\omega)\cdot O_x\pt_\theta \omega  d\theta,
\emas
which gives
$$
|g(x,r)|\le C(1+|x|r)^{-1}.
$$
Similarly, we can prove
\bmas
|\pt_r g(x,r)|&=\bigg|\int^{2\pi}_0 d \varphi\int^{\pi}_0i |x| \cos\theta e^{i |x|r\cos\theta}\sin\theta\alpha(O_x\omega)  d\theta\bigg|\\
&\le C|x|(1+|x|r)^{-1}.
\emas
Thus
\bmas
\intr e^{ix\cdot\xi} e^{i\vartheta |\xi|} \alpha(\omega)\phi(|\xi|)d\xi &=\frac1{i\vartheta}\int^{\infty}_0 g(x,r) \phi(r)r^2 de^{i\vartheta r}\\
&=\frac1{i\vartheta}g(x,r) \phi(r)r^2 e^{i\vartheta r}\bigg|^{\infty}_0-\frac1{i\vartheta}\int^{\infty}_0 e^{i\vartheta r}\pt_r g(x,r) \phi(r)r^2 dr\\
&\quad-\frac1{i\vartheta}\int^{\infty}_0 e^{i\vartheta r}g(x,r) (\phi'(r)r^2+2\phi(r)r) dr,
\emas
which yields
$$
\bigg|\intr e^{ix\cdot\xi} e^{i\vartheta |\xi|}\alpha(\omega)\phi(|\xi|)d\xi\bigg|\le C|\vartheta|^{-1}.
$$
And this completes the proof of the lemma.
\end{proof}

\begin{lem}\label{limit5}
{\rm (1)} For any $\eps\ll 1$, any integer $k\ge 0$  and  $f_0\in L^2 $, it holds that
\bma
  &\Big\| P_{||} e^{\frac{t}{\eps^2}\BB_\eps}f_0  \Big\|_{W^{k,\infty} }
   \le C\bigg( \eps(1+t)^{-\frac52}+\(1+\frac{t}{\eps}\)^{-1}\bigg)(\| f_0\|_{H^{k+3} }+\| f_0\|_{W^{k+3,1}}), \label{limit4}
\\
  &\Big\| P_{\bot} e^{\frac{t}{\eps^2}\BB_\eps}f_0-Y_1(t)P_0f_0 \Big\|_{H^{k} }
  \le C\( \eps(1+t)^{-\frac54}+e^{-\frac{bt}{\eps^2}} \)(\| f_0\|_{H^{k+1} }+\| f_0\|_{L^{1}}), \label{limit4aa}
\ema
where $Y_1(t)$ is defined in \eqref{v2},  and $b>0$ is a constant given by \eqref{E_4}.  Moreover, if $f_0$ satisfies \eqref{initial}, i.e., $P_{||}f_0=P_1f_0=0,$ then
\bma
&\left\|P_{||} e^{\frac{t}{\eps^2}\BB_\eps}f_0 \right\|_{W^{k,\infty} }
 \le C\eps (1+t)^{-\frac52}\(\|f_0\|_{H^{k+3}}+\|f_0\|_{L^{1}}\),\label{limit3}\\
 &\left\|P_{\bot} e^{\frac{t}{\eps^2}\BB_\eps}f_0-Y_1(t)P_0f_0\right\|_{H^{k} }
 \le C\eps (1+t)^{-\frac54}\(\|f_0\|_{H^{k+1}}+\|f_0\|_{L^{1}}\). \label{limit4a}
\ema
{\rm (2)} For any $\eps\ll 1$, any integer $k\ge 0$ and $f_0\in L^2 $ satisfying $P_0f_0=0$, it holds that
\bma
&\left\|\frac1\eps P_{||} e^{\frac{t}{\eps^2}\BB_\eps}f_0  \right\|_{W^{k,\infty}}
 \le C\bigg(\eps (1+t)^{-\frac72}+\(1+\frac{t}{\eps}\)^{-1}+\frac1{\eps}e^{-\frac{bt}{\eps^2}}\bigg)\(\|f_0\|_{H^{k+4}}+\|f_0\|_{W^{k+4,1}}\), \label{limit4c}\\
&\left\|\frac1\eps P_{\bot} e^{\frac{t}{\eps^2}\BB_\eps}f_0-Y_1(t)Z_2 \right\|_{H^{k} }
 \le C\bigg(\eps (1+t)^{-\frac74}+\frac1{\eps}e^{-\frac{bt}{\eps^2}}\bigg)\(\|f_0\|_{H^{k+2}}+\|f_0\|_{L^{1}}\), \label{limit4d}
\ema where $Z_2=P_0(v\cdot\Tdx L^{-1}f_0)$.
\end{lem}

\begin{proof}
 Again we only prove the case when  $k=0$ because the proof for $k>0$ is similar. By \eqref{E_3} and     by taking $\eps\le r_0$ with $r_0>0$ given in Theorem \ref{spect3a}, we have
\bma
e^{\frac{t}{\eps^2}\BB_\eps}f_0-Y_1(t)P_0f_0 &= \intr e^{ix\cdot\xi}\(e^{\frac{t}{\eps^2}\BB_\eps(\xi)}\hat{f}_0-Y_1(t,\xi)P_0\hat{f}_0\) d\xi \nnm\\
&= \int_{|\xi|\le \frac{r_0}{\eps} }e^{ix\cdot\xi}\(S_4(t,\xi,\eps)\hat{f}_0-Y_1(t,\xi)P_0\hat{f}_0\)  d\xi \nnm\\
&\quad+\intr e^{ix\cdot\xi}S_5(t,\xi,\eps)\hat{f}_0  d\xi +\int_{|\xi|\ge \frac{r_0}{\eps}} e^{ix\cdot\xi}Y_1(t,\xi)P_0\hat{f}_0  d\xi \nnm\\
&=:I_1+I_2+I_3. \label{S_4e}
\ema

We estimate $I_j$, $j=1,2,3$ one by one as follows. Since
$$
S_4(t,\xi,\eps)\hat{f}_0=\sum^4_{j=0}e^{\frac{i\mu_j|\xi|}{\eps}t-a_j|\xi|^2t+ O(\eps|\xi|^3)t}\[\(P_0\hat{f}_0, h_j\) h_j+O(\eps|\xi|)\],
$$
it follows that
\bma
I_{1}&= \sum^3_{j=-1}\int_{ |\xi|\le \frac{r_0}{\eps} }e^{ix\cdot\xi}\bigg\{e^{ \frac{i\mu_j|\xi|}{\eps}t-a_j|\xi|^2t+ O(\eps|\xi|^3)t}\[\(P_0\hat{f}_0, h_j \) h_j+O(\eps|\xi|)\] \nnm\\
&\qquad-e^{\frac{i\mu_j|\xi|}{\eps}t-a_j|\xi|^2t }\(P_0\hat{f}_0, h_j\) h_j \bigg\}d\xi\nnm\\
&\quad+ \sum_{j=\pm1} \int_{ |\xi|\le \frac{r_0}{\eps} }e^{ix\cdot\xi}e^{\frac{i\mu_j|\xi|}{\eps}t-a_j|\xi|^2t }
\(P_0\hat{f}_0, h_j\) h_jd\xi \nnm\\
&=:I_{11}+I_{12}. \label{S4a}
\ema
For $I_{11}$, it holds that
\bma
\|P_{||}I_{11}\|_{L^{\infty}}&\le C\eps\int_{|\xi|\le \frac{r_0}{\eps}}e^{-c|\xi|^2t}\(|\xi|^3t\|P_0\hat{f}_0\| +|\xi|\|\hat{f}_0\|\) d\xi \nnm\\
&\le C\eps\sup_{|\xi|\le 1}\|\hat{f}_0\| \int_{|\xi|\le 1}e^{-c|\xi|^2t}(|\xi|^2t  +1)|\xi| d\xi \nnm\\
&\quad+C\eps \bigg(\int_{|\xi|> 1}e^{-c|\xi|^2t}\frac{(1+|\xi|^2t)^2}{ |\xi|^4}d\xi\bigg)^{1/2} \bigg(\int_{|\xi|> 1} |\xi|^6\|\hat{f}_0\|^2 d\xi\bigg)^{1/2} \nnm\\
&\le C\eps (1+t)^{-\frac52}\(\|f_0\|_{H^3}+\| f_0\|_{L^{1}}\), \label{S_3b-1}
\\
\|P_{\bot}I_{11}\|_{L^{2}}&\le C\eps\bigg(\int_{|\xi|\le \frac{r_0}{\eps}}e^{-2c |\xi|^2t}\(|\xi|^6t^{2}\|P_0\hat{f}_0\|^{2} +|\xi|^{2} \|\hat{f}_0\|^{2} \) d\xi\bigg)^{1/2} \nnm\\
&\le C\eps\sup_{|\xi|\le 1}\|\hat{f}_0\| \bigg(\int_{|\xi|\le 1}e^{-c |\xi|^2t} |\xi|^2d\xi\bigg)^{1/2}+C\eps \bigg(\int_{|\xi|> 1} e^{- c|\xi|^2 t} |\xi|^2\|\hat{f}_0\|^2 d\xi\bigg)^{1/2} \nnm\\
&\le C\eps (1+t)^{-\frac5{4}}\(\|f_0\|_{H^1}+\| f_0\|_{L^{1}}\). \label{S_3a}
\ema
To estimate  $I_{12}$, we first note that
\be P_{||}I_{12}=I_{12}, \quad P_{\bot}I_{12}=0. \label{S5c}\ee
 It is straightforward to verify that
\be
\|P_{||}I_{12}\|_{L^{\infty}}\le C\int_{|\xi|\le \frac{r_0}{\eps}}e^{-c|\xi|^2t} \|P_0\hat{f}_0\| d\xi\le C  \|P_0f_0\|_{H^2} . \label{S_3e}
\ee
Set
\bma I_{12}&=\sum_{j=\pm1} \(\intr -\int_{ |\xi|\ge \frac{r_0}{\eps} }\)e^{i x\cdot\xi } e^{\frac{\mu_j|\xi|}{\eps}t-a_j|\xi|^2t }\(P_0\hat{f}_0, h_j\) h_jd\xi \nnm\\
&=:I_{13}+I_{14}. \label{S4b}
\ema
For $I_{14}$,  it holds that
\bma
\|P_{||}I_{14}\|_{L^{\infty}}&\le  C\int_{ |\xi|\ge \frac{r_0}{\eps}} e^{-c|\xi|^2t}\|P_0\hat{f}_0\| d\xi \nnm\\
&\le C \(\int_{ |\xi|\ge \frac{r_0}{\eps}}e^{-\frac{2c r_0^2}{\eps^2}t}\frac{1}{(1+|\xi|^2)^2}d\xi\)^{1/2}
\(\int_{ |\xi|\le \frac{r_0}{\eps}}( 1+|\xi|^2)^2 \|\hat{f}_0\|^2 d\xi\)^{1/2} \nnm\\
&\le  Ce^{-\frac{c r_0^2}{\eps^2}t} \|f_0\|_{H^2} .\label{S_7-4}
\ema
Note that
\bmas
 (I_{13},v\chi_0)_{||}&=-\sum_{j=\pm 1} \frac12\intr e^{ix\cdot\xi}e^{\frac{i\mu_j|\xi|}{\eps}t-a_j|\xi|^2t }\bigg[j\bigg(\sqrt{\frac3{5}}\hat{n}_0 + \sqrt{\frac2{5}}\hat{q}_0\bigg)-(\hat{m}_0\cdot\omega)\bigg]\omega d\xi,\\
 (I_{13}, \tilde{h}_1)&=\sum_{j=\pm 1}\frac12 \intr e^{ix\cdot\xi}e^{\frac{i\mu_j|\xi|}{\eps}t-a_j|\xi|^2t }\bigg[\bigg(\sqrt{\frac3{5}}\hat{n}_0 + \sqrt{\frac2{5}}\hat{q}_0\bigg)-j(\hat{m}_0\cdot\omega)\bigg] d\xi,
\emas
where $$(n_0,m_0,q_0)=((f_0,\chi_0),(f_0,v\chi_0),(f_0,\chi_4)).$$
Set
$$
\left\{\bln
&G_{jk}(t,x)=\intr e^{i x\cdot\xi } e^{\frac{i\mu_j|\xi|}{\eps}t }(1+|\xi|)^{-3}\alpha_k(\omega)d\xi, \ \ k=0,1,2, \\
&H_j(t,x)=\intr e^{i x\cdot\xi }e^{ -a_j|\xi|^2t }d\xi=Ct^{-\frac32}e^{-\frac{|x|^2}{4a_jt}},\\
&\alpha_0(\omega)=1, \ \ \alpha_1(\omega)=\omega, \ \  \alpha_2(\omega)=\omega\otimes \omega.
\eln\right.
$$
Then by Lemma \ref{S3a}, we have
\bma
 \|P_{||}I_{13}\|_{L^{\infty}}&\le \|(I_{13},v\chi_0)_{||}\|_{L^{\infty}_x} + \|(I_{13}, \tilde{h}_1)\|_{L^{\infty}_x} \nnm\\
 &\le C\sum_{j=\pm 1}\sum^2_{k=0} \|G_{jk}(t)\|_{L^{\infty}_x}\|H_j(t)\|_{ L^1_x}\| (n_0+\sqrt{\frac23}q_0,m_0)\|_{W^{3,1}_x}\nnm\\
&\le  C\(\frac{t}{\eps}\)^{-1} \|P_0f_0\|_{W^{3,1}}.\label{S_4}
\ema
By combining \eqref{S_3e}--\eqref{S_4}, we obtain
\be
\|P_{||}I_{12}\|_{L^{\infty}} \le C \(1+\frac{t}{\eps}\)^{-1} \(\|P_0f_0\|_{H^2}+\|P_0 f_0\|_{W^{3,1}}\). \label{S4e}
\ee

Thus, it follows from \eqref{S4a}--\eqref{S5c} and \eqref{S4e} that
\bma
\|P_{||}I_{1}\|_{L^{\infty}}&\le C\bigg(\eps (1+t)^{-\frac52}+\(1+\frac{t}{\eps}\)^{-1}\bigg)\(\|f_0\|_{H^3}+\| f_0\|_{W^{3,1}}\), \label{S_3b}
\\
\|P_{\bot}I_{1}\|_{L^{2}}&\le C\eps (1+t)^{-\frac5{4}}\(\|f_0\|_{H^1}+\| f_0\|_{L^{1}}\). \label{S_3c}
\ema
$I_2$ and $I_3$ can be estimated directly as follows.
\bma
\|I_2\|^2_{L^2}&\le \intr \|S_5(t,\xi,\eps)\hat{f}_0 \|^2  d\xi\le C\intr e^{-\frac{2b }{\eps^2}t} \|\hat{f}_0\|^2d\xi\le Ce^{-\frac{2b }{\eps^2}t}\|f_0\|^2_{L^2}, \label{S5a}
\\
\|I_3\|^2_{L^2}&\le C\int_{ |\xi|\ge \frac{r_0}{\eps}} e^{-2c|\xi|^2t}\|P_0\hat{f}_0\|^2 d\xi\le C\eps^2\int_{ |\xi|\ge \frac{r_0}{\eps}} e^{-\frac{2c r_0^2}{\eps^2}t} |\xi|^2\|P_0\hat{f}_0\|^2d\xi\nnm\\
&\le C\eps^2e^{-\frac{2c r_0^2}{\eps^2}t}\|P_0f_0\|^2_{H^1}. \label{S5b}
\ema
By \eqref{S_3b}, \eqref{S_3c}, \eqref{S5a} and \eqref{S5b}, and the Sobolev embedding theorem, we  obtain \eqref{limit4} and \eqref{limit4aa}.

We now turn to \eqref{limit4a}. If $f_0$ satisfies \eqref{initial}, we have
$$
\(P_0\hat{f}_0,h_j(\xi)\) =0,\quad j=-1,1,
$$
which implies that $I_{12}=0$. 
By Lemma \ref{S2a}, we have
\bmas
\|I_2\|^2_{L^2}&\le C\int_{|\xi|\le \frac{r_0}{\eps}}\eps^2 |\xi|^2e^{-\frac{2bt}{\eps^2}}\| \hat{f}_0\|^2d\xi+C\int_{|\xi|\ge \frac{r_0}{\eps}}e^{-\frac{2bt}{\eps^2}}\| \hat{f}_0\|^2 d\xi\\
&\le C\eps^2 e^{-\frac{2bt}{\eps^2}} \|f_0\|^2_{H^1} .
\emas
Thus, \eqref{limit3} and \eqref{limit4a} hold.

Finally, we prove \eqref{limit4c} and \eqref{limit4d}. By \eqref{E_3}, we have
\bma
\frac1\eps e^{\frac{t}{\eps^2}\BB_\eps}f_0-Y_1(t)Z_2
&=  \int_{|\xi|\le \frac{r_0}{\eps} }e^{ix\cdot\xi}\(\frac1\eps S_4(t,\xi,\eps)\hat{f}_0-Y_1(t,\xi)(iv\cdot\xi L^{-1}\hat{f}_0)\)  d\xi \nnm\\
&\quad+\intr \frac1\eps e^{ix\cdot\xi}S_5(t,\xi,\eps)\hat{f}_0  d\xi +\int_{|\xi|\ge \frac{r_0}{\eps}}  e^{ix\cdot\xi}Y_1(t,\xi)(iv\cdot\xi L^{-1}\hat{f}_0)  d\xi \nnm\\
&=:I_4+I_5+I_6. \label{S_4e}
\ema
For any $f_0\in L^2 $ satisfying $P_0f_0=0$, we obtain
\be
S_4(t,\xi,\eps)\hat{f}_0=\eps\sum^4_{j=0}e^{\frac{i\mu_j|\xi|}{\eps}t-a_j|\xi|^2t+ O(\eps^3|\xi|^3)t}\[\(i(v\cdot\xi)L^{-1}\hat{f}_0, h_j\) h_j+O(\eps|\xi|^2)\].
\ee
Thus, it follows that
\bma
I_{4}&= \sum^3_{j=-1}\int_{ |\xi|\le \frac{r_0}{\eps} } e^{ix\cdot\xi}\bigg\{e^{\frac{i\mu_j|\xi|}{\eps}t-a_j|\xi|^2t+ O(\eps|\xi|^3)t}\[\(i(v\cdot\xi)L^{-1}\hat{f}_0, h_j\) h_j+O(\eps|\xi|^2)\] \nnm\\
&\qquad-e^{\frac{i\mu_j|\xi|}{\eps}t-a_j|\xi|^2t }\(i(v\cdot\xi)L^{-1}\hat{f}_0,  h_j\) h_j\bigg\} d\xi\nnm\\
&\quad+ \sum_{j=\pm1}\int_{ |\xi|\le \frac{r_0}{\eps} }e^{ix\cdot\xi}e^{\frac{i\mu_j|\xi|}{\eps}t-a_j|\xi|^2t }\(i(v\cdot\xi)L^{-1}\hat{f}_0,  h_j\) h_jd\xi\nnm\\
&=:I_{41}+I_{42}. \label{S4a1}
\ema
For $I_{41}$, it holds that
\bma
\|P_{||}I_{41}\|_{L^{\infty}}&\le C\eps\int_{|\xi|\le \frac{r_0}{\eps}}e^{-c|\xi|^2t}\(|\xi|^4t\|P_0\hat{f}_0\| +|\xi|^2\|\hat{f}_0\|\) d\xi \nnm\\
&\le C\eps (1+t)^{-\frac72}\(\|f_0\|_{H^4}+\| f_0\|_{L^{1}}\), \label{S_3b1}
\\
\|P_{\bot}I_{41}\|_{L^{2}}&\le C\eps\bigg(\int_{|\xi|\le \frac{r_0}{\eps}}e^{-2c |\xi|^2t}\(|\xi|^8t^{2}\|P_0\hat{f}_0\|^{2} +|\xi|^{4} \|\hat{f}_0\|^{2} \) d\xi\bigg)^{1/2} \nnm\\
&\le C\eps (1+t)^{-\frac7{4}}\(\|f_0\|_{H^2}+\| f_0\|_{L^{1}}\). \label{S_3a1}
\ema
Denote
\bma I_{42}&=\sum_{j=\pm1} \(\intr -\int_{ |\xi|\ge \frac{r_0}{\eps} }\)e^{i x\cdot\xi } e^{\frac{i\mu_j|\xi|}{\eps}t-a_j|\xi|^2t }\(i(v\cdot\xi)L^{-1}\hat{f}_0, h_j\) h_jd\xi \nnm\\
&=:I_{43}+I_{44}. \label{S4b1}
\ema
For $I_{44}$,  it holds that
\be
\|P_{||}I_{44}\|_{L^{\infty}}\le  C\int_{ |\xi|\ge \frac{r_0}{\eps}} e^{-c|\xi|^2t}\|P_0(v\cdot\xi L^{-1}\hat{f}_0)\| d\xi \le  Ce^{-\frac{c r_0^2}{\eps^2}t} \|f_0\|_{H^3} .\label{S_7a}
\ee
Note that
\bmas
 (I_{43},v\chi_0)_{||}&=-\sum_{j=\pm 1} \frac12\intr e^{ix\cdot\xi}e^{\frac{i\mu_j|\xi|}{\eps}t-a_j|\xi|^2t }\bigg[ j \sqrt{\frac2{5}}\hat{F}_4 -(\hat{F}_1\cdot\omega)\bigg]\omega,\\
 (I_{43}, \tilde{h}_1)&=\sum_{j=\pm 1}\frac12 \intr e^{ix\cdot\xi}e^{\frac{i\mu_j|\xi|}{\eps}t-a_j|\xi|^2t }\bigg[  \sqrt{\frac2{5}}\hat{F}_4 -j(\hat{F}_1\cdot\omega)\bigg],
\emas
where $$(F_1,F_4)=( (v\cdot\Tdx L^{-1}f_0,v\chi_0),(v\cdot\Tdx L^{-1}f_0,\chi_4)).$$
Hence, by Lemma \ref{S3a}, we have
\bma
 \|P_{||}I_{43}\|_{L^{\infty}}
 &\le C\sum_{j=\pm 1}\sum^2_{k=0} \|G_{jk}(t)\|_{L^{\infty}_x}\|H_j(t)\|_{ L^1_x}\|(F_1,F_4)\|_{W^{3,1}_x}\nnm\\
&\le  C \( \frac{t}{\eps}\)^{-1} \|f_0\|_{W^{4,1}}. \label{S_4c}
\ema
Thus, it follows from \eqref{S4a1}--\eqref{S_3a1} and \eqref{S4b1}--\eqref{S_4c} that
\bma
\|P_{||}I_{4}\|_{L^{\infty}}&\le C\bigg(\eps (1+t)^{-\frac52}+\(1+\frac{t}{\eps}\)^{-1}\bigg)\(\|f_0\|_{H^4}+\| f_0\|_{W^{4,1}}\), \label{S_3b1}
\\
\|P_{\bot}I_{4}\|_{L^{2}}&\le C\eps (1+t)^{-\frac5{4}}\(\|f_0\|_{H^2}+\| f_0\|_{L^{1}}\). \label{S_3c1}
\ema

Finally, $I_5$ and $I_6$ can be estimated directly as follows.
\bma
\|I_5\|^2_{L^2}&\le\intr \frac1{\eps^2}\|S_5(t,\xi,\eps)\hat{f}_0 \|^2  d\xi\le C\intr \frac{1}{\eps^2}e^{-\frac{2b }{\eps^2}t} \|\hat{f}_0\|^2d\xi\le C \frac1{\eps^2}e^{-\frac{2b }{\eps^2}t}\|f_0\|^2_{L^2}, \label{S5a1}
\\
\|I_6\|^2_{L^2}&\le C\int_{ |\xi|\ge \frac{r_0}{\eps}} e^{-2c|\xi|^2t}\|P_0(v\cdot\xi L^{-1}\hat{f}_0)\|^2 d\xi\le C\eps^2\int_{ |\xi|\ge \frac{r_0}{\eps}} e^{-\frac{2c r_0^2}{\eps^2}t} |\xi|^4\| \hat{f}_0\|^2d\xi\nnm\\
&\le C\eps^2e^{-\frac{2c r_0^2}{\eps^2}t}\|f_0\|^2_{H^2}. \label{S5b1}
\ema
By \eqref{S_3b1}, \eqref{S_3c1}, \eqref{S5a1} and \eqref{S5b1}, and the Sobolev embedding theorem, we  obtain \eqref{limit4c} and \eqref{limit4d}. And this completes the proof of the lemma.
\end{proof}


\begin{lem}\label{time6a}
For any $\eps\in (0,1)$, $\alpha\in \mathbb{N}^3$ and any $f_0\in L^2$, we have
\bma
 \left\|P_0\dxa e^{\frac{t}{\eps^2}\BB_\eps}f_0 \right\|_{L^{2}}
 &\le C (1+t)^{-\frac34-\frac{m}2}\(\|\dxa f_0\|_{L^{2} }+\|\dx^{\alpha'}f_0\|_{L^{1} }\), \label{time1}
\\
 \left\|P_1\dxa e^{\frac{t}{\eps^2}\BB_\eps}f_0 \right\|_{L^{2}}  &\le C\(\eps (1+t)^{-\frac54-\frac{m}2} + e^{-\frac{bt}{\eps^2}}\)\(\|\dxa f_0\|_{H^{1} }+\|\dx^{\alpha'}f_0\|_{L^{1} }\),  \label{time2}
\ema
where $\alpha'\le \alpha$, $m=|\alpha-\alpha'|$, and $b>0$ is a constant given by \eqref{E_4}. 
Moreover,  if $P_0f_0=0$, then
\bma
 \left\|P_0\dxa e^{\frac{t}{\eps^2}\BB_\eps}f_0 \right\|_{L^{2} }
 &\le C\( \eps(1+t)^{-\frac54-\frac{m}2}+ e^{-\frac{bt}{\eps^2}}\)\(\|\dxa f_0\|_{H^{1} }+\|\dx^{\alpha'}f_0\|_{L^{1} }\),   \label{time5}
\\
 \left\|P_1\dxa e^{\frac{t}{\eps^2}\BB_\eps}f_0 \right\|_{L^{2}}  &\le C\(\eps^2 (1+t)^{-\frac74-\frac{m}2} + e^{-\frac{bt}{\eps^2}}\)\(\|\dxa f_0\|_{H^{2} }+\|\dx^{\alpha'}f_0\|_{L^{1} }\) . \label{time6}
\ema
\end{lem}

\begin{proof}
By  \eqref{E_3},  for $j=0,1$ we have
\bma
\left\|P_j\dxa e^{\frac{t}{\eps^2}\BB_\eps}f_0\right\|^2_{L^{2} }
=& \intr  \left\|P_j\xi^{\alpha}e^{\frac{t}{\eps^2}\BB_{\eps}(\xi)}\hat{f}_0 \right\|^2 d\xi  \nnm\\
\le &  \int_{|\xi|\le \frac{r_0}{\eps}} \left\|\xi^{\alpha}P_jS_4(t,\xi,\eps)\hat{f}_0\right\|^2 d\xi +\intr \left\|\xi^{\alpha}S_5(t,\xi,\eps)\hat{f}_0\right\|^2 d\xi. \label{D1a}
\ema
 By \eqref{E_4},  we can estimate the second term on the right hand side of \eqref{D1a} as follows.
\be
\intr (\xi^\alpha)^2 \|S_5(t,\xi,\eps)\hat{f}_0\|^2 d\xi\le C e^{-2\frac{bt}{\eps^2}}\intr(\xi^\alpha)^2\|\hat{f}_0\|^2 d\xi \le Ce^{-2\frac{bt}{\eps^2}} \|\dxa f_0\|^2_{L^{2}} .\label{t3b}
\ee

By Theorems \ref{rate2} and \ref{spect3a}, we have
\bma
\int_{|\xi|\le \frac{r_0}{\eps}}\|\xi^{\alpha}P_0S_4(t,\xi,\eps)\hat{f}_0\|^2 d\xi
&\le  C\int_{|\xi|\le \frac{r_0}{\eps}} e^{-c|\xi|^2t}(\xi^{\alpha})^2\|\hat{f}_0\|^2 d\xi\nnm\\
&\le C(1+t)^{-\frac32-m}\(\|\dxa f_0\|^2_{L^{2} }+\|\dx^{\alpha'}f_0\|^2_{L^{1} }\), \label{t1}
\\
\int_{|\xi|\le \frac{r_0}{\eps}}\|\xi^{\alpha}P_1S_4(t,\xi,\eps)\hat{f}_0\|^2d\xi
&\le  C\int_{|\xi|\le \frac{r_0}{\eps}} e^{-c|\xi|^2t} \eps^2|\xi|^{2 }(\xi^{\alpha})^2\|\hat{f}_0\|^2 d\xi\nnm\\
&\le C\eps^2(1+t)^{-\frac52-m}\(\|\dxa f_0\|^2_{H^{1} }+\|\dx^{\alpha'}f_0\|^2_{L^{1} }\),\label{t2}
\ema
where $\alpha'\le \alpha$ and $m=|\alpha-\alpha'|$. Combining \eqref{D1a}--\eqref{t2} yields \eqref{time1} and \eqref{time2}.

Moreover, it holds that  for $P_0f_0=0$,
\bma
\int_{|\xi|\le \frac{r_0}{\eps}}\|\xi^{\alpha}P_0S_4(t,\xi,\eps)\hat{f}_0\|^2  d\xi
&\le  C\int_{|\xi|\le \frac{r_0}{\eps}} e^{-c|\xi|^2t} \eps^2|\xi|^{2 }(\xi^{\alpha})^2\|\hat{f}_0\|^2 d\xi\nnm\\
&\le C\eps^2(1+t)^{-\frac52-m}\(\|\dxa f_0\|^2_{H^{1} }+\|\dx^{\alpha'}f_0\|^2_{L^{1} }\),\label{t5}
\\
\int_{|\xi|\le \frac{r_0}{\eps}}\|\xi^{\alpha}P_1S_4(t,\xi,\eps)\hat{f}_0\|^2  d\xi
&\le  C\int_{|\xi|\le \frac{r_0}{\eps}} e^{-c|\xi|^2t} \eps^4|\xi|^{4 }(\xi^{\alpha})^2\|\hat{f}_0\|^2 d\xi\nnm\\
&\le C\eps^4(1+t)^{-\frac72-m}\(\|\dxa f_0\|^2_{H^{2} }+\|\dx^{\alpha'}f_0\|^2_{L^{1} }\).\label{t6}
\ema
Combining \eqref{D1a}, \eqref{t5} and \eqref{t6} gives  \eqref{time5}--\eqref{time6}. The proof of the lemma  is completed.
\end{proof}

\begin{lem} \label{timev}
For any $1\le q\le 2$,  $\alpha\in \mathbb{N}^3$ and any $u_0\in N_0$, we have
\be
\|\dxa Y_1(t)u_0\|_{L^{2}} \le C (1+t)^{-\frac32(\frac1q-\frac12)-\frac{m}2} \|\dx^{\alpha'}u_0\|_{L^{q}}+Ct^{-\frac{k}{2}}e^{-ct}\|\dx^{\alpha''} u_0\|_{L^{2}}, \label{time1b}
\ee
where $\alpha',\alpha''\le \alpha$, $m=|\alpha-\alpha'|$ and $k=|\alpha-\alpha''|$.
\end{lem}

\begin{proof}
By \eqref{v1}, we have
\bmas
 \left\|\dxa Y_1(t)u_0\right\|^2_{L^{2}} &\le C \(\int_{|\xi|\leq 1}+\int_{|\xi|\ge 1}\) (\xi^{\alpha})^2e^{-2c|\xi|^2t}\|\hat{u}_0\|^2d\xi \\
 &\le  C\(\int_{|\xi|\leq 1} |\xi|^{2pm} e^{-2c p|\xi|^2t}d\xi\)^{1/p}\(\int_{|\xi|\leq 1} \|\xi^{\alpha'}\hat{u}_0\|^{2p'}
d\xi\)^{1/p'}  \\
 &\quad+\sup_{|\xi|\ge 1} \(|\xi|^{2k}e^{-2c|\xi|^2t}\)\int_{|\xi|\ge 1} (\xi^{\alpha''})^2\|\hat{u}_0\|^2d\xi \\
 &\le C(1+t)^{-\frac32(\frac2q-1)-m}  \|\dx^{\alpha'} u_0\|^2_{L^{q}}+Ct^{-k}e^{-2ct}\|\dx^{\alpha''}u_0\|^2_{L^{2}} ,
\emas
where $1/q+1/p'=1$, $1/(2p')+1/q=1$,  $\alpha'\le \alpha$ and $m=|\alpha-\alpha'|$. This completes the proof of the lemma.
\end{proof}

\section{Diffusion limit}
\setcounter{equation}{0}
\label{sect4}
In this section, we will study the diffusion limit of  the nonlinear VMB system \eqref{VMB4}--\eqref{VMB4d} based on the fluid approximations of  the linear VMB system given in Section 3.

Since the operators $\AA_{\eps}$ and $\BB_{\eps}$ generate contraction semigroups in $H^k$, the solution $U_{\eps}(t)=(f_{\eps},V_{\eps} )(t)$ with $V_{\eps}(t)=( g_{\eps},E_{\eps},B_{\eps})(t)$ to the VMB system \eqref{VMB4}--\eqref{VMB2i} can be represented by
\bma
f_{\eps}(t)&=e^{\frac{t}{\eps^2}\BB_{\eps}}f_0+\intt e^{\frac{t-s}{\eps^2}\BB_{\eps}}\(G_{1}(s)+\frac1{\eps}G_{2}(s)\) ds, \label{fe}
\\
V_{\eps}(t)&=e^{\frac{t}{\eps^2}\AA_{\eps}}V_0+\intt e^{\frac{t-s}{\eps^2}\AA_{\eps}} \(G_{3}(s)+\frac1{\eps}G_{4}(s)\)  ds, \label{ue}
\ema
where $V_0=( g_{0},E_{0},B_{0})$, and the nonlinear terms $G_k$, $k=1,2,3,4$ are given by
\bmas
&G_{1}=\frac12v\cdot E_{\eps}g_{\eps}-E_{\eps}\cdot\Tdv g_{\eps}-\frac1{\eps}P_0(v\times B_{\eps})\cdot\Tdv P_rg_{\eps},\\
&G_{2}=-P_1(v\times B_{\eps})\cdot\Tdv P_rg_{\eps}+\Gamma(f_{\eps},f_{\eps}),\\
&G_{3}=(G_{31},0,0), \quad G_{31}=\frac12v\cdot E_{\eps}f_{\eps}-E_{\eps}\cdot\Tdv f_{\eps}, \\
&G_{4}=(G_{41},0,0), \quad G_{41}=-(v\times B_{\eps})\cdot\Tdv f_{\eps}+\Gamma(g_{\eps},f_{\eps}).
\emas
 Also, the solution $U(t)=(u_1,V_1 )(t)$ with $u_1 =n\chi_0+m\cdot v\chi_0+q\chi_4$ and $V_1 =(\rho \chi_0,E,B)$ to the NSMF system  \eqref{NSM_2}--\eqref{NSP_5i} can be represented by
\bma
u_1(t)&=Y_1(t)P_0f_0+\intt Y_1(t-s)(H_1(s)+\Tdx\cdot H_2(s)) ds,\label{fe1}
\\
V_1(t)&=Y_2(t)P_2V_0+\intt Y_2(t-s) H_3(s)ds,\label{ue1}
\ema
where
$$
\left.\bln
&H_1=(\rho E+j\times B)\cdot v\chi_0, \quad H_2=-(m\otimes m)\cdot v\chi_0-\frac53 (qm)\chi_4,\\
&H_3=-(\Tdx\cdot H_4\chi_0, H_4,0), \quad H_4=\rho m-\eta m\times B.
\eln\right.
$$


\subsection{Energy estimate}

We first derive some energy estimates.
Let $N\ge 1$ be a positive integer and $U_{\eps}=(f_{\eps},g_{\eps},E_{\eps},B_{\eps})$, and
\bma
E_{N,k}(U_{\eps})&=\sum_{|\alpha|+|\beta|\le N}\|\nu^k\dxa\dvb (f_{\eps},g_{\eps})\|^2_{L^2 }+\sum_{|\alpha|\le N}\|\dxa(E_{\eps},B_{\eps})\|^2_{L^2_x},\label{energy3}
\\
D_{N,k}(U_{\eps})&=\sum_{|\alpha|+|\beta|\le N}\frac1{\eps^2}\|\nu^{\frac12+k}\dxa\dvb  ( P_1f_{\eps},P_rg_{\eps})\|^2_{L^2 }+\sum_{1\le |\alpha|\le N-1}\|\dxa(E_{\eps},B_{\eps})\|^2_{L^2_x}\nnm\\
&\quad+\sum_{|\alpha|\le N-1}\|\dxa\Tdx  ( P_0f_{\eps},P_{d}g_{\eps})\|^2_{L^2 } +\| E_{\eps}\|^2_{L^2_x} ,
\ema
for $k\ge 0$. For brevity, we denote $E_N(U_{\eps})=E_{N,0}(U_{\eps})$   and $D_N(U_{\eps})=D_{N,0}(U_{\eps})$.

Firstly, by taking the inner product of $\chi_j\ (j=0,1,2,3,4)$ and \eqref{VMB4}, we have   a  compressible Euler-Maxwell type system
\bma
&\dt n_{\eps}+\frac{1}{\eps}\divx  m_{\eps}=0,\label{G_3}\\
&\dt  m_{\eps}+\frac{1}{\eps}\Tdx n_{\eps}+\frac{1}{\eps}\sqrt{\frac23}\Tdx q_{\eps}=\rho_{\eps} E_{\eps}+\frac{1}{\eps}u_{\eps}\times B_{\eps}-\frac{1}{\eps}( v\cdot\Tdx( P_1f_{\eps}), v\chi_0),\label{G_5}\\
&\dt q_{\eps}+\frac{1}{\eps}\sqrt{\frac23}\divx m_{\eps}=\sqrt{\frac23} E_{\eps}\cdot u_{\eps}-\frac{1}{\eps}(v\cdot\Tdx( P_1f_{\eps}), \chi_4 ), \label{G_6}
\ema
where
\be
(n_{\eps},m_{\eps},q_{\eps})=((f_{\eps},\chi_0),(f_{\eps},v\chi_0),(f_{\eps},\chi_4)),\quad (\rho_{\eps},u_{\eps})=((g_{\eps},\chi_0),(g_{\eps},v\chi_0)). \label{macro}
\ee
Taking the microscopic projection $ P_1$ on \eqref{VMB4} gives
\be
\dt( P_1f_{\eps})+ \frac{1}{\eps}P_1(v\cdot\Tdx  P_1f_{\eps})-\frac{1}{\eps^2}L( P_1f_{\eps})=- \frac{1}{\eps}P_1(v\cdot\Tdx P_0f_{\eps})+ P_1  G_{1}- \frac{1}{\eps}P_1G_{2} .\label{GG1}
\ee

By \eqref{GG1}, we can express the microscopic part $ P_1f_{\eps}$ as
\bq   \frac{1}{\eps}P_1f_{\eps}= L^{-1}[\eps\dt( P_1f_{\eps})+P_1(v\cdot\Tdx  P_1f_{\eps})-\eps P_1  G_{1}-P_1  G_{2}]+ L^{-1} P_1(v\cdot\Tdx P_0f_{\eps}). \label{p_1}\eq
By substituting \eqref{p_1} into \eqref{G_3}--\eqref{G_6}, we obtain
a compressible Navier-Stokes-Maxwell type system
\bma
&\dt n_{\eps}+\frac{1}{\eps}\divx  m_{\eps}=0,\label{G_9}\\
&\dt  m_{\eps}+\eps\dt R_1+\frac{1}{\eps}\Tdx n_{\eps}+\frac{1}{\eps}\sqrt{\frac23}\Tdx q_{\eps}\nnm\\
&=\kappa_0 \(\Delta_x m_{\eps}+\frac13\Tdx{\rm div}_x m_{\eps}\)+\rho_{\eps} E_{\eps}+\frac{1}{\eps}u_{\eps}\times B_{\eps}+R_3,\label{G_7}\\
&\dt q_{\eps}+\eps\dt R_2+\frac{1}{\eps}\sqrt{\frac23}\divx m_{\eps}=\kappa_1 \Delta_x q_{\eps}+\sqrt{\frac23} E_{\eps}\cdot u_{\eps}+R_4,\label{G_8}
\ema
where the remainder terms $R_1, R_2, R_3, R_4$ are given by
 \bmas
 R_1&=( v\cdot\Tdx L^{-1}( P_1f_{\eps}),v\chi_0), \ \ R_2=( v\cdot\Tdx L^{-1}( P_1f_{\eps}),\chi_4),\\
 R_3&=( v\cdot\Tdx L^{-1}[ P_1(v\cdot\Tdx  P_1f_{\eps})-\eps P_1  G_{1}-P_1  G_{2}],v\chi_0),\\
 R_4&=( v\cdot\Tdx L^{-1}[ P_1(v\cdot\Tdx  P_1f_{\eps})-\eps P_1  G_{1}-P_1  G_{2}],\chi_4).
 \emas

By taking the inner product between $\sqrt M$ and \eqref{VMB4a}, we obtain
\be
\dt \rho_{\eps}+\frac{1}{\eps}\divx u_{\eps}=0. \label{G_3a}
\ee
Taking the microscopic projection $ P_r$ on \eqref{VMB4a} gives
\bma
&\dt( P_rg_{\eps})+ \frac{1}{\eps}P_r(v\cdot\Tdx  P_rg_{\eps})-\frac{1}{\eps}v \sqrt{M}\cdot E_{\eps}-\frac{1}{\eps^2}L_1( P_rg_{\eps})\nnm\\
&=-\frac{1}{\eps} (v\cdot\Tdx \rho_{\eps})\sqrt{M}+ G_{31}+\frac{1}{\eps}G_{41} .\label{GG2}
\ema
By \eqref{GG2}, we can express the microscopic part $ P_rg_{\eps}$ as
\bq \frac{1}{\eps}P_rg_{\eps}= L_1^{-1}[\eps\dt( P_rg_{\eps})+ P_r(v\cdot\Tdx  P_rg_{\eps})-\eps G_{31}-G_{41}]- L_1^{-1}[v\chi_0\cdot (E_{\eps}-\Tdx \rho_{\eps})] . \label{p_c}\eq
Substituting \eqref{p_c} into \eqref{G_3a} and \eqref{VMB4b} gives
\bma
\dt \rho_{\eps}+\eps\dt\divx R_5&=-\eta \rho_{\eps}+\eta \Delta_x \rho_{\eps}-\divx R_6,\label{G_9a}\\
\dt E_{\eps}+\eps\dt R_5&=\Tdx\times B_{\eps}+\eta\Tdx \rho_{\eps}-\eta E_{\eps}+R_6,\label{G_9b}\\
\dt B_{\eps}&=-\Tdx\times E_{\eps},\label{G_9c}
\ema
where the remainder terms $R_5, R_6$ are defined by
\bmas
R_5=(L_1^{-1}P_r g_{\eps},v\chi_0), \quad
R_6=(L_1^{-1}(P_r(v\cdot\Tdx P_rg_{\eps})-\eps G_{31}-G_{41}),v\chi_0).
\emas

Similar to \cite{Duan4,Li4,Strain}, we have the existence and the energy estimate for the solution $U_{\eps}=(f_{\eps},g_{\eps},E_{\eps},B_{\eps})$ to the VMB system \eqref{VMB4}--\eqref{VMB2i}.

\begin{lem}[Macroscopic dissipation] \label{macro-en} Given $N\ge 2$.  Let $(n_{\eps},m_{\eps},q_{\eps})$ and $(\rho_{\eps},E_{\eps},B_{\eps})$ be the strong solutions to \eqref{G_9}--\eqref{G_8} and \eqref{G_9a}--\eqref{G_9c} respectively. Then, there are two constants $s_0,s_1>0$ such that for any $\eps\in (0,1)$,
\bmas
&\Dt \sum_{|\alpha|\le N-1}s_0\(\|\dxa(n_{\eps}, m_{\eps},q_{\eps})\|^2_{L^2_x}+2\eps\intr \dxa R_1\dxa m_{\eps}dx+2\eps\intr \dxa R_2\dxa q_{\eps}dx\)\nnm\\
&+\Dt \sum_{|\alpha|\le N-1}4\eps\intr \dxa m_{\eps} \dxa\Tdx n_{\eps}dx+\sum_{|\alpha|\le N-1} \|\dxa\Tdx (n_{\eps}, m_{\eps},q_{\eps})\|^2_{L^2_x}
\nnm\\
\le & C\sqrt{E_N(U_{\eps})}D_N(U_{\eps})+C\sum_{|\alpha|\le N-1}\|\dxa\Tdx P_1f_{\eps}\|^2_{L^2 },
\\
&\Dt \sum_{|\alpha|\le N-1}s_1\(\|\dxa (\rho_{\eps},E_{\eps},B_{\eps})\|^2_{L^2_x}+2\eps \intr\dxa\divx R_5\dxa \rho_{\eps} dx+2\eps \intr\dxa R_5\dxa E_{\eps} dx\)\nnm\\
&-\Dt \sum_{|\alpha|\le N-2}4\intr \dxa E_{\eps}\dxa(\Tdx\times B_{\eps})dx\nnm\\
&+\sum_{|\alpha|\le N-1}(\|\dxa \rho_{\eps}\|^2_{L^2_x}+\|\dxa \Tdx \rho_{\eps}\|^2_{L^2_x} +\|\dxa E_{\eps}\|^2_{L^2_x})+\sum_{1\le |\alpha|\le N-1}\|\dxa B_{\eps}\|^2_{L^2_x}
\nnm\\
\le& CE_N(U_{\eps})D_N(U_{\eps})+\frac{C}{\eps^2}\sum_{|\alpha|\le N}\|\dxa  P_rg_{\eps}\|^2_{L^2 }.
\emas
\end{lem}


\begin{lem}[Microscopic dissipation]
\label{micro-en}
 Given $N\ge 2$. Let $U_{\eps}=(f_{\eps},g_{\eps},E_{\eps},B_{\eps})$ be a strong solution to VMB system  \eqref{VMB4}--\eqref{VMB2i}.
Then, there are constants $p_k>0$, $1\le k\le N$ and $\mu_1>0$ such that for any $\eps\in (0,1)$,
\bmas
&\frac12\Dt \sum_{|\alpha|\le N} (\|\dxa (f_{\eps},g_{\eps})\|^2_{L^2 }+\|\dxa (E_{\eps},B_{\eps})\|^2_{L^2_x})+\frac{\mu_1}{\eps^2} \sum_{|\alpha|\le N}\|\nu^{\frac12} \dxa (P_1f_{\eps},P_rg_{\eps})\|^2_{L^2 }\nnm\\
&\le C\sqrt{E_N(U_{\eps})}D_N(U_{\eps}), 
\\
&\Dt \sum_{1\le k\le N}p_k\sum_{|\beta|=k \atop |\alpha|+|\beta|\le N}\|\dxa\dvb (P_1f_{\eps},P_rg_{\eps})\|^2_{L^2 } +\frac{\mu_1}{\eps^2} \sum_{|\beta|\ge 1 \atop |\alpha|+|\beta|\le N}\|\nu^{\frac12}\dxa\dvb (P_1f_{\eps},P_rg_{\eps})\|^2_{L^2 }\nnm\\
&\le C\sum_{|\alpha|\le N-1}(\|\dxa\Tdx ( f_{\eps}, g_{\eps})\|^2_{L^2 }+\|\dxa E_{\eps}\|^2_{L^2_x}) +C\sqrt{E_N(U_{\eps})}D_N(U_{\eps}). 
\emas
\end{lem}

\begin{lem}\label{energy1}  Let $N\ge 2$. For any $\eps\in (0,1)$, there exists a small constant $\delta_0>0$ and an energy functional $\mathcal{E}_{N}(U_{\eps})\sim E_{N}(U_{\eps})$ such that if the initial data $U_0=(f_{0},g_{0},E_{0},B_{0})$ satisfies  $E_N(U_0)\le \delta_0^2$, then the  system \eqref{VMB4}--\eqref{VMB2i} admits a unique global solution $U_{\eps}=(f_{\eps},g_{\eps},E_{\eps},B_{\eps})$ satisfying
\be
   \Dt \mathcal{E}_N(U_{\eps}(t))+  D_N(U_{\eps}(t)) \le 0. \label{G_1}
\ee

Moreover, there exists an energy functional $\mathcal{E}_{N,1}(U_{\eps})\sim E_{N,1}(U_{\eps})$  such that if the initial data $U_0$ satisfies $E_{N,1}(U_0)\le \delta_0^2$, then
\be
   \Dt \mathcal{E}_{N,1}(U_{\eps}(t))+  D_{N,1}(U_{\eps}(t)) \le 0. \label{G_1b}
\ee
\end{lem}

With Lemmas \ref{time}, \ref{time6a} and \ref{energy1}, we have the optimal  time decay rate of $U_{\eps}=(f_{\eps},g_{\eps},E_{\eps},B_{\eps})$ stated in the following lemma.

\begin{lem}\label{time7} Let $N\ge 3$. For any $\eps\in (0,1)$,  there exists a small constant $\delta_0>0$ such that if the initial data $U_0=(f_{0},g_{0},E_{0},B_{0})$ satisfies that $E_{N+2,1}(U_0)+\|U_0\|^2_{L^1 }\le \delta_0^2$, then the solution
   $U_{\eps}(t)=(f_{\eps},g_{\eps},E_{\eps},B_{\eps})$ to the system \eqref{VMB4}--\eqref{VMB2i} has the following time-decay rate estimates:
\be
 \|(f_{\eps},g_{\eps})\|_{D^N_1}+\| (E_{\eps},B_{\eps})\|_{H^N_x}\le C\delta_0 (1+t)^{-\frac34}. \label{G_11}
\ee
In particular, we have
\bma
\|P_1f_{\eps}(t)\|_{H^{N-3}}\le C\delta_0\(\eps(1+t)^{-\frac54}+e^{-\frac{bt}{4\eps^2}}\),\label{G_5}\\
\|P_rg_{\eps}(t)\|_{H^{N-3}}\le C\delta_0\(\eps(1+t)^{-\frac34} +e^{-\frac{bt}{4\eps^2}}\),\label{G_5a}
\ema
where $b>0$ is a constant.
\end{lem}
\begin{proof}
Define
\bmas
Q_{\eps}(t)= \sup_{0\le s\le t}\Big\{(1+s)^{\frac34}E_{N,1}(U_{\eps}(s))^{\frac12}
 +\|P_rg_{\eps}(s)\|_{H^{N-3}} \(\eps(1+t)^{-\frac34} +e^{-\frac{bt}{4\eps^2}}\)^{-1} \Big\}.
\emas
We claim that
  \be
 Q_{\eps}(t)\le C\delta_0.  \label{assume}
 \ee
It is straightforward  to check that the estimates \eqref{G_11}  and \eqref{G_5a} follow  from \eqref{assume}.

Since $P_0G_{2} =0$, it follows from Lemma \ref{time6a} and \eqref{ue}  that
 \bma
 \| P_0f_{\eps} (t)\|_{L^2}
 &\le
 C(1+t)^{-\frac34}\(\| f_0\|_{L^2 }+\|f_0\|_{L^1 }\)
 \nnm\\
&\quad+C \intt (1+t-s)^{-\frac34}\(\| G_{1}(s)\|_{L^2 }
  +\| G_{1}(s)\|_{L^1 }\)ds
  \nnm\\
  &\quad+C \intt \((1+t-s)^{-\frac54}+\frac1{\eps}e^{-\frac{b(t-s)}{\eps^2}}\)\nnm\\
&\qquad\qquad\times\(\| G_{2}(s)\|_{H^1 }
  +\| G_{2}(s)\|_{L^1 }\)ds
  \nnm\\
&\le
 C\delta_0(1+t)^{-\frac34}+CQ_{\eps}(t)^2 (1+t)^{-\frac34},   \label{density_1}
 \ema
 where we have used
\bma
\| G_{1}(s)\|_{H^{N-1} }+\| G_{1}(s)\|_{L^{1} }&\le CQ_{\eps}(t)^2 \[(1+s)^{-\frac32}+\frac{1}{\eps}(1+s)^{-\frac34}e^{-\frac{bs}{4\eps^2}}\]  ,  \label{h1}\\
\| G_{2}(s)\|_{H^{N-1} }+\| G_{2}(s)\|_{L^{1} }&\le CQ_{\eps}(t)^2 (1+s)^{-\frac32}. \label{h2}
\ema
Let $V_{\eps}=(g_{\eps},E_{\eps},B_{\eps})$ and $V_{0}=(g_{0},E_{0},B_{0})$. Since  $P_2G_{31} =P_2G_{41} =0$,
it follows from Lemma \ref{time} and \eqref{fe}  that
  \bma
\| V_{\eps} (t)\|_{L^2 }&\le
C(1+t)^{-\frac34}\(\| V_0\|_{Z^{2} }+\|V_0\|_{L^{1} }+\|\Tdx^2 V_0\|_{Z^{2} }\) \nnm\\
&\quad+C\sum^4_{k=3} \int^{t}_0 \( (1+t-s)^{-\frac34}+\frac1{\eps}e^{-\frac{b(t-s)}{\eps^2}}\)\nnm\\
&\qquad\times\(\| G_k(s)\|_{H^{1} }+\|G_k(s)\|_{L^{1} }+\|\Tdx^2 G_k(s)\|_{Z^{2} }\)ds\nnm\\
&\le
  C\delta_0(1+t)^{-\frac34}+CQ_{\eps}(t)^2 (1+t)^{-\frac34},   \label{density_1a}
 \ema
 where we have used
\be
\| G_k(s)\|_{H^{N-1} }+\| G_k(s)\|_{L^{1} } \le CQ_{\eps}(t)^2 (1+s)^{-\frac32} , \quad k=3,4.\label{h3}
\ee
 Let $1<l< 2$ and $n\ge 2$. Multiplying \eqref{G_1b} by $(1 + t)^l$ and
then taking time integration over $[0, t]$ gives
\bma
&(1+t)^lE_{n,1}(U_{\eps})(t)+ \intt (1+s)^lD_{n,1}(U_{\eps})(s)ds
\nnm\\
\le &CE_{n,1}(U_0)+Cl\intt (1+s)^{l-1}E_{n,1}(U_{\eps})(s)ds\nnm\\
\le &CE_{n,1}(U_0)+Cl\intt (1+s)^{l-1}D_{n+1,1}(U_{\eps})(s)ds\nnm\\
&+Cl\intt (1+s)^{l-1}(\|P_0f_{\eps}(s)\|^2_{L^2}+\|B_{\eps}(s)\|^2_{L^2_x})ds, \label{l5}
\ema
where we have used
$$E_{n,1}(U_{\eps})\le CD_{n+1,1}(U_{\eps})+C(\|P_0f_{\eps} \|^2_{L^2}+\|B_{\eps} \|^2_{L^2_x}).$$
Similarly, we can obtain the estimate for $n+1$ as follows.
 \bma
&\quad (1+t)^{l-1}E_{n+1,1}(U_{\eps})(t)+ \intt (1+s)^{l-1}D_{n+1,1}(U_{\eps})(s)ds
\nnm\\
&\le CE_{n+1,1}(U_0)+C(l-1)\intt (1+s)^{l-2}(\|P_0f_{\eps}(s)\|^2_{L^2}+\|B_{\eps}(s)\|^2_{L^2_x})ds\nnm\\
&\quad+C(l-1)\intt (1+s)^{l-2}D_{n+2,1}(U_{\eps}) ds .
\ema
And it follows from \eqref{G_1b} that
 \be
 E_{n+2,1}(U_{\eps})(t)+ \intt D_{n+2,1}(U_{\eps})(s)ds \le  CE_{n+2,1}(U_{0}). \label{l6}
\ee

By \eqref{l5}--\eqref{l6}, we obtain
 \bmas
&\quad (1+t)^lE_{n,1}(U_{\eps})(t)+ \intt (1+s)^lD_{n,1}(U_{\eps})(s)ds
\nnm\\
&\le CE_{n+2,1}(U_0)+C\intt (1+s)^{l-1}(\|P_0f_{\eps}(s)\|^2_{L^2}+\|B_{\eps}(s)\|^2_{L^2_x})ds
\emas
for $1<l<2$ and $n\ge 2$. Taking $l=3/2+\epsilon$ for a fixed constant $0<\epsilon<1/2$  yields
\bmas
&(1+t)^{\frac32+\epsilon}E_{n,1}(U_{\eps})(t)+ \intt (1+s)^{\frac32+\epsilon}D_{n,1}(U_{\eps})(s)ds
\nnm\\
\le& CE_{n+2,1}(U_0)+C(\delta_0+Q_{\eps}^2(t))^2 \intt (1+s)^{\frac12+\epsilon}(1+s)^{-\frac32}ds\nnm\\
\le& CE_{n+2,1}(U_0)+C(1+t)^{\epsilon}(\delta_0+Q_{\eps}^2(t))^2.
\emas  This gives
\be
E_{n,1}(U_{\eps})(t)\le C(1+t)^{-\frac32}(\delta_0+Q_{\eps}^2(t))^2.\label{J_6z}
\ee

Let  $0\le k\le N-3$. It follows from Lemma \ref{time6a} and \eqref{h1}  that
 \bma
\| P_1f_{\eps} (t)\|_{H^k }&\le
C\(\eps(1+t)^{-\frac54 }+e^{-\frac{bt}{\eps^2}}\)\(\| f_0\|_{H^{k+1}}+\|f_0\|_{L^1 }\)\nnm\\
&\quad+C \intt \(\eps(1+t-s)^{-\frac54 }+ e^{-\frac{b(t-s)}{\eps^2}}\)\nnm\\
&\qquad\times\(\| G_{1}(s)\|_{H^{k+1} }+\|G_{1}(s)\|_{L^1 }\)ds\nnm\\
&\quad+C \intt \(\eps(1+t-s)^{-\frac74 }+\frac1{\eps}e^{-\frac{b(t-s)}{\eps^2}}\)\nnm\\
&\qquad\times\(\| G_{2}(s)\|_{H^{k+2} }+\| G_{2}(s)\|_{L^1 }\)ds\nnm\\
&\le
C(\delta_0 +Q_{\eps}(t)^2)\(\eps(1+t)^{-\frac54}+e^{-\frac{bt}{4\eps^2}}\). \label{micro_1}
 \ema
 And by  Lemma \ref{time} and \eqref{h3}, we have
 \bma
\| P_rg_{\eps} (t)\|_{H^k }&\le
C\(\eps(1+t)^{-\frac34}+e^{-\frac{bt}{\eps^2}}\)\(\| V_0\|_{H^{k+1} }+\|V_0\|_{L^{1} }+\|\Tdx^2 V_0\|_{H^{k} }\)\nnm\\
&\quad+C\sum^4_{k=3} \int^{t}_0 \( \eps(1+t-s)^{-\frac34}+\frac1{\eps}e^{-\frac{b(t-s)}{\eps^2}}\)\nnm\\
&\qquad\times\(\| G_k(s)\|_{H^{k+2} }+\|G_k(s)\|_{L^{1} }+ \|\Tdx^2 G_k(s)\|_{H^{k}}\)ds\nnm\\
&\le
C(\delta_0 +Q_{\eps}(t)^2)\(\eps(1+t)^{-\frac34}+e^{-\frac{bt}{4\eps^2}}\). \label{micro_1a}
 \ema
 Combining \eqref{l6} and \eqref{micro_1a} gives
$$
Q_{\eps}(t)\le C\delta_0+CQ_{\eps}(t)^2 ,
$$
which shows \eqref{assume} provided $\delta_0>0$  being sufficiently small. Finally,  \eqref{G_5} follows  from \eqref{micro_1}.
The proof of the lemma is then completed.
\end{proof}

By \eqref{fe1}, \eqref{ue1} and Lemma \ref{timev}, we will prove the following lemma.

\begin{lem}\label{energy3} Let $N\ge 2$. There exists a small constant $\delta_0>0$ such that if $\|U_{0}\|_{H^N}+\|U_0\|_{L^1}\le \delta_0$, then
the NSMF system \eqref{NSM_2}--\eqref{NSP_5i} admits a unique global solution $\tilde{U}(t,x)=(n,m,q,\rho,E,B)(t,x)$ satisfying
\be
  \|\tilde{U} (t)\|_{H^N } \le C\delta_0 (1+t)^{-\frac34}. \label{time8}
\ee
\end{lem}
\begin{proof}
Define
$$
Q(t)=\sup_{0\le s\le t}\left\{(1+s)^{3/4}\(\|u_1(s)\|_{H^N}+\|V_1(s)\|_{H^N }\)\right\},
$$ where $u_1$ and $V_1$ are defined by \eqref{fe1} and \eqref{ue1} respectively.

Then, it follows from Lemmas \ref{timey} and \ref{timev} that
 \bma
 \|  u_1 (t)\|_{H^N}&\le  C(1+t)^{-\frac34}(\| P_0f_0\|_{H^N}+\|P_0f_0\|_{L^1})
 \nnm\\
&\quad+C\sum^2_{k=1}\intt \((1+t-s)^{-\frac34}+(t-s)^{-\frac12}e^{-c(t-s)}\)\nnm\\
&\qquad\qquad\times(\| H_k(s)\|_{H^N} +\| H_k(s)\|_{ L^1 })ds
  \nnm\\
&\le  C\delta_0(1+t)^{-\frac34}+CQ(t)^2 (1+t)^{-\frac34},\label{density_1}
\\
 \| V_1 (t)\|_{H^N}&\le  C(1+t)^{-\frac34}(\| P_2V_0\|_{H^N}+\|P_2V_0\|_{L^1})
 \nnm\\
&\quad+C \intt \((1+t-s)^{-\frac34}+(t-s)^{-\frac12}e^{-c(t-s)}\)\nnm\\
&\qquad\qquad\times(\| H_4(s)\|_{H^N} +\| H_4(s)\|_{ L^1 })ds
  \nnm\\
&\le  C\delta_0(1+t)^{-\frac34}+CQ(t)^2 (1+t)^{-\frac34}.\label{density_2}
 \ema
By \eqref{density_1} and \eqref{density_2}, we can obtain
$$
Q(t)\le C\delta_0 ,
$$
 provided $\delta_0>0$  being sufficiently small. This proves \eqref{time8}.
The existence of the solution can be proved by the contraction mapping theorem with
the details omitted.
 Then the  proof of the lemma is completed.
\end{proof}

Note that Theorem \ref{thm1.1}  follows directly from Lemmas \ref{time7} and \ref{energy3}.

\subsection{Optimal convergence rate}

In this section, we will complete the proof of Theorem \ref{thm1.2} about the optimal
convergence rate of the diffusion limit.

\begin{lem}[\cite{Li1}]\label{gamma1}For any $i,j=1,2,3,$ it holds that
\bma
\Gamma_*(v_i\chi_0,v_j\chi_0) &=-\frac12LP_1(v_iv_j\chi_0), \\
\Gamma_*(v_i\chi_0,|v|^2\chi_0) &=-\frac12LP_1(v_i|v|^2\chi_0), \\
\Gamma_*(|v|^2\chi_0,|v|^2\chi_0)&=-\frac12LP_1(|v|^4\chi_0),
\ema
where
$$\Gamma_*(f,g) =\frac12[\Gamma(f,g)+\Gamma(g,f)].$$
\end{lem}

\begin{lem}\label{gamma2}
For any $j=1,2,3,$ it holds that
\bma
\Gamma(\chi_0,v_j\chi_0) &=-L_1(v_j\chi_0), \\
\Gamma(\chi_0,|v|^2\chi_0) &=-L_1(|v|^2\chi_0).
\ema
\end{lem}

\begin{proof}
Let $u=v_*$, $u'=v'_*$.
By \eqref{conser}, it holds that
$$ v_j+u_j=v'_j+u'_j, \ \ j=1,2,3,\quad |v|^2+|u|^2=|v'|^2+|u'|^2. $$
Thus
\bmas
\Gamma(\chi_0,v_j\chi_0)=& \sqrt{M}\intr\ints  |(v-v_*)\cdot\omega|(u'_j-u_j)M(u) d\omega du\\
=&- \sqrt{M}\intr\ints  |(v-v_*)\cdot\omega|(v'_j-v_j)M(u)d\omega du\\
=&- L_1 (v_j\chi_0),
\\
\Gamma(\chi_0, |v|^2\chi_0)=& \sqrt{M}\intr\ints  |(v-v_*)\cdot\omega|(|u'|^2-|u|^2)M(u) d\omega du\\
=&- \sqrt{M}\intr\ints  |(v-v_*)\cdot\omega|(|v'|^2-|v|^2)M(u)d\omega du\\
=&- L_1 (|v|^2\chi_0),
\emas
which proves the lemma.
\end{proof}

\begin{proof}[\underline{\bf Proof of Theorem \ref{thm1.2}}]
First, for  \eqref{limit0}, set
\bma
\Lambda_{\eps}(t)=&\sup_{0\le s\le t}\bigg(\eps|\ln \eps|^2 (1+s)^{-\frac12}+ \bigg(1+\frac{s}{\eps}\bigg)^{-1}\bigg)^{-1} \nnm\\
& \quad\times\Big(\|P_{||} f_{\eps}(s)  \|_{ W^{2,\infty}} +\|P_{\bot}(f_{\eps}-u_1)(s)\|_{ H^{2} }+\| V_{\eps}(s)-V_1(s) \|_{ H^{2} }\Big),
\ema
where $P_{||}$ and $P_{\bot}$ are defined by \eqref{div1}. Note that
\bmas
\|P_{||} f_{\eps}  \|&\sim \left|(m_{\eps})_{||}\right|+\Big|n_{\eps}+\sqrt{\frac23}q_{\eps}\Big|,\\
\|P_{\bot}(f_{\eps}-u_1)\|&\sim \left|(m_{\eps})_{\bot}-m\right|+\Big|n_{\eps}-\sqrt{\frac32}q_{\eps}-n+\sqrt{\frac32}q\Big|+\|P_1f_{\eps}\|.
\emas

We claim that
\be
\Lambda_{\eps}(t) \le C\delta_0 ,\quad \forall\, t>0. \label{limit6}
\ee
It is straightforward to check  that the estimate  \eqref{limit0} follows  from \eqref{limit6}.

 By \eqref{fe} and \eqref{ue}, we have
\bma
 V_{\eps}(t)-V_1(t)
&= \left(e^{\frac{t}{\eps^2}\AA_{\eps}}U_0-Y_2(t)P_2U_0\right)+\intt e^{\frac{t-s}{\eps^2}\AA_{\eps}}G_{3}(s) ds\nnm\\
&\quad+\intt \bigg(\frac1{\eps}e^{\frac{t-s}{\eps^2}\AA_{\eps}}G_{4}(s)-Y_2(t-s) H_3(s)\bigg)ds\nnm\\
&=:I_1+I_2+I_3, \label{c1a}
\\
f_{\eps}(t)-u_1(t)
&= \(e^{\frac{t}{\eps^2}\BB_{\eps}}f_0-Y_1(t)P_0f_0\) +\intt\(e^{\frac{t-s}{\eps^2}\BB_{\eps}}G_{1}(s)-Y_1(t-s)H_1(s)\)ds\nnm\\
&\quad+ \intt \(\frac1{\eps}e^{\frac{t-s}{\eps^2}\BB_{\eps}}G_{2}(s)-Y_1(t-s)\divx H_2(s)\)ds \nnm\\
&=:I_4+I_5+I_6. \label{c1b}
\ema
By \eqref{limit1}, \eqref{limit4} and \eqref{limit4aa}, we can estimate $I_1$ and $I_4$
as follows.
\bma
\|I_1\|_{H^2} &\le C\delta_0\(\eps (1+t)^{-\frac34}+ e^{-\frac{bt}{\eps^2}}\),\label{I1}
\\
\|P_{||}I_4\|_{W^{2,\infty} }&\le C\delta_0\bigg( \eps(1+t)^{-\frac52}+\(1+\frac{t}{\eps}\)^{-1}\bigg) ,
\\
\|P_{\bot}I_4\|_{ H^{2} }&\le C\delta_0\(\eps (1+t)^{-\frac54}+ e^{-\frac{bt}{\eps^2}}\).
\ema

By \eqref{limit1}, \eqref{h3} and noting that $ P_2G_{3} =0$, we have
\bma
\|I_{2}\|_{H^2}&\le C \intt  \(\eps (1+t-s)^{-\frac34}+ e^{-\frac{b(t-s)}{\eps^2}}\)\(\| G_{3}\|_{H^3}+\|G_{3}\|_{L^1}\)ds\nnm\\
 &\quad +C \intt \eps^2 (1+t-s)^{-2}\|\Tdx^2 G_{3}\|_{H^2} ds\nnm\\
&\le C \delta_0^2\intt  \(\eps (1+t-s)^{-\frac34}+ e^{-\frac{b(t-s)}{\eps^2}}\)(1+s)^{-\frac32}ds \nnm\\
&\le C \delta_0^2 \eps (1+t)^{-\frac34}. \label{I2}
\ema
To estimate $I_3$, we decompose
\bma
 I_3 &= \intt  \bigg(\frac1{\eps}e^{\frac{t-s}{\eps^2}\AA_{\eps}}G_{4}(s)-Y_2(t-s)Z_0(s)\bigg)ds\nnm\\
&\quad+\intt\Big(Y_2(t-s)Z_0(s)-Y_2(t-s)H_3(s)\Big)ds\nnm\\
&=:I_{31}+I_{32}, \label{I3}
\ema where $Z_0=( P_d(v\cdot\Tdx L^{-1}_1G_{41}), ( L^{-1}_1G_{41},v\chi_0),0).$
Thus, it follows from  \eqref{limit2} and \eqref{h3} that
\bma
\| I_{31}\|_{H^2} &\le C\intt \( \eps(1+t-s)^{-\frac34}+ \frac1{\eps}e^{-\frac{b(t-s)}{\eps^2}}\)\(\|G_{4}\|_{H^4}+\|G_{4}\|_{L^1}\)ds \nnm\\
 &\quad+C \intt \eps^2 (1+t-s)^{-2} \|\Tdx^2 G_{4}\|_{H^2} ds\nnm\\
&\le C \delta_0^2\intt  \(\eps (1+t-s)^{-\frac34}+ \frac1{\eps}e^{-\frac{b(t-s)}{\eps^2}}\)(1+s)^{-\frac32}ds \nnm\\
&\le C \delta_0^2 \eps (1+t)^{-\frac34}. \label{I31}
\ema

By Lemma \ref{gamma2}, we obtain
\bma
 ( L^{-1}_1G_{41}, v\chi_0)&= ( -L^{-1}_1[(v\times B_{\eps})\cdot\Tdv P_0f_{\eps}]+ L^{-1}_1\Gamma(P_dg_{\eps},P_0f_{\eps})+L^{-1}_1\tilde{R},v\chi_0) \nnm\\
&=( -L^{-1}_1[(v\times B_{\eps})\cdot m_{\eps}\chi_0]- \rho_{\eps}(m_{\eps}\cdot v)\chi_0, v\chi_0)+(L^{-1}_1\tilde{R},v\chi_0) \nnm\\
&=\eta (m_{\eps}\times B_{\eps})- \rho_{\eps}m_{\eps} +(L^{-1}_1\tilde{R},v\chi_0), \label{p_d}
\ema
where 
\be
\tilde{R}= -(v\times B_{\eps})\cdot\Tdv P_1f_{\eps}+ \Gamma(P_rg_{\eps},f_{\eps})+ \Gamma(P_dg_{\eps},P_1f_{\eps}).
\ee
Thus, by noting that $m_{\eps}-m=(m_{\eps})_{||}+(m_{\eps})_{\bot}-m$ and by using \eqref{p_d}, \eqref{G_11}--\eqref{G_5a} and \eqref{time5b}, we have
\bma
\| I_{32}\|_{H^2} \le &C\intt \[(1+t-s)^{-\frac12}+(t-s)^{-\frac12}e^{-\frac{\eta}{2}(t-s)}\]\nnm\\
&\qquad  \times  \| ( L^{-1}_1G_{41}, v\chi_0)+ (\rho m-\eta m\times B)\|_{H^{2} }ds \nnm\\
\le &C\intt (t-s)^{-\frac12} \Big[\|m\|_{H^{2}_x}(\|\rho_{\eps}-\rho\|_{H^{2}_x}+\|B_{\eps}-B\|_{H^{2}_x})\nnm\\
&\qquad +(\|(m_{\eps})_{||} \|_{W^{2,\infty}_x }+\|(m_{\eps})_{\bot}-m\|_{H^{2}_x })(\|\rho_{\eps}\|_{H^{2}_x}+\|B_{\eps}\|_{H^{2}_x})\nnm\\
&\qquad +(\|\rho_{\eps}\|_{H^{2}_x}+\|B_{\eps}\|_{H^{2}_x})\|P_1f_{\eps}\|_{H^{3}}+\|P_rg_{\eps}\|_{H^{2}}\|f_{\eps}\|_{H^{2}} \Big]ds\nnm\\
\le &C \delta_0\Lambda_{\eps}(t)\intt  (t-s)^{-\frac12}
 \bigg(\eps|\ln \eps|^2  (1+s)^{-\frac12}+ \bigg(1+\frac{s}{\eps}\bigg)^{-1}\bigg)(1+s)^{-\frac34} ds\nnm\\
 &+C \delta_0^2\intt  (t-s)^{-\frac12}  \(\eps(1+s)^{-\frac34}+e^{-\frac{bs}{4\eps^2}}\)(1+s)^{-\frac34} ds. \label{I32a}
\ema
Denote
\bma J_0&=\intt  (t-s)^{-\frac12}  \bigg(1+\frac{s}{\eps}\bigg)^{-1} (1+s)^{-\frac34}  ds ,\\
 J_1&=\intt  (t-s)^{-\frac12} e^{-\frac{bs}{4\eps^2}}(1+s)^{-\frac34}  ds. \label{I23}
 \ema
For $J_0$, it holds that
\bma
J_0 &=\(\int^{t/2}_0+\int^{t}_{t/2}\)  (t-s)^{-\frac12} \bigg(1+\frac{s}{\eps}\bigg)^{-1}(1+s)^{-\frac34}  ds\nnm\\
&\le C t^{-\frac12} \int^{t/2}_0\bigg(1+\frac{s}{\eps}\bigg)^{-1}(1+s)^{-\frac34}ds\nnm\\
&\quad +C\bigg(1+\frac{t}{\eps}\bigg)^{-1}(1+t)^{-\frac34}\int^{t}_{t/2}  (t-s)^{-\frac12}ds
\nnm\\
&\le C\eps|\ln\eps|^2 (1+t)^{-\frac{1}2} +C\bigg(1+\frac{t}{\eps}\bigg)^{-1},  \label{I24a}
\ema
where we have used
\bmas
&\quad t^{-\frac12}\int^{t/2}_0\bigg(1+\frac{s}{\eps}\bigg)^{-1}(1+s)^{-\frac34}ds \\
&\le \left\{\bal t^{-\frac12}\int^t_0 ds \le \sqrt{t}\le C\sqrt{\eps}(1+ \frac{t}{\eps})^{-1}, & t\le \eps, \\
t^{-\frac12}\int^1_0 (1+\frac{s}{\eps})^{-1}ds \le Ct^{-\frac12} \eps|\ln\eps| \le C\eps|\ln\eps|^2  + C(1+ \frac{t}{\eps})^{-1}, & \eps< t\le 1, \\
t^{-\frac12}\(\int^1_0 (1+\frac{s}{\eps})^{-1}ds+\eps \int^{t}_{1} s^{-\frac74} ds\) \le C\eps |\ln \eps|(1+t)^{-\frac12}, & t> 1.
\ea\right.
\emas
For $J_1$, it holds that
\bma
J_1 &=\(\int^{t/2}_0+\int^{t}_{t/2}\)  (t-s)^{-\frac12}e^{-\frac{bs}{4\eps^2}}(1+s)^{-\frac34}  ds\nnm\\
&\le C t^{-\frac12} \int^{t/2}_0e^{-\frac{bs}{4\eps^2}}ds +Ce^{-\frac{bt}{4\eps^2}}(1+t)^{-\frac34}\int^{t}_{t/2} (t-s)^{-\frac12}ds
\nnm\\
&\le C\eps^2 t^{-\frac12}(1-e^{-\frac{bt}{4\eps^2}}) +Ct^{\frac12}e^{-\frac{bt}{4\eps^2}}(1+t)^{-\frac34}
\nnm\\
&\le C\eps (1+t)^{-\frac{1}2} ,  \label{I24a}
\ema
where we have used
$$
t^{-\frac12}(1-e^{-\frac{bt}{4\eps^2}})\le
\left\{\bal C\eps^{-1}, & t\le 1, \\ C(1+t)^{-\frac12}, & t>1.
\ea\right.
$$
Thus, it follows from \eqref{I32a}--\eqref{I24a} that
\be
\| I_{32}\|_{H^2}\le C ( \delta^2_0+\delta_0 \Lambda_{\eps}(t)) \bigg(\eps|\ln\eps|^2 (1+t)^{-\frac12}+ \bigg(1+\frac{t}{\eps}\bigg)^{-1}\bigg).\label{I32}
\ee
By combining \eqref{I1}, \eqref{I2}, \eqref{I3}, \eqref{I31} and \eqref{I32}, we obtain
\be
\| V_{\eps}(t)-V_1(t)\|_{H^2}\le C (\delta_0+\delta^2_0+\delta_0 \Lambda_{\eps}(t)) \bigg(\eps|\ln\eps|^2 (1+t)^{-\frac12}+ \bigg(1+\frac{t}{\eps}\bigg)^{-1}\bigg).\label{limit7}
\ee

To estimate $I_5$, we decompose
\bma
I_5&= \intt \left( e^{\frac{t-s}{\eps^2}\BB_{\eps}}G_{1}(s) -Y_1(t-s)P_0G_{1}(s) \right)ds\nnm\\
&\quad+\intt\Big(Y_1(t-s)P_0G_{1}(s)-Y_1(t-s)H_1(s)\Big)ds\nnm\\
&=:I_{51}+I_{52}.\label{I5}
\ema
By Lemma \ref{limit5} and \eqref{h1}, we obtain
\bma
\|P_{||}I_{51}\|_{W^{2,\infty}}&\le C\intt  \bigg(\eps(1+t-s)^{-\frac32}+\(1+\frac{t-s}{\eps}\)^{-1}\bigg) \(\|G_{1}\|_{H^5}+\|G_{1}\|_{W^{5,1}}\)ds \nnm\\
&\le C\delta_0^2\intt \bigg(\eps(1+t-s)^{-\frac32}+\(1+\frac{t-s}{\eps}\)^{-1}\bigg)
 \((1+s)^{-\frac32}+\frac{1}{\eps} e^{-\frac{bs}{4\eps^2}}\) ds \nnm\\
&\le C \delta_0^2 \eps|\ln \eps| (1+t)^{-\frac34} , \label{I51}
\\
\|P_{\bot}I_{51}\|_{H^2} &\le C\intt \( \eps(1+t-s)^{-\frac34}+ e^{-\frac{b(t-s)}{\eps^2}}\)\(\|G_{1}\|_{H^3}+\|G_{1}\|_{L^1}\)ds \nnm\\
&\le C\delta_0^2\intt \( \eps(1+t-s)^{-\frac34}+ e^{-\frac{b(t-s)}{\eps^2}}\)\((1+s)^{-\frac32}+\frac{1}{\eps}e^{-\frac{bs}{4\eps^2}}\) ds \nnm\\
&\le C \delta_0^2 \eps (1+t)^{-\frac34} , \label{I51a}
\ema
where we have used
\bmas
\intt\(1+\frac{t-s}{\eps}\)^{-1}(1+s)^{-\frac32}ds&\le \intt\(1+\frac{t-s}{\eps}\)^{-1}ds\\
&=\eps\ln\(1+\frac{t}{\eps}\)\le \eps|\ln \eps|, \quad t\le 1;
\\
\intt\(1+\frac{t-s}{\eps}\)^{-1}(1+s)^{-\frac32}ds&= \(\int^{t/2}_0+\int^{t}_{t/2}\)\(1+\frac{t-s}{\eps}\)^{-1}(1+s)^{-\frac32}ds\\
&\le C\(1+\frac{t}{\eps}\)^{-1}+C(1+t)^{-\frac32}\eps\ln\(1+\frac{t}{\eps}\)\\
&\le C\eps|\ln \eps|(1+t)^{-1}, \quad t>1.
\emas

Since
$$P_0G_{1} =\( \rho_{\eps} E_{\eps}\)\cdot v\chi_0 +\sqrt{\frac23}(u_{\eps}\cdot E_{\eps})\chi_4+\frac1{\eps}(u_{\eps}\times B_{\eps})\cdot v\chi_0,$$
it follows from \eqref{I5} that
\bma
I_{52}&= \intt Y_1(t-s)\bigg((\rho_{\eps}E_{\eps}-\rho E)\cdot v\chi_0+\sqrt{\frac23} (u_{\eps}\cdot E_{\eps}) \chi_4\bigg) ds \nnm\\
&\quad + \intt Y_1(t-s)\(\frac1{\eps}(u_{\eps}\times B_{\eps})-j\times B \)\cdot v\chi_0ds \nnm\\
&=:J_{51}+J_{52}. \label{I52}
\ema
For $J_{51},$ it follows from Lemmas \ref{timev} and \ref{time7} that
\bma
 \|J_{51}\|_{H^{2}} \le &C\intt (1+t-s)^{-\frac34}  \Big(\|\rho_{\eps}E_{\eps}-\rho E\|_{L^{1}_x\cap H^{2}_x}+\|u_{\eps}\cdot E_{\eps}\|_{L^{1}_x\cap H^{2}_x}\Big)ds\nnm\\
\le &C \delta_0\Lambda_{\eps}(t)\intt  (1+t-s)^{-\frac34}
\bigg(\eps|\ln \eps|^2  (1+s)^{-\frac12}+ \bigg(1+\frac{s}{\eps}\bigg)^{-1}\bigg)(1+s)^{-\frac34} ds\nnm\\
 &+C \delta_0^2 \intt  (1+t-s)^{-\frac34}\(\eps (1+s)^{-\frac34}+ e^{-\frac{bs}{\eps^2}}\)(1+s)^{-\frac34} ds\nnm\\
\le &C ( \delta^2_0+\delta_0 \Lambda_{\eps}(t)) \eps|\ln \eps|^2(1+t)^{-\frac34}  . \label{I22}
 \ema
By \eqref{p_c} and \eqref{p_d}, we have
\bmas
\frac1{\eps} u_{\eps}= &(L_1^{-1}[\eps\dt( P_rg_{\eps})+ P_r(v\cdot\Tdx  P_rg_{\eps})-\eps G_{31}-  G_{41}],v\chi_0)+ \eta(E_{\eps}-\Tdx \rho_{\eps}) \\
= &\eps\dt(L_1^{-1}( P_rg_{\eps}),v\chi_0)+ (L_1^{-1}P_r(v\cdot\Tdx  P_rg_{\eps} -\eps G_{31} - \tilde{R}),v\chi_0)\\
&+\eta(E_{\eps}-\Tdx \rho_{\eps})+ (\rho_{\eps}m_{\eps}-\eta m_{\eps}\times B_{\eps}),
\emas
which leads to
\bmas
\frac1{\eps} (u_{\eps}\times B_{\eps})= &\eps\dt[(L_1^{-1}( P_rg_{\eps}),v\chi_0)\times B_{\eps}]+\eps (L_1^{-1}( P_rg_{\eps}),v\chi_0)\times (\Tdx\times E_{\eps})\\
&+ (L_1^{-1}P_r(v\cdot\Tdx  P_rg_{\eps} -\eps G_{31} - \tilde{R}),v\chi_0)\times B_{\eps}+j_{\eps}\times B_{\eps}.
\emas
Thus
\bma
J_{52}&=\intt  Y_1(t-s) [j_{\eps}\times B_{\eps}-j \times B ]\cdot v\chi_0 ds\nnm\\
&\quad+\eps\intt Y_1(t-s)\pt_s [(L_1^{-1}( P_rg_{\eps}),v\chi_0)\times B_{\eps}]\cdot v\chi_0ds\nnm\\
&\quad+\eps\intt Y_1(t-s) [(L_1^{-1}( P_rg_{\eps}),v\chi_0)\times (\Tdx\times E_{\eps})]\cdot v\chi_0ds\nnm\\
&\quad+\intt Y_1(t-s) [(L_1^{-1}P_r(v\cdot\Tdx  P_rg_{\eps}-\eps G_{31}- \tilde{R}),v\chi_0)\times B_{\eps}]\cdot v\chi_0 ds\nnm\\
&=:  J^1_{52}+ J^2_{52}+ J^3_{52}+ J^4_{52} . \label{J52}
\ema
We estimate $J^i_{52}$, $i=1,2,3,4$ as follows. By Lemma \ref{timev}, we obtain
\bma
\|J^1_{52}\|_{H^2}&\le C\intt \((1+t-s)^{-\frac34}+(t-s)^{-\frac12}e^{-c(t-s)}\) \|j_{\eps}\times B_{\eps}-j \times B \|_{L^{1}_x\cap H^{1}_x}ds \nnm\\
&\le C \delta_0\Lambda_{\eps}(t)\intt \((1+t-s)^{-\frac34}+(t-s)^{-\frac12}e^{-c(t-s)}\)\nnm\\
&\qquad\qquad\qquad \times \bigg(\eps|\ln \eps|^2  (1+s)^{-\frac12}+ \bigg(1+\frac{s}{\eps}\bigg)^{-1}\bigg)(1+s)^{-\frac34} ds \nnm\\
&\le C \delta_0\Lambda_{\eps}(t) \bigg(\eps|\ln \eps|^2  (1+t)^{-\frac34}+ \bigg(1+\frac{t}{\eps}\bigg)^{-1}\bigg),
\ema
\bma
\|J^3_{52}\|_{H^2}&\le C\eps \intt (1+t-s)^{-\frac34}\|(L_1^{-1}( P_rg_{\eps}),v\chi_0)\times (\Tdx\times E_{\eps})\|_{L^{1}_x\cap H^{2}_x}ds \nnm\\
&\le C  \eps \intt  (1+t-s)^{-\frac34}\|P_rg_{\eps}\|_{ H^{2}}\| \Tdx\times E_{\eps}\|_{ H^{2}_x}ds \nnm\\
&\le C \delta_0^2\eps(1+t)^{-\frac34},
\ema
and
\bma
\|J^4_{52}\|_{H^2}
&\le C \intt  (1+t-s)^{-\frac34}\Big[(\|\Tdx P_rg_{\eps}\|_{H^{2}}+\eps\|E_{\eps}\|_{ H^{2}_x}\|f_{\eps}\|_{ H^{2}})\| B_{\eps}\|_{ H^{2}_x} \nnm\\
&\qquad\quad + ( \|(\rho_{\eps},B_{\eps})\|_{ H^{2}_x}\|P_1f_{\eps}\|_{ H^{2}}+\|P_rg_{\eps}\|_{ H^{2}}\|f_{\eps}\|_{ H^{2}})\| B_{\eps}\|_{ H^{2}_x}\Big]ds \nnm\\
&\le C \delta_0^2\intt  (1+t-s)^{-\frac34} \(\eps(1+s)^{-\frac34}+e^{-\frac{bs}{\eps^2}}\)(1+s)^{-\frac34}ds \nnm\\
&\le C \delta_0^2\eps(1+t)^{-\frac32}.
\ema
Moreover, 
 it holds that
\bma
\|J^2_{52}\|_{H^2}&= \eps \bigg\| [(L_1^{-1}( P_rg_{\eps}),v\chi_0)\times B_{\eps}]\cdot v\chi_0-Y_1(t)[(L_1^{-1}( P_rg_{0}),v\chi_0)\times B_{0}]\cdot v\chi_0 \nnm\\
&\quad+\intt \dt Y_1(t-s) [(L_1^{-1}( P_rg_{\eps}),v\chi_0)\times B_{\eps}]\cdot v\chi_0ds\bigg\|_{H^2} \nnm\\
&\le C \delta_0^2\eps(1+t)^{-\frac34}+C\eps \intt  (1+t-s)^{-\frac74}\|P_rg_{\eps}\|_{H^{4}}\| B_{\eps}\|_{ H^{4}_x}ds \nnm\\
&\le C \delta_0^2\eps(1+t)^{-\frac32},\label{J52a}
\ema
where we have used
$$
\|\dt Y_1(t)f_0\|^2_{L^2}\le \intr \bigg|\sum_{j=0,2,3}a_j|\xi|^2e^{-a_j|\xi|^2t}\bigg|^2 \|\hat f_0\|^2d\xi
\le C (1+t)^{-\frac72}\|f_0\|^2_{L^1\cap H^2}.
$$

Thus, it follows from \eqref{I52}--\eqref{J52a} that
\be
\|I_{52}\|_{H^2}\le C ( \delta_0^2+\delta_0 \Lambda_{\eps}(t)) \bigg(\eps|\ln\eps|^2 (1+t)^{-\frac12}+ \bigg(1+\frac{t}{\eps}\bigg)^{-1}\bigg).\label{I25}
\ee

To estimate $I_3$, we decompose
\bma
I_6&= \intt  \(\frac1{\eps}e^{\frac{t-s}{\eps^2}\mathbb{B}_{\eps}}G_{2}-Y_1(t-s)P_0(v\cdot\Tdx L^{-1}G_{2})\)ds\nnm\\
&\quad+\intt \Big(Y_1(t-s)P_0(v\cdot\Tdx L^{-1}G_{2})-Y_1(t-s)\divx H_2\Big) ds\nnm\\
&=:I_{61}+I_{62}. \label{I6}
\ema
By \eqref{limit2}, \eqref{h1}, and noting that $P_0G_{2}=0$, we have
\bma
\|P_{||}I_{61}\|_{W^{2,\infty}}\le &C \intt  \bigg( \eps(1+t-s)^{-\frac32}+\bigg(1+\frac{t-s}{\eps}\bigg)^{-1}+\frac1{\eps}e^{-\frac{b(t-s)}{\eps^2}}\bigg)\nnm\\
 &\quad \times \(\|G_{2}\|_{H^6}+\|G_{2}\|_{W^{6,1}}\)ds\nnm\\
\le &C \delta_0^2\intt \bigg(\eps(1+t-s)^{-\frac32}+ \bigg(1+\frac{t-s}{\eps}\bigg)^{-1}+\frac1{\eps}e^{-\frac{b(t-s)}{\eps^2}}\bigg)(1+s)^{-\frac32}ds\nnm \\
\le &C \delta_0^2 \eps|\ln \eps|(1+t)^{-1} ,\label{I61}
\ema
and
\bma
 \|P_{\bot}I_{61}\|_{H^{2}}&\le C\intt \( \eps(1+t-s)^{-\frac34}+ \frac1{\eps}e^{-\frac{b(t-s)}{\eps^2}}\)\(\|G_{2}\|_{H^4}+\|G_{2}\|_{L^1}\)ds\nnm\\
&\le C\delta_0^2\intt \( \eps(1+t-s)^{-\frac34}+ \frac1{\eps}e^{-\frac{b(t-s)}{\eps^2}}\) (1+s)^{-\frac32} ds\nnm\\
&\le C \delta_0^2 \eps (1+t)^{-\frac34}. \label{I61a}
\ema

To estimate $I_{32}$, we decompose
\bmas
 P_0(v\cdot\Tdx L^{-1}G_{2} )
&=P_0(v\cdot\Tdx L^{-1}\Gamma_*(P_0f_{\eps},P_0f_{\eps}))+2P_0(v\cdot\Tdx L^{-1}\Gamma_*(P_0f_{\eps},P_1f_{\eps}))\\
&\quad+P_0(v\cdot\Tdx L^{-1}\Gamma_*(P_1f_{\eps},P_1f_{\eps}))+P_0(v\cdot\Tdx L^{-1}(v\times B_{\eps})\cdot \Tdv P_rg_{\eps})\\
&=:J_1+J_2+J_3+J_4,
\emas where $2\Gamma_*(f,g)= \Gamma(f,g)+\Gamma(g,f) $.
By Lemma \ref{gamma1}, we can obtain (cf. \cite{Li1})
$$
J_1=-\sum^3_{i, j=1}\pt_i(m^i_{\eps} m^j_{\eps})v_j\chi_0+\frac13\sum^3_{i,j=1}\pt_j( m^i_{\eps})^2 v_j\chi_0 -\frac53\sum^3_{j=1}\pt_j(m^j_{\eps}q_{\eps}) \chi_4.
$$
This  and \eqref{v1} give
\be
Y_1(t)J_1=-Y_1(t)\divx\[(m_{\eps}\otimes m_{\eps})\cdot v\chi_0+\frac53 (q_{\eps}m_{\eps})\chi_4\]=:Y_1(t)\divx J_4. \label{eee1}
\ee
Note that
$$
|m_{\eps}-m|+|q_{\eps}-q|\le C(\|P_{\bot}(f_{\eps}-u_1)\| +\|P_{||}f_{\eps}\| ).
$$
Thus, by  \eqref{time1b}, \eqref{G_11}, \eqref{G_5}  and \eqref{eee1}, we have
\bma
\|I_{62}\|_{H^{2}}\le &\intt  \left\|Y_1(t-s)\divx (J_4-H_2)\right\|_{ H^{2} }ds+\sum^4_{k=2}\intt  \left\|Y_1(t-s)J_k\right\|_{ H^{2}}ds\nnm\\
\le &C\intt  (t-s)^{-\frac12} \|J_4-H_2 \|_{ H^{2} }ds +\sum^4_{k=2}C\intt  (1+t-s)^{-\frac34 } \|J_k \|_{L^1\cap H^{2}}ds\nnm\\
\le  &C\intt  (t-s)^{-\frac12} (\|P_{\bot}(f_{\eps}-u_1)\|_{H^{2} }+\|P_{||}f_{\eps}\|_{W^{2,\infty} })\|(f_{\eps}, u_1)\|_{H^{2}} ds\nnm\\
  &+C \intt (1+t-s)^{-\frac34}\(\|P_1f_{\eps}\|_{ H^{3}} \| f_{\eps}\|_{ H^{3}}+\|B_{\eps}\|_{ H^{3}_x}\|P_rg_{\eps}\|_{ H^{3}}\)ds\nnm\\
\le &C \delta_0\Lambda_{\eps}(t)  \intt  (t-s)^{-\frac12}\bigg(\eps|\ln \eps|^2 (1+s)^{-\frac12}+ \bigg(1+\frac{s}{\eps}\bigg)^{-1}\bigg)(1+s)^{-\frac34} ds \nnm\\
 &+C \delta_0^2 \intt  (1+t-s)^{-\frac34}\(\eps (1+s)^{-\frac34}+ e^{-\frac{bs}{\eps^2}}\)(1+s)^{-\frac34} ds\nnm\\
 \le& C(\delta_0^2+ \delta_0\Lambda_{\eps}(t)) \bigg(\eps|\ln\eps|^2 (1+t)^{-\frac12}+ \bigg(1+\frac{t}{\eps}\bigg)^{-1}\bigg). \label{I62}
\ema

Therefore, it follows from \eqref{I5}--\eqref{I52}, \eqref{I61}, \eqref{I61a} and \eqref{I62} that
\bma
&\quad \|P_{||}f_{\eps} \|_{ W^{2,\infty} } +\|P_{\bot}(f_{\eps} -u_1)  \|_{ H^{2} }  \nnm\\
 &\le C\(\delta_0+\delta^2_0+\delta_0 \Lambda_{\eps}(t)\) \bigg(\eps|\ln\eps|^2 (1+t)^{-\frac12}+ \bigg(1+\frac{t}{\eps}\bigg)^{-1}\bigg).\label{limit8}
\ema

By combining \eqref{limit7} and \eqref{limit8}, we obtain
$$
\Lambda_{\eps}(t)\le  C\(\delta_0+\delta_0^2+\delta_0\Lambda_{\eps}(t)\).
$$
By taking $\delta_0>0$  small enough, we  obtain \eqref{limit6} which gives \eqref{limit0}.

Next, we prove \eqref{limit_1a} as follows.  Set
$$
\Omega_{\eps}(t)=\sup_{0\le s\le t} (\eps|\ln \eps|)^{-1} (1+s)^{\frac12} \(\|P_{||}f_{\eps}(s) \|_{ W^{2,\infty} } +\|P_{\bot}(f_{\eps} -u_1)(s)  \|_{ H^{2} }+\|(V_{\eps}-V_1)(s)\|_{ H^{2}}\).
$$
By \eqref{limit1a}, \eqref{limit3} and \eqref{limit4a},  $I_1$ and $I_4$ are estimated by
\be
\|I_1\|_{ H^{2}}+\|P_{||}I_4\|_{ W^{2,\infty}}+\|P_{\bot}I_4\|_{ H^{2}}\le C\delta_0 \eps (1+t)^{-\frac34}. \label{I1a}
\ee
By \eqref{time5b} and \eqref{p_d}, we have
\bma
\|I_{32}\|_{ H^{2}} \le &C\intt(t-s)^{-\frac12} \| ( L^{-1}_1G_{41}, v\chi_0)+ (\rho m-\eta m\times B)\|_{H^{2}_x }ds \nnm\\
\le  &C (\delta_0^2+\delta_0\Omega_{\eps}(t)) \intt(t-s)^{-\frac12} \(  \eps|\ln \eps| (1+s)^{-\frac12} +e^{-\frac{bs}{4\eps^2}} \)(1+s)^{-\frac34}ds \nnm\\
\le &C(\delta^2_0+\delta_0 \Omega_{\eps}(t)) \eps|\ln \eps|(1+t)^{-\frac12} .\label{I31a}
\ema
By combining \eqref{I1a}, \eqref{I2}, \eqref{I3}, \eqref{I31} and \eqref{I31a}, we obtain
\be
\| V_{\eps}(t)-V_1(t)\|_{H^2}\le C \(\delta_0+\delta^2_0+\delta_0 \Omega_{\eps}(t)\)\eps|\ln \eps| (1+t)^{-\frac12}.\label{limit7c}
\ee
Similarly,
\be
\|P_{||}f_{\eps} \|_{ W^{2,\infty} } +\|P_{\bot}(f_{\eps} -u_1)  \|_{ H^{2} }
 \le C\(\delta_0+\delta^2_0+\delta_0 \Omega_{\eps}(t)\) \eps|\ln \eps| (1+t)^{-\frac12} .\label{limit8c}
\ee
By \eqref{limit7c} and \eqref{limit8c}, we obtain
$$
\Omega_{\eps}(t)\le  C\delta_0+C\delta_0^2+C\delta_0\Omega_{\eps}(t).
$$
By taking $\delta_0>0$  small enough, we  obtain  \eqref{limit_1a}. And this completes
the  proof of the theorem.
\end{proof}

\medskip
\noindent {\bf Acknowledgements:}
The first author was supported by
a fellowship award from the Research Grants Council of the Hong Kong Special Administrative Region, China (Project no. SRFS2021-1S01). The second author was supported by  the National Natural Science Foundation of China  grants No. 12171104, and Guangxi Natural Science Foundation No. 2019JJG110010.


\end{document}